%% file: dffe.tex
\documentclass[11pt,a4paper,twoside,reqno]{amsart}      
\usepackage[T1]{fontenc} 
\usepackage[utf8]{inputenc}  
\usepackage{amsmath,amssymb,amsfonts,amsthm,amscd} 
\usepackage{bigints}
\usepackage{graphicx}                 
\usepackage{color}                    
\usepackage{ mathrsfs }
\usepackage{epstopdf}
\usepackage[margin=2.5cm]{geometry}
\usepackage{enumerate}
\usepackage[normalem]{ulem}
\usepackage{url}         
\usepackage{colonequals} 
\usepackage[foot]{amsaddr}

\numberwithin{equation}{section}

\newtheorem{theorem}{Theorem}[section]
\newtheorem{proposition}[theorem]{Proposition}
\newtheorem{corollary}[theorem]{Corollary}
\newtheorem{lemma}[theorem]{Lemma}
\theoremstyle{definition}
\newtheorem{remark}[theorem]{Remark}

\newcommand{\smfrac}[2]{{\textstyle \frac{#1}{#2}}}

\def\div{\mathop{\mathrm{div}}\nolimits}
\def\!{\mathop{\mathrm{!}}}

\def\R{\mathbb{ R}}
\def\CN{\mathbb{ C}}
\def\N{\mathbb{ N}}

\def\A{\mathbb{A}}
\def\E{\mathcal{E}}
\def\F{\mathcal{F}}

\def\B{\mathcal{B}}

\def\rcut{{r_{\rm cut}}}

\def\brho{\boldsymbol\rho}

\def\L{\Lambda}
\def\Rc{\mathcal{R}}

\def\Rr{\mathscr{R}}

\def\Z{\mathbb{Z}}
\def\W{\mathcal{W}}
\def\Wc{\dot{\mathcal{W}}^{\rm c}}
\def\Wi{\dot{\mathcal{W}}^{1,2}}

\def\bsep{\,\big|\,}
\def\b{\big}
\def\bg{\bigg}

\def\<{\langle}
\def\>{\rangle}

\def\mA{{\sf A}}
\def\mB{{\sf B}}

\def\ktst{\mathcal{K}^{\rm TST}}
\def\khtst{\mathcal{K}^{\rm HTST}}

\def\E{\mathcal{E}}
\def\Ehom{\mathcal{E}^{\rm hom}}
\def\Edef{\mathcal{E}}

\def\Shom{\SS^{\rm hom}}
\def\SS{\mathcal{S}}

\def\Rdef{\Rr}

\def\RMt{\Rr^{M,t}}

\def\Rhom{\Rr^{\rm hom}}
\def\Hdef{H}
\def\Hhom{H^{\rm hom}}
\def\Hs{H^{\rm s}}

\def\HsN{H^{\rm s}_N}

\def\HhomN{H^{\rm hom}_N}

\def\XXint#1#2#3{{\setbox0=\hbox{$#1{#2#3}{\int}$ }
\vcenter{\hbox{$#2#3$ }}\kern-.6\wd0}}

\def\HM{H^{M}}
\def\HMt{H^{M,t}}

\def\Lr{\mathscr{L}}

\def\Fop{{\bf F}}
\def\Fopd{{\bf F}^*}

\def\Wper{\mathcal{W}^{\rm per}}   
\def\Usper{\Wper}
\def\us{\bar{u}}                    
\def\usa{\bar{u}^{\rm s}}

\def\usaN{\bar{u}^{\rm s}_{N}}

\def\kreg{p}  

\def\detp{{\rm det}^{+}}
\def\logp{\log^+ \hspace{-0.3em}}

\newlength{\boxwidth}
\date \today

\title[Transition Rate of a Crystalline Defect]{Thermodynamic Limit of the Transition Rate \\
    of a Crystalline Defect}

\author{Julian Braun}
\author{Manh Hong Duong}
\author{Christoph Ortner}
\thanks{JB and CO are supported by ERC Starting Grant 335120 and by EPSRC Grant EP/R043612/1}
\thanks{MHD was supported by ERC Starting Grant 335120}

\address[JB, CO]{Mathematics Institute, University of Warwick, Coventry CV4 7AL, UK.}
\address[MHD]{School of Mathematics, University of Birmingham, Birmingham B15 2TT, UK.}


\begin{document}
\begin{abstract}
  We consider an isolated point defect embedded in a homogeneous crystalline
  solid. We show that, in the harmonic approximation, a periodic supercell
  approximation of the formation free energy as well as of the transition rate
  between two stable configurations converge as the cell size tends to infinity. We characterise the limits and establish sharp convergence rates. Both cases can be reduced to a careful renormalisation analysis of the vibrational entropy
  difference, which is achieved by identifying an underlying spatial
  decomposition.
\end{abstract}

\maketitle

\input{intro.tex}

\input{results.tex}

\input{resolvents.tex}

\input{locality.tex}

\input{limit.tex}

\input{saddle.tex}

\input{appendix.tex}

\bibliographystyle{alpha}
\bibliography{bib}

\end{document}

%% file: intro.tex

\section{Introduction}
\label{sec:intro}
The presence of defects in crystalline materials significantly affects their
mechanical and chemical properties, hence determining defect geometry, energies,
and mobility is a fundamental problem of materials modelling. The inherent
discrete nature of defects requires that any ``ab initio'' theory should start
from an atomistic description. The purpose of the present work is to extend the
model of crystalline defects of~\cite{EOS2016} (cf.\ \S\,\ref{sec:results}) to
incorporate vibrational entropy, in order to describe the thermodynamic limit of
transition rates (mobility) of point defects. As an intermediate step we will
also discuss the thermodynamic limit of defect formation free energy.

Apart from being interesting in their own right, our results provide the
analytical foundations for a rigorous derivation of coarse-grained models
\cite{TadmorLegoll2013,Voter2007-nw,Boateng2014-ol,Hudson2017-bx}, and of
numerical and multi-scale models at finite temperature
\cite{hyperqc,LuskinShapeev2014TMP,TadmorLegoll2013,BlancBrisLegollPatz2010,BlancLegoll2013}
which entirely lack the solid foundations that static zero-temperature
multi-scale schemes enjoy \cite{LuskinOrtner2013-Acta,2014-bqce,Lu2013-xj}.

Precise definitions will be given in Section \ref{sec:results} but, for the
purpose of a purely formal motivation, we consider a crystalline solid with an
embedded defect described by an energy landscape {$\E_N : (\R^m)^{\L_N} \to
\R$}, based on a set of reference atoms $\L_N \subset \R^d$.
We then consider a local minimizer $\bar{u}^{\rm min}_N$ of $\E_N$
representing a defect state.

In transition state theory (TST) \cite{Eyring,Wigner}, the transition rate $\mathcal{K}_N$ from $\bar{u}^{\rm min}_N$ to a nearby state $\bar{u}^{\rm min2}_N$ is given by comparing the equilibrium density in a basin {$A \subset  (\R^m)^{\L_N} $} around $\bar{u}^{\rm min}_N$ to the density on a hyper-surface {$S \subset (\R^m)^{\L_N} $} separating $A$ from a similar basin around $\bar{u}^{\rm min2}_N$. That is,
\begin{equation*}
\mathcal{K}_N^{\rm TST} =
    \frac{
      \int_S e^{-\beta \E_N(u)} \, du
    }{
      \int_A e^{-\beta \E_N(u)} \, du
    },
\end{equation*}
with inverse temperature $\beta$.
The {\em transition state} is an index-1 saddle point $\bar{u}^{\rm saddle}_N \in S$ of $\E_N$ representing the most likely transition path between the two minima. For sufficiently large $\beta$, $\int_S e^{-\beta \E_N(u)} \, du$ is concentrated close to $\bar{u}^{\rm saddle}_N$. Similarly, $\int_A e^{-\beta \E_N(u)} \, du$ is concentrated around the local minimum $\bar{u}^{\rm min}_N$. Therefore, it is reasonable to consider the \emph{harmonic approximations}
\begin{align*}
\E_N(u) &\approx \E_N(\bar{u}^{\rm saddle}_N) + {\textstyle \frac12} \big\< \nabla^2 \E_N(\bar{u}^{\rm saddle}_N) (u-\bar{u}^{\rm saddle}_N), u - \bar{u}^{\rm saddle}_N \big\>\\
\E_N(u) &\approx \E_N(\bar{u}^{\rm min}_N) + {\textstyle \frac12} \big\< \nabla^2 \E_N(\bar{u}^{\rm min}_N) (u-\bar{u}^{\rm min}_N), u - \bar{u}^{\rm min}_N \big\>,
\end{align*}
and to integrate over all states instead of $A$ and the tangent space of $S$ at $\bar{u}^{\rm saddle}_N$ instead of $S$. The argument is classical, see \cite{GHVineyard:1957}, and evaluating the Gaussian integrals leads to the well known harmonic TST (HTST) with transition rate
\begin{equation} \label{eq:HTST}
\mathcal{K}^{\rm HTST}_N :=
   \bigg(\frac{
      \textstyle \prod \lambda_j^{\rm min}
   }{
      \textstyle \prod \lambda_j^{\rm saddle}
   }\bigg)^{1/2}
   \, \exp\Big( - \beta \big[
      \E_N(\bar{u}^{\rm saddle}_N) -
      \E_N(\bar{u}^{\rm min}_N) \big]
      \Big),
\end{equation}
where the $\lambda_j^{*}$ enumerate the positive eigenvalues of $\nabla^2 \E_N(\bar{u}^*_N)$, with $*= {\rm min}$ or $*= {\rm saddle}$.

Formally, $\beta^{-1} \log \mathcal{K}^{\rm TST}_N = \beta^{-1} \log
\mathcal{K}^{\rm HTST}_N + O(\beta^{-2})$, and indeed in materials modelling
applications far from the melting temperature, the harmonic approximation is
considered an excellent model \cite{HTB90, Voter2007-nw}. Making this statement
rigorous is an interesting question in its own right, especially in the limit as
$N \to \infty$, but will not be the purpose of the present work. Related results
in this direction, though with a very different setup, can be found, for example,
in \cite{Berglund2007-hk, BBM2010}.

Instead, the goal of this paper is to show that the thermodynamic limit
$\mathcal{K}_N^{\rm HTST}\to \mathcal{K}^{\rm HTST}$ exists as $\L_N$ tends to
an infinite lattice $\L$ and to characterise the limit $\mathcal{K}^{\rm HTST}$.
The interest in this result is two-fold: (1) it establishes that the
finite-domain model is meaningful in that increasingly large domains yield
consistent answers; and (2) it provides a benchmark against which various
numerical schemes to compute transition rates can be measured.

Our starting point in establishing the thermodynamic limit of $\mathcal{K}_N^{\rm
HTST}$ is a model for the equilibration of an isolated defect embedded in
a homogeneous crystalline solid introduced in \cite{EOS2016, Hudson2014}. Briefly,
it is shown under suitable conditions on the boundary condition that, as $\L_N
\to \L$, $\bar{u}_N^{\rm *}$ has a limit $\bar{u}^{\rm *}$ and moreover the
decay of $\bar{u}^{\rm *}$ away from the defect core is precisely quantified.
These results directly give a convergence result for the energy difference
$\E_N(\bar{u}^{\rm saddle}_N) - \E_N(\bar{u}^{\rm min}_N)$
and also supply us with structures that can be exploited in the analysis of the
Hessians $\nabla^2 \E(\bar{u}^{\rm *})$.

Still, the convergence of $\mathcal{K}_N^{\rm HTST}$ is a difficult problem. In
the limit, one would expect to find both a continuous spectrum as well as
infinitely many eigenvalues for the Hessian, hence the representation of $\lim_N
\mathcal{K}_N^{\rm HTST}$ will unlikely be in terms of the spectra of the
associated operators.

Mathematically, it turns out to be expedient to rewrite \eqref{eq:HTST} in terms of a free energy difference or an entropy difference. That is, we write
\begin{align*}
\mathcal{K}^{\rm HTST}_N &= \exp\Big( - \beta \Big(\big[
      \E_N^{\rm def}(\bar{u}^{\rm saddle}_N) -
      \E_N^{\rm def}(\bar{u}^{\rm min}_N) \big]-\beta^{-1}\big[
      \SS_N(\bar{u}^{\rm saddle}_N) -
      \SS_N(\bar{u}^{\rm min}_N) \big] \Big)
      \Big)\\
      &= \exp\Big( - \beta \big[
      \F_N(\bar{u}^{\rm saddle}_N) -
      \F_N(\bar{u}^{\rm min}_N) \big] \Big),
\end{align*}
and then consider the limiting behaviour of the difference of vibrational entropies $\SS_N(\bar{u}^{\rm saddle}_N) - \SS_N(\bar{u}^{\rm min}_N)$. A key idea in the analysis of the entropy difference is then to discard the spectral decomposition of the Hessians and instead work with a {\em spatial decomposition} that we will derive in \S~\ref{sec:results}. We then prove locality estimates in this spatial decomposition that allow us to renormalise before taking  the limit $N \to \infty$.

In our analysis of the free energy difference, i.e., differences of $\F_N$, one can also compare the homogeneous lattice with a defect state, allowing us to additionally get a result on the thermodynamic limit for the formation free energy of a defect in the harmonic approximation.

We point out that, for technical reasons and to simplify the presentation of our
main ideas, our paper admits only defects where the number of atoms is equal to
that in the reference configuration, including for example substitutional
impurities, Frenkel pairs, and the Stone-Wales defect. However, we expect that
it is possible to adapt our methods and results to the cases of vacancies and
interstitials, while extensions to long-ranged defects such as dislocations and
cracks may be more challenging; cf. \S~\ref{sec:results:conclusion}.

While there is a  substantial literature on the scaling limit (free energy
per particle), see e.g.\ \cite{DemboFunaki2005} and references therein, we are
aware of only two references that attempt to rigorously capture atomistic
details of the limit $N \to \infty$ of crystalline defects in a finite
temperature setting \cite{LuskinShapeev2014TMP,2016-defectFE1d}. While
\cite{LuskinShapeev2014TMP} considers the somewhat different setting of
observables rather than formation energies there is a close connection in that
those observables are localised. Moreover, an asymptotic series in $\beta$ is
derived instead of focusing only on leading terms. By contrast
\cite{2016-defectFE1d} addresses the finite $\beta$ regime, but severely
restricts the admissible interaction laws.  Both of these references are
restricted to one dimension, which yields significant simplifications
highlighted for example by the fact that discrete Green's functions decay
exponentially. Thus, treating the $d$-dimensional setting with $d > 1$, relevant
for applications, requires different techniques.

\subsection*{Outline}

In \S\,\ref{sec:results}, we will precisely define all relevant quantities and present our main results, namely, the construction of limit quantities $\F$ and $\mathcal{K}^{\rm HTST}$ on an infinite lattice $\L$, as well as the convergence results $\F_N \to \F$ and $\mathcal{K}^{\rm HTST}_N \to \mathcal{K}^{\rm HTST}$  with explicit convergence rates.

In the subsequent sections we will prove these results. Based on operator estimates in \S\,\ref{sec:resolventestimates}, we construct $\F$ in \S\,\ref{sec: locality}. In \S\,\ref{sec:limit}, we then prove the convergence $\F_N \to \F$. Finally, in \S\,\ref{sec:saddle}, we will discuss saddle points in the energy landscape and use the results from \S\S\,\ref{sec:resolventestimates}--\ref{sec:limit} to construct $\mathcal{K}^{\rm HTST}$ and show $\mathcal{K}^{\rm HTST}_N \to \mathcal{K}^{\rm HTST}$. In the appendix in \S\,\ref{sec: appendix}, we collect several auxiliary results and proofs used throughout the previous sections.

\subsection*{General Notation}
If $X$ is a (semi-)Hilbert space with dual $X^*$ then we denote the duality pairing by
$\< \cdot, \cdot \>$. The space of bounded linear operators from $X$ to another (semi-)Hilbert space $Y$ is denoted by $\mathcal{L}(X, Y)$. If $\E \in C^2(X)$ then $\delta \E(x) \in X^*$ denotes
the first variation, while $\< \delta \E(x), v\>$ with $v \in X$ denotes the
directional derivative. Further $\delta^2 \E(x) \in \mathcal{L}(X, X^*)$ denotes
the second variation (informally we may also call it the Hessian).

If $V \in C^p(\R^m)$ then we will denote its derivatives by $\nabla^j V(x)$ and
interpret them as multi-linear forms, which supplied with arguments read
$\nabla^j V(x)[a_1, \dots, a_j]$ for $a_i \in \R^m$.

If $\Lambda$ is a countable index-set (usually a Bravais lattice $\L = \mA \Z^d$
with $\mA\in\R^{d\times d}$ non-singular) then $\ell^2(\Lambda;\R^m)=\{u : \L
\to \R^m : \sum_{\ell \in \L} \lvert u \rvert^2 < \infty\}$. When the range
is clear from the context then we often just write $\ell^2(\Lambda)$ or $\ell^2$.

Given $A \in \mathcal{L}(\ell^2(\Lambda;\R^m), \ell^2(\Lambda;\R^m))$ we define
the components $A_{\ell i n j} = \big(A (\delta_\ell e_i), \delta_n e_j
\big)_{\ell^2(\Lambda;\R^d)}$ for $\ell,n \in \Lambda$ and $i,j \in
\{1,...,m\}$. We will also use the notation $A_{\ell n} = (A_{\ell i n j})_{ij} \in
\R^{m \times m}$ for the matrix blocks corresponding to atom sites. The identity is denoted by $(I_{\ell^2(\Lambda;\R^m)})_{\ell i n j} := \delta_{\ell n} \delta_{ij}$, sometimes shortened to $I_{\ell^2(\Lambda)}$ or just $I$, if the context is clear.

%% file: results.tex

\section{Results}
\label{sec:results}
We consider a point defect embedded in a homogeneous lattice, following the
models in \cite{EOS2016}. To simplify the presentation, we consider a Bravais
lattice, a finite interaction radius, and a smooth interatomic potential.
Moreover, we only formulate the model for substitutional impurities, short-range
Frenkel defects, and other point defects that do not change the number of atoms.

On a Bravais lattice $\L= \mA \Z^d \subset \R^d$, lattice displacements are functions $u : \L \to \R^m$,
for some $m \in \N$, typically $m = d$. Let $\rcut > 0$ be an interaction cut-off radius, then $\Rc := (\L \setminus \{0\}) \cap B_\rcut$ is the interaction range and
\begin{equation*}
   Du(\ell) := \b( D_\rho u(\ell) \b)_{\rho\in\Rc}
             := \b( u(\ell+\rho) - u(\ell) \b)_{\rho\in\Rc}
\end{equation*}
a finite difference gradient. We assume $\rcut$ is large enough such that ${\rm
span}_{\Z}( \Rc) = \L$. For each $\ell \in \L$ let $V_\ell \in C^\kreg((\R^{m})^\Rc)$, $\kreg \geq 4$ be a site energy potential so that the total
energy contribution from site $\ell$ is given by $V_\ell(Du(\ell))$.

 We assume that the interaction is homogeneous away from the defect, i.e., $V_\ell \equiv V$ for all $\lvert \ell \rvert >  \rcut$, and that $V$ satisfies the natural point symmetry $V(A) = V((-A_{-\rho})_{\rho \in\Rc})$ for all $A \in (\R^m)^{\Rc}$. The presence of a substitutional
impurity defect can then be encoded in the fact that possibly $V_\ell \neq V$
when $|\ell| < \rcut$. (We also allow $V_\ell \equiv V$ for all $\ell$, for
example to model a short range Frenkel pair.)

To simplify the notation we assume that $V_\ell(0) = 0$ for all $\ell$, which is equivalent to considering a potential energy-difference.

\subsection{Supercell Model}
\label{sec:intro:pbc}
Take a non-singular $\mB \in \R^{d \times d}$ with columns in $\L$, i.e., $\mA^{-1}\mB \in \Z^{d \times d}$. For each $N \in \N$ we let
\begin{equation*}
   \L_N := \L \cap \mB (-N,N]^d =  \mA \Z^d \cap \mB (-N,N]^d.
\end{equation*}
denote the discrete periodic supercell. We assume throughout that $N$
is sufficiently large such that $B_\rcut \cap \L \subset \L_N$. The associated
space of periodic displacements is given by
\begin{equation*}
   \Wper_N := \b\{ u : \L \to \R^m \bsep u \text{ is $\L_N$-periodic}\b\},
\end{equation*}
that is $u \in \Wper_N$ if and only if $u(\ell + 2N\mB n) = u(\ell)$ for all $n
\in \Z^d$. An equilibrium defect geometry is obtained by solving
\begin{align} \label{eq:intro:pbc_equil}
   &\us_N \in \arg\min \b\{ \E_N(u) \bsep u \in \Wper_N \b\},  \\
   %
   %
   \notag
   \text{where} \quad & \E_N(u)
   := \sum_{\ell \in \L_N} V_\ell(Du(\ell)) \qquad \text{for } u \in \Wper_N
\end{align}
is the potential energy functional for the periodic cell problem. In
\S~\ref{sec:results:tst} we will also consider more general critical points
$\delta \E_N(\us_N) = 0$.
For future reference, we also define the analogous functional for the
homogeneous (defect-free) supercell,
\begin{equation} \label{eq:intro:pbc_Ehom}
   \Ehom_N(u) := \sum_{\ell \in \L_N} V(Du(\ell))
    \qquad \text{for } u \in \Wper_N.
\end{equation}

Due to the assumption that $V_\ell(0) = 0$, the energy $\E_N(\us_N)$ can in fact
be interpreted as an energy difference, $\E_N(\us_N) - \Ehom_N(0)$, between the
defective and homogeneous crystal in the supercell approximation, called the
\emph{defect formation energy}. In \S~\ref{sec:intro:tdl_equil} we review the
limit, as $N \to \infty$, of \eqref{eq:intro:pbc_equil} and of the associated
energetics, which was established in \cite{EOS2016}.

\subsection{Supercell approximation of formation free energy}
\label{eq:intro:freeE}
The focus of the present work will be to incorporate vibrational entropy into
this model. Our first quantity of interest  is the \textit{defect-formation free
energy}, which is used, for example, to obtain the equilibrium defect
concentration~\cite{Putnis92, Walsh11} or to analyse defect
clustering~\cite{Seebauer09, Herbert2014}.

In the harmonic approximation model (thus incorporating only vibrational entropy
into the model) we approximate the nonlinear potential energy landscapes by
their respective quadratic expansions about the energy minima of interest,
\begin{align*}
   \Ehom_N(w) & \approx
            {\textstyle \frac{1}{2}} \b\<\Hhom_N w,w \b\>,
            \quad \text{and} \\
   \E_N(\us_N + w) & \approx
      \E_N(\us_N) + {\textstyle \frac{1}{2}} \b\<\Hdef_N(\us_N) w, w \b\>,
\end{align*}
where we used $\delta \Ehom_N(0) = \delta\Edef_N(\us_N) = 0$. Here and in the following, we use the notation $\Hdef_N(u) := \delta^2 \Edef_N(u)$, $\Hhom_N(u) := \delta^2 \Ehom_N(u)$, and $\Hhom_N := \Hhom_N(0)$ for the Hessians as mappings $\Wper_N \to \Wper_N$.

The harmonic approximation of the partition function is then given by
\begin{align}
\int_{\Wper_{N,0}} e^{-\beta \frac{1}{2} \<\Hdef_N w,w \>}\,du &= \Big[{\rm det}_{\Wper_{N,0}}\big(\beta \Hdef_N/(2 \pi)\big)\Big]^{-1/2}
= C_{\beta, N} ( \detp\Hdef_N)^{-1/2} \label{eq: gaussian integral}
\end{align}
where $C_{\beta,N} = (2 \pi /\beta)^{((2N)^d-1)m/2}$ and we introduced the
notation $\Wper_{N,0} := \{ u \in \Wper_N : \sum u =0\}$, as well as $\detp (A)
:= \prod_{j} \lambda_j$, where $\lambda_j$ enumerates the positive eigenvalues of $A$ (with multiplicities). We also implicitly used an assumption that we will formulate below in \eqref{eq: stability condition} and
\eqref{eq:stab homogeneous}, that $\Hdef_N(\us_N)$ and $\Hhom_N$ have only
one non-positive eigenvalue, namely $\lambda =0$ with all translations making up
the associated eigenspace (cf. Lemma~\ref{thm:pbc stab}).

The resulting harmonic approximation of formation free energy (derived analogously
to \eqref{eq:HTST}) is then given by
\begin{align}
   \notag
   \mathcal{F}_N(\us_N)
   &:= \E_N(\us_N)
      - \beta^{-1} \Big(-\smfrac{1}{2}\log \detp \Hdef_N(\us_N) +\smfrac{1}{2} \log\detp\Hhom_N \Big)
   \\ &=: \E_N(\us_N) - \beta^{-1} \mathcal{S}_N(\us_N).
   \label{eq:defn_SN}
\end{align}
The limit of $\E_N(\us_N)$ is identified in \cite{EOS2016}, and will be reviewed
in \S~\ref{sec:intro:tdl_equil}. One of the main results of this work is the
identification of the limit of the entropy difference $\lim_{N \to \infty} \mathcal{S}_N$, which we summarize
in \S~\ref{sec:convofentropy}.

\subsection{Thermodynamic Limit of Energy}
\label{sec:intro:tdl_equil}
To establish the limit of $\us_N$ and $\E_N(\us_N)$, we review the results of \cite{EOS2016}. For $u : \L \to \R^m$ let
\begin{equation*}
   |Du(\ell)|^2 := \sum_{\rho \in \Rc} |D_\rho u(\ell)|^2
   \qquad \text{and} \qquad
   \|Du\|_{\ell^2} := \b\|\,|Du|\b\|_{\ell^2}.
\end{equation*}
This defines a semi-norm on the natural spaces of compact and finite energy displacements
\begin{equation}
   \label{eq: spaces}
   \begin{split}
   \Wc &:= \b\{u:\L\to\R^m \bsep {\rm supp}(Du)~\text{is compact}\b\}
         \qquad \text{and}
   \\ \Wi &:= \b\{u:\L\to\R^m \bsep Du \in \ell^2\b\}.
   \end{split}
\end{equation}
The homogeneous and defective energy functionals for the infinite lattice are
given, respectively, by
\begin{equation} \label{eq: energy functional} \begin{split}
   \Ehom(u) &= \sum_{\ell\in\L} V(Du(\ell)) \qquad\text{and}  \\
   \Edef(u) &= \sum_{\ell\in\L} V_\ell(Du(\ell)) \qquad \text{for}~~u\in\dot{\W}^{\rm c}.
\end{split} \end{equation}
\begin{lemma}\cite[Lemma 2.1]{EOS2016} \label{lem:energyregular} $\Ehom, \Edef: (\Wc,\|D \cdot\|_{\ell^2})\to \R$ are continuous. In particular, there exist unique continuous extensions of $\Ehom$ and $\Edef$ to $\Wi$ as $\Wc$ is dense in $\Wi$. The extension will still be denoted by $\Ehom$ and $\Edef$. These extended functionals
   $\Ehom, \Edef: \Wi\to \R$ are $\kreg$ times continuously Fréchet
   differentiable.
\end{lemma}
We then set $\Hdef(u):=\delta^2 \Edef(u)$, $\Hhom(u):=\delta^2 \Ehom(u)$, and for convenience $\Hhom:=\Hhom(0)$.

\textbf{(STAB)}: We assume throughout that there exists a strongly stable equilibrium $\us \in\Wi$, i.e., $\delta \Edef (\us)=0$ and that there are constants $c_0,c_1>0$ such that
\begin{equation}
\label{eq: stability condition}
   c_0 \|D v\|_{\ell^2}^2 \leq \<\Hdef(\us)v,v\>
   \leq c_1 \|D v\|_{\ell^2}^2 \qquad \text{ for all } v \in \Wc.
\end{equation}
A necessary condition for \eqref{eq: stability condition} is that the
homogeneous lattice is stable, i.e.,
\begin{equation} \label{eq:stab homogeneous}
   c_0 \| D v \|_{\ell^2}^2  \leq \< \Hhom v,v\>
   \leq c_1 \|D v\|_{\ell^2}^2
    \qquad \text{ for all } v \in \Wc.
\end{equation}
(Note that the upper bounds in \eqref{eq: stability condition}, \eqref{eq:stab homogeneous} are immediate consequences of $\E \in C^p$ and are stated here only for the sake of convenience.)
%
%
\begin{theorem}{\cite[Thm 1]{EOS2016}}
\label{thm: EOS2016}
   Suppose that $u \in \dot{\W}^{1,2}$ is a critical point of $\Edef$, and that
   \eqref{eq:stab homogeneous} holds, then there exists a constant $C>0$ such
   that for $1\leq j\leq \kreg-2$ and for $|\ell|$ sufficiently large
   \begin{equation}
   \label{eq: decay estimate}
      |D^j u(\ell)|\leq C|\ell|^{1-d-j}.
   \end{equation}
\end{theorem}

Strong stability \eqref{eq: stability condition} and
regularity \eqref{eq: decay estimate} imply convergence of the
supercell approximation:

\begin{theorem}{\cite[Thm 3]{EOS2016} and \cite[Thm 2.1]{2018-uniform}}
   \label{thm:pbc convergence}
   For $N$ sufficiently large, \eqref{eq:intro:pbc_equil} has a locally unique
   solution $\us_N$ (up to translations) satisfying
   \begin{align}
      \| D\us_N - D\us \|_{\ell^2(\L_N)} &\lesssim N^{-d/2}, \label{eq:strainerror} \\
      \| D\us_N - D\us \|_{\ell^\infty(\L_N)} &\lesssim N^{-d}, \label{eq:uniformstrainerror} \\
      \b| \Edef_N(\us_N) - \Edef(\us) \b| &\lesssim N^{-d}. \label{eq:energyerror}
   \end{align}
\end{theorem}

A key ingredient in the proof of Theorem \ref{thm:pbc convergence} is the
stability of the supercell approximation, i.e., positivity of the Hessians
$\Hdef_N = \Hdef_N(\us_N)$ and $\Hhom_N$:

\begin{lemma}{\cite[Eq (18)]{EOS2016}}
   \label{thm:pbc stab}
   For $N$ sufficiently large and for all $v \in \Wper_N$,
   \begin{align*}
      \< \Hdef_N v,v\> &\geq \smfrac12 c_0 \|D v\|_{\ell^2(\L_N)}^2, \qquad \text{and} \\
      \< \Hhom_N v, v \> &\geq \smfrac12 c_0 \|Dv\|_{\ell^2(\L_N)}^2.
   \end{align*}
   In particular, for $N$ sufficiently large, \eqref{eq: gaussian integral} holds.
\end{lemma}

\subsection{Spatial decomposition of entropy}
\label{sec:results:decompositionS}
Our goal is to characterise the thermodynamic limit of the entropy
difference $\mathcal{S}_N \to \mathcal{S}$ as $N \to \infty$,
as an entropy difference, which \emph{formally} one might expect to be of the form $\mathcal{S}(u) = -\smfrac{1}{2}\log\detp \Hdef(u) +\smfrac{1}{2} \log\detp\Hhom$, but this expression is not well-defined.

In the following, let $\pi_N :  \Wper_N \to \Wper_N$ be the orthogonal projector
onto the space of constant displacements. This allows us to define an operator
that acts as $(\Hhom_N)^{-1/2}$ orthogonal to constant displacements:
\begin{lemma} \label{lem:FopN}
There exist linear operators $\Fop_N : \Wper_N \to \Wper_N $ such that
\begin{align}
&\Fop_N^* = \Fop_N, \label{eq:FopN1}\\
&\Fop_N \Hhom_N \Fop_N + \pi_N = I_{\Wper_N}, \label{eq:FopN2}\\\
&\Fop_N \pi_N = \pi_N \Fop_N =0. \label{eq:FopN3}\
\end{align}
\end{lemma}
These operators and additional properties will be discussed in detail in
\S\S\,\ref{sec:limit:estimateFN}--\ref{sec:limit:spectralpropertiesFN}. It follows that
\[(\Fop_N +\pi_N)(\Hhom_N+\pi_N) (\Fop_N +\pi_N)= I_{\Wper_N},\]
and we can rewrite the entropy difference as
\begin{align}
\label{eq:logdetminuslogdetcalc}
   \mathcal{S}_N(u) &= -\smfrac{1}{2}\log\detp  \Hdef_N(u) +\smfrac{1}{2}  \log\detp \Hhom_N\\ \nonumber
&= -\smfrac{1}{2} \log\det  \big(\Hdef_N(u) + \pi_N\big)+\smfrac{1}{2}  \log\det \big(\Hhom_N + \pi_N\big)\\ \nonumber
&= -\smfrac{1}{2}\log\det  \big(\Hdef_N(u) + \pi_N\big) -\log \det (\Fop_N + \pi_N\big)\\ \nonumber
&= -\smfrac{1}{2}\log\det  \big((\Fop_N + \pi_N) (\Hdef_N(u) + \pi_N)(\Fop_N + \pi_N)\big)\\ \nonumber
&= -\smfrac{1}{2}\log\det  \big(\Fop_N\Hdef_N(u)\Fop_N + \pi_N\big)\\ \nonumber
&= -\smfrac{1}{2}{\rm Trace} \log \big(\Fop_N\Hdef_N(u)\Fop_N + \pi_N\big).
\end{align}

While ``$\log\det$'' is a sum over eigenvalues, which are global objects,
the key observation is that ``${\rm Trace} \log$''  can be interpreted as a
sum over atoms. Thus, upon defining
\begin{equation} \label{eq:def SNell}
   \mathcal{S}_{N,\ell}(u) := -\smfrac{1}{2}{\rm Trace} \Big[\log \big(\Fop_N\Hdef_N(u)\Fop_N + \pi_N\big)\Big]_{\ell\ell},
\end{equation}
where $[L]_{\ell\ell}$ denotes the $3 \times 3$ block of $L$ corresponding to
an atomic site $\ell \in \L$, we obtain
\begin{equation}
   \mathcal{S}_N(u)  = \sum_{\ell \in \L_N} \mathcal{S}_{N,\ell}(u).
\end{equation}
This spatial decomposition of the entropy will play a central role throughout this
paper. Indeed, it is straightforward to write down a suitable limit quantity for each $\mathcal{S}_{N,\ell}$,
\begin{equation} \label{eq:defSell1}
\begin{split}
   \mathcal{S}_\ell(u) &:= -\smfrac{1}{2}{\rm Trace} \Big[\log \big(\Fopd \Hdef(u) \Fop\big)\Big]_{\ell\ell}, \\
   \Fop &:= \big( \Hhom \big)^{-1/2} \in \mathcal{L}\big(\ell^2, \dot{\W}^{1,2}\big).
\end{split}
\end{equation}
For a rigorous definition of $\Fop$ via Fourier transform, as well as
$\log\big(\Fopd \Hdef(u) \Fop\big): \ell^2(\Lambda) \to \ell^2(\Lambda)$ see
\S\S\,\ref{sec:defn_F}--\ref{sec:funccalc}. Since $\ell^2(\Lambda)$ does not
contain any constant displacements, there is no need for a projector analogous
to $\pi_N$ in the definition of $\Fop$.

We will call $\mathcal{S}_{N,\ell}$ and $\mathcal{S}_\ell$ \emph{site
entropies}, since they are contributions from individual lattice sites to the
global (vibrational) entropy. There is moreover a direct analogy with a
definition of site energies in the tight-binding model \cite{2015-qmtb1}.

To formulate our main results, we also define the corresponding
homogeneous local entropy
\begin{equation}
   \label{eq: Shom_l}
   \Shom_\ell (u) := -\smfrac{1}{2}{\rm Trace} \Big[\log \big(\Fopd \Hhom(u) \Fop\big) \Big]_{\ell\ell}.
\end{equation}

The next steps are to define the total entropy $\SS$ and show that it is the limit of $\SS_N$.

As we will see in Proposition \ref{prop: localitydecomposedentropy}, however, the operator $\log\big(\Fopd \Hdef(u) \Fop\big)$ cannot be expected to be of trace class. Consequently we cannot simply define  $\SS(u) := -\smfrac{1}{2}{\rm Trace} \log\big(\Fopd \Hdef(u) \Fop\big)$ which would be the sum of the site contributions $\SS_\ell(u)$, but
 a more careful definition of $\SS(u)$ is required.

In this analysis we heavily employ estimates quantifying the \emph{locality} of
the site entropies. This locality is twofold. First, the site entropies
$\SS_{\ell}$ become smaller as the distance to the defect $\lvert \ell \rvert$
grows larger, and, second, each individual $\SS_{\ell}$ only depends weakly on
far away atom sites which is quantifiable by the decay of derivatives such as
$\frac{\partial \SS_\ell(u)}{ \partial D u(n)} $ as $\lvert \ell - n \rvert$
grows. More precisely, one has estimates of the form
\begin{equation} \label{eq:localityintuition}
\Big\lvert \frac{\partial \SS_\ell(\us)}{ \partial D u(n) } - \frac{\partial \Shom_\ell(0)}{ \partial D u(n) } \Big\rvert \lesssim \lvert \ell - n \rvert^{-2d} \lvert n \rvert^{-d} + \text{higher order terms}.
\end{equation}
While we will not \emph{explicitly} use or prove it, \eqref{eq:localityintuition}
and similar statements for second derivatives are implicit in Proposition
\ref{prop: localitydecomposedentropy} and its proof. More importantly,
 \eqref{eq:localityintuition} gives a good first intuition about the locality
 of $\SS_\ell(u)$ and why one can hope that its sum over $\ell$ may be
 controlled.

\subsection{Definition and convergence of entropy}
\label{sec:convofentropy}
Let us come to the first main result of the present paper. The following theorem establishes a rigorously defined notion of the limit entropy difference $\SS(\us)$ and justifies this definition via a thermodynamic limit result.

\begin{theorem}
   \label{theo: main result}

   (1) $u \mapsto\SS_\ell(\us+u),\, u \mapsto \Shom_\ell(u)$ are well-defined and $C^{p-2}$ on $B_\delta(0) \subset
   \Wi$, $\delta>0$ sufficiently small.

   (2) The sequence $\ell \mapsto \SS_\ell(\us) - \big\< \delta \Shom_\ell(0), \us \big\>$ belongs to $\ell^1(\L)$ and hence

   \begin{equation} \label{eq:results:defn of S}
      \mathcal{S}(\us) := \sum_{\ell \in \L} \Big( \SS_\ell(\us) - \big\< \delta \Shom_\ell(0), \us \big\> \Big)
   \end{equation}
   is well-defined.
   %

   (3) Let $\us_N \in \Usper_N$ denote the locally unique solution to
   \eqref{eq:intro:pbc_equil} identified in Theorem~\ref{thm:pbc convergence},
   then
   \begin{equation} \label{eq:results:convSNS}
      \big| \SS(\us) - \SS_N(\us_N) \big| \lesssim N^{-d} \log^5(N).
   \end{equation}
   In particular,
   \begin{equation} \label{eq:results:convFNF}
      \big| \F(\us) - \F_N(\us_N) \big| \lesssim N^{-d} + \beta^{-1} N^{-d} \log^5(N),
   \end{equation}
   where $\F(\us) := \E (\us)- \beta^{-1}\SS(\us)$.
\end{theorem}
\begin{proof}
The proofs of (1) and (2) are given in \S\,\ref{sec: locality}, the proof of (3)
in \S\,\ref{sec:limit}.
\end{proof}
\begin{remark} \label{rem: main theorem}
The definition of $\SS(\us)$ in the theorem can be interpreted as follows: One
can show (with the methods in the proof of Proposition \ref{prop:
localitydecomposedentropy}) that for $u \in B_{\delta'}(\us) \cap \Wc$, the sums
$\SS(u)=\sum_\ell \SS_\ell(u)$ and $\sum_{\ell \in \L} \< \delta \Shom_\ell(0),
u \>$ converge absolutely and $\sum_{\ell \in \L} \< \delta \Shom_\ell(0), u
\>=0$. As $\Wc \subset \Wi$ is dense, \eqref{eq:results:defn of S} then becomes
the unique continuous extension. The renormalised expression
\eqref{eq:results:defn of S} becomes necessary, as the separate sums do not
converge any longer for $\bar{u}$.

(2) There is no reason to believe that the logarithmic factor $\log^5(N)$ is sharp. However, we will discuss the sharpness of the rate $N^{-d}$ up to logarithmic terms in \S\,\ref{sec:results:conclusion}.
\end{remark}

\subsection{Application to defect migration}
\label{sec:results:tst}
Recall from \S~\ref{sec:intro} that transition state theory (TST) characterises
the transition rate from one stable defect configuration (energy minimum) to
another via the associated transition state, i.e., the lowest saddle point that
must be crossed. A free energy difference between saddle and minimum describes
the transition rate. Thus, our techniques to characterise the thermodynamic
limit of defect formation free energy are almost directly applicable to (harmonic) TST as well.

Suppose for the moment, that in addition to a sequence of energy minima $\us_N$
there exists a sequence of saddle points $\usaN \in \Wper_N$ with associated
unstable eigenpair $\bar\phi_N \in \Wper_N, \bar\lambda_N< 0$ such that
\begin{equation} \label{eq:index 1 saddle - supercell}
\begin{split}
   \delta \E_N(\usaN) &= 0, \\
   \HsN \bar\phi_N &= \bar\lambda_N \bar\phi_N, \\
   \bar\lambda_N & < 0,  \qquad \text{and} \\
   \< \HsN v, v \> & > 0 \qquad \text{for } v \in {\Usper_{N,0}}, \quad \text{with } ( v, \bar\phi_N )_{\ell^2(\Lambda_N)}= 0,\\
\end{split}
\end{equation}
where $\HsN := \delta^2 \E_N(\usaN)$. Then, the transition rate according to HTST is given by \eqref{eq:HTST}, i.e.,
\begin{equation} \label{eq:results:defn kN}
\mathcal{K}^{\rm HTST}_N :=
   \bigg(\frac{
      \textstyle \prod \lambda_j^{\rm min}
   }{
      \textstyle \prod \lambda_j^{\rm saddle}
   }\bigg)^{1/2}
   \, \exp\Big( - \beta \big[
      \E_N(\usaN) -
      \E_N(\us_N) \big]
      \Big),
\end{equation}
where the $\lambda_j^{\rm min}$ and $\lambda_j^{\rm saddle}$ enumerate the positive eigenvalues of, respectively, $\Hdef_N$ and $\HsN$ including multiplicities.
While \eqref{eq:results:defn kN} is the  common definition, it is more
convenient for our purpose to restate it as
\begin{align}
   \notag
   \khtst_N &:= \exp\Big( - \beta \Delta  \F_N \Big) :=
   \exp\Big( - \beta \big( \Delta \E_N - \beta^{-1} \Delta \SS_N \big) \Big),
   \qquad \text{where}
   \\
   \label{eq:results:defn kN v2}
   \Delta\E_N &:= \E_N(\usaN) - \E_N(\us_N), \qquad \text{and}
   \\
   \notag
   \Delta\SS_N &:= \SS_N(\usaN) - \SS_N(\us_N)\\
   \notag &= -\smfrac{1}{2} \log \detp \HsN + \smfrac{1}{2} \log \detp H_N \\
   \notag &=
   -\smfrac{1}{2} \sum \log\lambda_j^{\rm saddle}
   +\smfrac{1}{2} \sum \log\lambda_j^{\rm min},
\end{align}

This establishes the connection to the vibrational entropy functional analysed in Theorem~\ref{theo: main result}. Note that, $\SS_N(\usaN)$ is defined in the same way for the saddle point, as $\detp$ now also excludes the negative eigenvalue as well.

With the natural embeddings $( \dot{\W}^{1,2})' \hookrightarrow \ell^2
\hookrightarrow \dot{\W}^{1,2}$, the canonical thermodynamic limit of the saddle
point and natural analogue of {\bf (STAB)} can be formulated as
\begin{equation} \label{eq:index-1-saddle-ell2}
   \begin{split}
      \delta \E(\usa) &= 0, \\
      \Hs \bar\phi &= \bar\lambda \bar\phi, \\
      \< \Hs v, v \> & \geq c_0 \| D v \|_{\ell^2}^2 \qquad \text{for all } v \in \dot{\W}^{1,2} \text{ with } \< v, \bar\phi \>_{\dot{\W}^{1,2}, (\dot{\W}^{1,2})'} = 0, \\
      \bar\lambda &< 0 \quad \text{ and } c_0 > 0.
   \end{split}
\end{equation}

We now make \eqref{eq:index-1-saddle-ell2} our starting assumption and prove the existence of  a sequence of approximate saddle points in the supercell approximation. Moreover, we can establish the limit of the transition rate. In that part, we will also assume that naturally $\E(\usa) > \E(\us)$.

\begin{theorem}  \label{th:main thm saddle}
   (1) Suppose that \eqref{eq:index-1-saddle-ell2} holds, then for $N$ sufficiently
   large there exist $\usaN, \bar\phi_N, \bar\lambda_N$ satisfying
   \eqref{eq:index 1 saddle - supercell}, such that
   \begin{align*}
      \| D\usaN - D\usa \|_{\ell^\infty} + \| \bar\phi_N - \bar\phi \|_{\ell^2}
      + |\bar\lambda_N - \bar\lambda|
      + |\E_N(\usaN) - \E(\usa)|  &\lesssim N^{-d}.
   \end{align*}

   (2) The limit $\khtst := \lim_{N \to \infty} \khtst_N$ exists,
   with rate
   \[
      \big| \khtst_N - \khtst \big| \lesssim N^{-d} \log^5(N),
   \]
   and is characterised in \eqref{eq:defktst}.
\end{theorem}
\begin{proof}
   The proof of (1) is an extension of \cite{2018-uniform} and is given in
   \S~\ref{sec:approx-saddle}. The proof of (2) is given in \S~\ref{sec:tst:k}.
\end{proof}

\begin{remark} \label{rem:relative_error}
For large $\beta$, the transition rate $\khtst$ becomes very small. In this case one might prefer to consider the relative error, which can be bounded by
\[
      \frac{\big| \khtst_N - \khtst \big|}{\khtst} \lesssim e^{C\beta N^{-d}} (\beta N^{-d} + N^{-d} \log^5(N)),
   \]
which follows from the estimates in the proof of Theorem \ref{th:main thm saddle}.
%
\end{remark}

\begin{remark} \label{rem:main thm saddle}
   The characterisation of the limit $\khtst = \lim_N \khtst_N$ is
   not as explicit as the limit $\SS(\us)$ in \eqref{eq:results:defn of S}, but
   is presented in full in \S\,\ref{sec:saddle}.
   %
   %
\end{remark}

\begin{remark} \label{rem:saddle assumption}
   At first glance our assumption \eqref{eq:index-1-saddle-ell2}, which
   {\em postulates} the existence of a stable saddle, may seem very strong.
   This is made necessary due to our weak assumptions on
   the interatomic potential, aimed at including realistic models of interaction
   in our analysis.

   However, one can also show that \eqref{eq:index-1-saddle-ell2} are the
   only possible limits of a sequence of  index-1 saddle
   points $\usaN$ with uniform upper and lower bounds on the spectrum, giving
   at least a partial justification.
   %
\end{remark}

\subsection{Conclusions and Discussion}
\label{sec:results:conclusion}
We have developed a technique to analyse the vibrational entropy of a
crystalline defect in the limit of an infinite lattice. Two applications of this
technique are to characterise the limit of formation free energy as well as of
transition rate, both in the harmonic approximation. These results are
interesting in their own right in that they demonstrate that boundary effects
vanish in this limit, but more generally establish the mathematical techniques
to study existing and develop novel coarse-grained models and multi-scale
simulation schemes incorporating temperature effects.

We briefly outline three extensions that may require substantial additional work:
\begin{enumerate}
	\item Extension to interstitials and vacancies: We expect that our
	convergence results can be exetended to these cases, with only minor
	differences in the characterisation of the limit. This is supported by
	numerical evidence displayed in Figure~\ref{fig:convergence}. The main
	additional difficulty comes from the different number of degrees of freedom
	compared to the homogeneous lattice when treating the Hessians. A possible
	approach is to extend the smaller Hessian to the larger dimension and perform
	a calculation similar to \eqref{eq:logdetminuslogdetcalc}. The overall
	strategy then proceeds similarly to what we present here, however, there will
	be an additional finite rank perturbation. This term is of a different
	structure for interstitials and vacancies and requires additional work.

   \item Extension to topological defects such as dislocations and cracks: the key
   difficulty is that an inhomogeneous reference configuration must
   be used in the analysis, for which the Green's functions are more difficult
   to estimate.
   \item It is in general difficult to observe logarithmic contributions in
   numerical tests, hence our numerical tests in
   Figure~\ref{fig:convergence} should not be taken as evidence
   that the sharp convergence rate for the entropy is indeed $O(N^{-3})$.
   It is unclear to us, at present, whether or not the sharp rate should
   include logarithmic contributions.

   In the example shown in Figure~\ref{fig:convergence} we even observe the
   rate $O(N^{-4})$ for $\Delta \SS_N$. Since the rate for $\Delta
   \mathcal{E}_N$ is still $O(N^{-3})$ we speculate that this is a
   pre-asymptotic effect likely caused if the dipole moments of the defect in
   its minimum and saddle point states nearly coincide; see
   \cite{BHOdefectdevelopment} for a detailed discussion of  such cancellation
   and near-cancellation effects.
\end{enumerate}

\begin{figure}
   \includegraphics[width=0.8\textwidth]{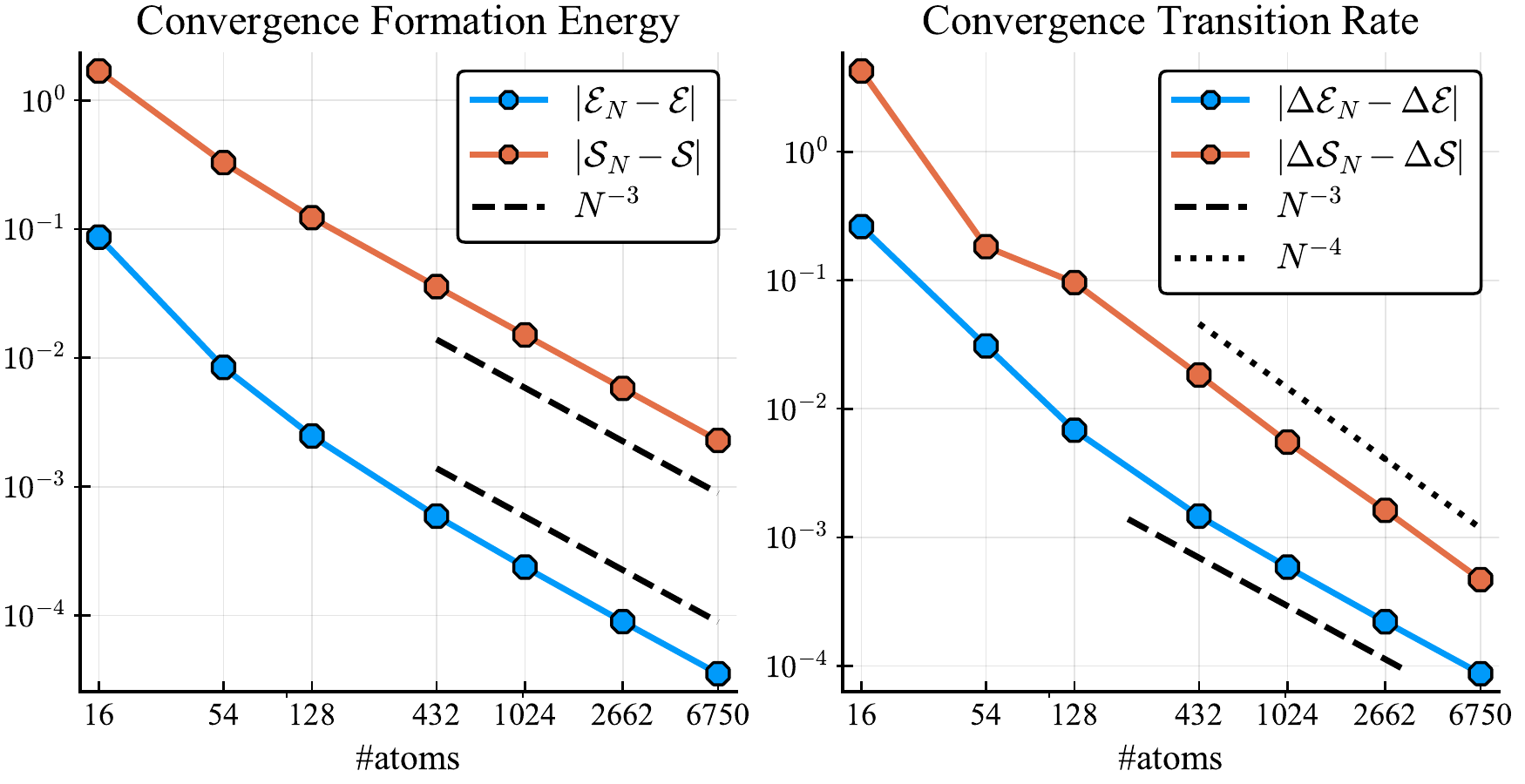}
   \caption{
      Convergence of energy and entropy contributions to formation energy
      $\mathcal{F}_N$ and transition rate $\mathcal{K}_N$, for a vacancy defect
      in bcc tungsten (W) modelled by a Finnis-Sinclair (EAM) potential
      \cite{FSWang2013}, employing a cubic computational cell, i.e., $\mB \propto
      I$.
   }
   \label{fig:convergence}
\end{figure}

%% file: resolvents.tex

\section{Resolvent Estimates}
\label{sec:resolventestimates}
\subsection{Notation / Preliminaries}

\label{sec:resolvents:notation}
Let us fix some more notation.
\begin{itemize}
\item  $|r|$ is the standard Euclidean norm and
\begin{equation} \label{eq:defn_lognorms}
|r|^{-n}_{l^k}:=(|r|+1)^{-n} \log^k(e+|r|),
\end{equation}
where $r$ can be a vector or scalar. For $M > 0$, we extend the definition by setting
\[
   \lvert n \rvert_{l^k, M}^{-d}
   = \min\big\{\lvert n \rvert_{l^k}^{-d}, \lvert M \rvert_{l^k}^{-d} \big\}.
\]
\item For $m,n \in \L$ and $M > 0,k \geq 0$ we then define
\begin{align}
\Lr_{k}(n,m) &:= |n|_{l^k}^{-d}|n-m|_{l^k}^{-d}+|m|_{l^k}^{-d}|n-m|_{l^k}^{-d}+|n|_{l^k}^{-d}|m|_{l^k}^{-d}, \label{eq:defLr}\\
\Lr_{k}^M (n,m) &:= \lvert n \rvert_{l^k, M}^{-d} \lvert n-m \rvert_{l^k}^{-d} + \lvert m \rvert_{l^k, M}^{-d} \lvert n-m \rvert_{l^k}^{-d} + \lvert n \rvert_{l^k, M}^{-d} \lvert m \rvert_{l^k,M}^{-d}. \label{eq:defLrM}
\end{align}
\item We use the semi-discrete Fourier transform
\begin{equation}
\label{eq: sdFourier}
\hat u (k):=\sum_{\ell \in \Lambda} e^{i k\cdot \ell} u(\ell),\quad\text{with inverse}\quad u(\ell)=\frac{1}{\lvert \B \rvert}\int_{\B} e^{-i k\cdot \ell} \hat u(k)\,dk,
\end{equation}
where $\B = \pi \mA^{-T} (-1,1)^d$ is a fundamental domain of reciprocal space
(equivalent to the first Brillouin zone) and has the volume $\lvert \B \rvert =
\frac{(2\pi)^d}{\lvert \det \mA \rvert}$.
\end{itemize}

\subsection{Estimate of $\Fop$}
\label{sec:defn_F}
We begin by defining and establishing decay estimates for the operator ${\bf
F}$. Since $\Hhom$ is circulant, it is natural to formally represent $\Fop w = F \ast
w$ and define $F$ via its Fourier transform. First, recall that
\begin{equation*}
   \<\Hhom u,v\> =
   \sum_{\ell\in \Z^d}\nabla^2 V (0) [Du(\ell), Dv(\ell)],
\end{equation*}
then applying the SDFT we obtain
\begin{align} \label{eq:F-transform_Hhom}
   \<\Hhom u,v\>
   &= \frac{1}{\lvert \B \rvert} \int_{\B} \hat{u}(k)^* \hat{h}(k) \hat{v}(k) \, dk,\\
    a^T \hat{h}(k) b &:=\nabla^2V(0)[((e^{-ik\cdot\rho}-1) a)_{\rho \in \Rc},((e^{ik\cdot\rho}-1)b)_{\rho \in \Rc}]. \nonumber
\end{align}
One can also reduce $\hat{h}(k)$ to the simpler form
\begin{equation}
\hat{h}(k) = 4 \sum_{\rho \in \Rc'} A_\rho
                 \sin^2\big( \smfrac{k \cdot \rho}{2} \big), \label{eq:fourierhamiltoniansin}
\end{equation}
with $\Rc' = (\Rc \cup \{0\}) + (\Rc \cup \{0\})$, see \cite[Sec. 6.2]{EOS2016}.
Furthermore, {\bf (STAB)} implies that $c_0 |k|^2 I \leq \hat{h}(k) \leq c_1 |k|^2 I$
in the matrix sense for all $k \in \B$, see \cite{2012-M2AN-CBstab}. We observe that
$|\hat{h}(k)^{-1/2}| \lesssim |k|^{-1}$ as $|k| \to 0$, hence
we can define
\begin{align} \label{eq:defn_F_fourier}
   F(\ell) &:= \frac{1}{\lvert \B \rvert} \int_\B e^{-i k\cdot \ell} \hat{F}(k) \, dk,
   \qquad \text{where} \quad \hat{F}(k) = \hat{h}(k)^{-1/2}, \\
    \label{eq:defn_Fop}
   (\Fop u)(\ell) &:=
      \sum_{m \in \L} \big( F(\ell-m) - F(-m) \big) u(m).
\end{align}
The constant shift $\sum_m F(-m) u(m)$ in the definition of $\Fop u$ ensures
that $\Fop u$ is well-defined (when $d = 2$ the separate sums need not converge).

\begin{lemma} \label{th:properties_F}
   Let $F : \L \to \R^{m \times m}$ be defined by \eqref{eq:defn_F_fourier} and $\Fop$ by \eqref{eq:defn_Fop}, then
   \begin{enumerate}
   \item[(i)]  For any $\brho\in\Rc^j$, $j \geq 0$, there exists a constant $C$ such that
   \begin{equation*}
      |D_{\brho} F(\ell)|\leq C |\ell|_{l^0}^{1-d-j}\quad \forall \ell\in\Z^d.
   \end{equation*}
   \item[(ii)] $\Fop \in \mathcal{L}(\ell^2,\Wi)$.
   \item[(iii)] $\Fopd \Hhom \Fop = I$, understood as operators $\ell^2 \to \ell^2$.
   \end{enumerate}
\end{lemma}
\begin{proof}
   Our argument closely follows the Green's function estimate of \cite{EOS2016},
   adapted to the fact that $F$ is the square-root of a Green's function.
   The details are given in \S~\ref{sec:proof_properties_F}.
\end{proof}

\subsection{Functional calculus}
\label{sec:funccalc}
Suppose that $A : \ell^2 \to \ell^2$ is a bounded, self-adjoint operator with $\sigma(A) \subset [\underline\sigma, \overline\sigma]$,
 $0 < \underline\sigma < 1 < \overline\sigma$, and $\mathcal{C}$ is a contour
that encloses $[\underline\sigma, \overline\sigma]$, but not the origin,
then \cite[Ch. VII.3]{Dunford1958-qc}
\[
   \log A := \frac{1}{2\pi i} \oint_{\mathcal{C}} \log z (z - A)^{-1} \, dz
\]
defines a bounded, self-adjoint operator on $\ell^2$. More generally, let
$A : \ell^2 \to \ell^2$ be bounded, self-adjoint with
$\sigma(A) \cap (0, \infty) \subset [\underline\sigma, \overline\sigma]$, we can use the \emph{same
contour} to define
\begin{equation}  \label{eq:defn-log+}
   \logp A := \frac{1}{2\pi i} \oint_{\mathcal{C}} \log z (z - A)^{-1} \, dz.
\end{equation}
This generalisation will be crucial to be able to apply the subsequent analysis
not only to the formation free energy (Theorem~\ref{theo: main result}), but
also to the analysis of transition rates (Theorem~\ref{th:main thm saddle}). As clearly $\log A = \logp A$ in the case that $\sigma(A) \subset [\underline\sigma, \overline\sigma]$, it suffices to consider $\logp A$ in the following.

In order to apply this in our setting we substitute $A = \Fopd H^t(u) \Fop$, where
\[
   H^t(u) := (1-t) \Hhom + t \Hdef(u), \qquad t \in [0, 1],
\]
for $u$ in a neighbourhood of $\us$.
Our first step is therefore to show that these operators remain uniformly
bounded above and below.

\begin{lemma} \label{th:spectrum_bound}
   Let $\us$ be a stable minimiser of $\E$, then there exist $\epsilon,
   \underline\sigma, \overline\sigma > 0$ such that,
   for all $u \in B_\epsilon(\us) \subset \Wi$ and $t \in [0, 1]$,
   \begin{equation} \label{eq:logcondition}
       \sigma\big[ \Fopd H^t(u) \Fop \big] \subset [\underline\sigma, \overline\sigma].
   \end{equation}

   More generally, assume $u_\infty \in \Wi$ satisfies $\sigma(\Fopd H^t(u_\infty) \Fop) \cap (-\underline\sigma, \infty) \subset [2\underline\sigma, \overline\sigma/2]$ for some $0 < \underline\sigma < \overline\sigma$, then
      \begin{equation} \label{eq:logpluscondition}
      \sigma\big[ \Fopd H^t(u) \Fop \big] \cap (0, \infty) \subset [\underline\sigma, \overline\sigma].
   \end{equation}
   for all $u \in B_\epsilon(u_\infty) \subset \Wi$ and $t \in [0, 1]$.
\end{lemma}
\begin{proof}
	According to {\bf (STAB)} we have
	\begin{equation}\label{eq:prf:spectrum_bound:1}
      c_0 \|Dv\|_{\ell^2}^2
      \leq \< H^t(\bar{u}) v, v \>
      \leq c_1 \|Dv\|_{\ell^2}^2.
   \end{equation}
 Hence, $H^t(\bar{u}) \in \mathcal{L}(\Wi, (\Wi)^*)$ and Lemma~\ref{th:properties_F}
   implies that $\Fopd H^t(\bar{u}) \Fop \in \mathcal{L}(\ell^2,\ell^2)$. Since $\Fopd \Hhom \Fop = I$ (see again Lemma~\ref{th:properties_F}) it
   follows that
   \begin{equation*}
      c_0 \|D \Fop w\|_{\ell^2}^2 \leq
         \< \Fopd \Hhom \Fop w, w \> = \| w \|_{\ell^2}^2
      \leq c_1 \|D \Fop w\|_{\ell^2}^2.
   \end{equation*}
Substituting $v = \Fop w$ into \eqref{eq:prf:spectrum_bound:1} we obtain
   \[
      \smfrac{c_0}{c_1} \|w\|_{\ell^2}^2
         \leq \< H^t(\bar{u}) \Fop w, \Fop w \>
         = \< \Fopd H^t(\bar{u})\Fop w, w \>
         \leq \smfrac{c_1}{c_0} \|w\|_{\ell^2}^2.
   \]
If now $u \in B_\epsilon(\us)$, we use the assumption $V_{\ell} \in C^3$ to estimate
\begin{align*}
\Big\lvert \big\< (H^t(u)-H^t(\bar{u})) \Fop w, \Fop w \big\>\Big\rvert &\lesssim \lVert D \Fop w \rVert_{\ell^2}^2 \lVert Du -D\bar{u} \rVert_{\ell^\infty}
\lesssim \epsilon \lVert w \rVert_{\ell^2}^2,
\end{align*}
which proves the remaining claims.
\end{proof}

In light of  the foregoing lemma, there  exists a contour
$\mathcal{C}$ encircling $[\underline\sigma, \overline\sigma]$ but not the
origin, such that, for $u \in B_\epsilon(u_\infty)$ and for all $t \in [0, 1]$,
\begin{equation} \label{eq:log-FHtF-contour}
   \log^+ \hspace{-0.3em}\big[ \Fopd H^t(u) \Fop \big]
   = \frac{1}{2\pi i} \oint_{\mathcal{C}} \log z\, \Rr^t_z(u) \,dz,
      \quad \text{where } \Rr^t_z = \Rr^t_z(u)
            := (z-\Fopd H^t(u) \Fop)^{-1}.
\end{equation}
From now on, we will fix this contour and always have $z \in \mathcal{C}$ and $t
\in [0, 1]$. We will also use the notation $\Rhom_z := \Rr^0_z(u)$ and $\Rr_z(u):= \Rr^1_z(u)$. We remark that, since $\Fopd\Hhom\Fop = I$, we have
$\Rhom_z=(z-I)^{-1} = (z-1)^{-1} I$.

To exploit the representation \eqref{eq:log-FHtF-contour} we will analyse the
resolvents $\Rr_z^t$. Specifically, we will estimate how $[\Rr^t_z]_{\ell n}$
decays as $|\ell|, |n| \to \infty$.

\subsection{Finite-rank corrections}
A basic technique that we will employ in the resolvent decay estimates  is to
decompose a Hessian operator $H$ into two components $H = H_{\rm r} + H_{\rm h}$
where  $H_{\rm r}$ has finite rank while $H_{\rm h}$ is close to $\Hhom$. To
estimate the correction to the resolvent due to $H_{\rm r}$, the following
lemma shows that we can instead estimate powers of the finite rank correction.

\begin{lemma}
\label{lem: finite-rank correction}
Let $X$ be a Hilbert space.
\begin{enumerate}[(i)]
\item Let $A \in \mathcal{L}(X, X)$ be a bounded linear operator with range of
finite dimension at most $r \in \N$ and $I +A$ is invertible, then there exist
$c_1(A), \dots, c_{r+1}(A) \in \R$ such that
\begin{equation}
\label{eq: finite-rank correction}
   (I+A)^{-1}=I + \sum_{j=1}^{r+1} c_j(A) A^j.
\end{equation}
If $U \subset \CN$ such that  $(I + \gamma A)$ is invertible for all $\gamma \in
U$, then $\gamma \mapsto c_j(\gamma A)$ are continuous functions on $U$.
\item More generally, let $X = X_1 \oplus X_2$  be a fixed orthogonal
decomposition with  ${\rm dim} (X_1) \leq r$ then \eqref{eq: finite-rank
correction} holds for all $A$ for which $X_2 \subset {\rm ker} A$ and $(I+A)$ is
invertible. The coefficients can be written as $c_j(A) = d_j(\pi_{X_1} A
|_{X_1})$ where $\pi_{X_1} A |_{X_1} \colon X_1 \to X_1$ is the restriction and
projection of $A$ to $X_1$ and the $d_{j}$ are continuous  on the
finite-dimensional set $\{B \in \mathcal{L}(X_1, X_1) \colon (I+B) \text{ is
invertible} \}$.
%
%
\end{enumerate}
\end{lemma}
\begin{proof}
   The result is a consequence of the Cayley--Hamilton theorem;
   we give the complete proof in \S~\ref{sec:proof:finite-rank correction}.
\end{proof}

\subsection{Resolvent estimates}
\label{sec:resolvent_estimates}
\quad The goal of this section is to estimate $\Rr^t_z(u) = (z - \Fopd H^t(u) {\bf
F})^{-1}$, where
\[
   u \in \mathcal{U}:=\mathcal{U}(\underline\sigma, \overline\sigma, C) :=
      \big\{ u \in \Wi \,:\, u \text{ satisfies } \eqref{eq:logpluscondition} \text{ with } \underline\sigma, \overline\sigma\ \text{ and }\ |Du(\ell)| \leq C |\ell|_{l^0}^{-d}\big\}.
\]
This more stringent condition is sufficient for our purposes and considerably simplifies several
proofs. In particular, $C, \overline\sigma > 0$ sufficiently large and $\underline\sigma>0$ sufficiently small are fixed throughout this discussion and all constants in the following are allowed to depend on them.

As already hinted to above, a key idea is to split the difference of the Hamiltonians $\Hdef (u)-\Hhom$ into a sum of a large finite rank operator representing the defect core and a small but infinite rank part representing the far field. Let
\begin{equation} \label{eq:res:defn of HM}
   \< \HM(u) v, z\>
   = \sum_{\lvert \ell \rvert \leq M} \nabla^2V(0)[Dv(\ell), Dz(\ell)]
      +\sum_{\lvert \ell \rvert > M} \nabla^2V(Du(\ell))[Dv(\ell), Dz(\ell)],
\end{equation}
and define $\HMt(u) := (1-t) \Hhom  + t \HM(u)$ and $\RMt_z$ analogously. Then,
\begin{equation} \label{eq:RMt-Rhom}
   \begin{split}
   \RMt_z(u)-\Rhom_z
   &= \Big( (z-1)I - t \Fopd \big(\HM(u)-\Hhom\big) \Fop \Big)^{-1} - (z-1)^{-1} I
   \\&= (z-1)^{-1} \bigg\{
      \Big(I- \smfrac{t}{z-1} \Fopd \big( \HM(u)- \Hhom\big) \Fop \Big)^{-1}
   - I \bigg\}.
   \end{split}
\end{equation}
We now show that $A_M := \Fopd(\HM(u)-\Hhom )\Fop$ is small provided that
$M$ is sufficiently large. Starting with this lemma, we will heavily rely on the convenient notation $\lvert \ell \lvert^{-d}_{l^k}$ defined in \eqref{eq:defn_lognorms} for a decay rate up to $k$ logarithmic factors, as well as the notation $\Lr_{k}$ and $\Lr_{k}^M$  defined in \eqref{eq:defLr} and \eqref{eq:defLrM} for operator estimates.

\begin{lemma}
\label{lem: aux perturbation}
There exist $C_0, C_1, C_2 >0$ independent of $m$, $n$, $M$, and $u \in
\mathcal{U}$ such that
\begin{align}
\label{eq: small pertubation}
   \lvert (A_M)_{mn} \rvert &\leq C_0 \Lr_{1}^M(m,n), \qquad \text{where} \\[1mm]
\label{eq: auxfrobenius}
\sum\limits_{m,n \in \Lambda} \Lr_{1}^M(m,n)^2 &\leq C_1  \lvert M\rvert_{l^3}^{-d},
\qquad \text{and} \\
\label{eq: auxprod}
\sum\limits_{\ell \in \Lambda} \Lr_{1}^M(m,\ell) \Lr_{1}^M(\ell, n) &\leq C_2  \lvert M\rvert_{l^3}^{-d} \Lr_{1}^M(m,n).
\end{align}
\end{lemma}
\begin{proof}
Many detailed sums we need here and in the following are collected in the appendix in Section \ref{sec: auxiliary estimates}. Specifically, \eqref{eq: auxiliary estimates eq6cor2} yields \eqref{eq: small pertubation}:
\begin{align*}
\lvert (A_M)_{mn} \rvert &\leq \sum_{\lvert \ell \rvert > M} \lvert \nabla^2 V(Du(\ell)) - \nabla^2V(0)\rvert \lvert DF(\ell -m) \rvert \lvert DF(\ell-n)\rvert \\
&\lesssim \sum_{\lvert \ell \rvert > M} \lvert \ell \rvert_{l^0}^{-d} \lvert \ell -m \rvert_{l^0}^{-d} \lvert \ell-n \rvert_{l^0}^{-d}\\
&\lesssim \Lr_{1}^M(m,n).
\end{align*}

For \eqref{eq: auxprod} we estimate
\begin{align}
   \notag
\sum_{\ell \in \Lambda} \Lr_{1}^M(m,\ell) \Lr_{1}^M(\ell,n) &\lesssim \lvert n \rvert_{l^1, M}^{-d} \lvert m \rvert_{l^1, M}^{-d} \sum_{\ell \in \Lambda} (\lvert \ell \rvert_{l^1, M}^{-d} + \lvert m-\ell \rvert_{l^1}^{-d})(\lvert \ell \rvert_{l^1, M}^{-d} + \lvert n-\ell \rvert_{l^1}^{-d})  \\
\notag
&\quad + \lvert n \rvert_{l^1, M}^{-d} \sum_{\ell \in \Lambda} \lvert \ell \rvert_{l^1, M}^{-d} \lvert m-\ell \rvert_{l^1}^{-d} (\lvert \ell \rvert_{l^1, M}^{-d} + \lvert n-\ell \rvert_{l^1}^{-d}) \\
\notag
&\quad + \lvert m \rvert_{l^1, M}^{-d} \sum_{\ell \in \Lambda} \lvert \ell \rvert_{l^1, M}^{-d} \lvert n-\ell \rvert_{l^1}^{-d} (\lvert \ell \rvert_{l^1, M}^{-d} + \lvert m-\ell \rvert_{l^1}^{-d}) \\
   \label{eq:proof:aux perturbation:1}
&\quad +\sum_{\ell \in \Lambda} \lvert \ell \rvert_{l^2, M}^{-2d} \lvert m-\ell \rvert_{l^1}^{-d} \lvert n-\ell \rvert_{l^1}^{-d}.
\end{align}
We look at each of the sums in detail. According to \eqref{eq: auxiliary estimates eq1} and \eqref{eq: auxiliary estimates eq3} we have the estimates
\begin{align*}
\sum_{\ell \in \Lambda} \lvert \ell \rvert_{l^2, M}^{-2d} &\lesssim \lvert M \rvert_{l^2}^{-d} , \\
\sum_{\ell \in \Lambda} \lvert \ell \rvert_{l^1, M}^{-d} \lvert n-\ell \rvert_{l^1}^{-d} &\lesssim \lvert n \rvert_{l^3,M}^{-d} \lesssim \lvert M \rvert_{l^3}^{-d},
\\
\sum_{\ell \in \Lambda} \lvert m-\ell \rvert_{l^1}^{-d} \lvert \ell \rvert_{l^1,M}^{-d} &\lesssim \lvert m \rvert_{l^3,M}^{-d} \lesssim \lvert M \rvert_{l^3}^{-d},
\\
\sum_{\ell \in \Lambda} \lvert m-\ell \rvert_{l^1}^{-d} \lvert n-\ell \rvert_{l^1}^{-d} &\lesssim \lvert n-m \rvert_{l^3}^{-d},
\end{align*}
Furthermore, according to \eqref{eq: auxiliary estimates eq4} and \eqref{eq: auxiliary estimates eq6cor3}  we also have
\begin{align*}
\sum_{\ell \in \Lambda} \lvert \ell \rvert_{l^2, M}^{-2d} \lvert m-\ell \rvert_{l^1}^{-d} &\lesssim \lvert m\rvert_{l^1,M}^{-d} \lvert M\rvert_{l^3}^{-d} , \\
\sum_{\ell \in \Lambda} \lvert \ell \rvert_{l^1, M}^{-d} \lvert m-\ell \rvert_{l^1}^{-d} \lvert n-\ell \rvert_{l^1}^{-d} &\lesssim \lvert m\rvert_{l^1,M}^{-d} \lvert n\rvert_{l^3,M}^{-d} + \lvert m-n\rvert_{l^3}^{-d} \lvert m\rvert_{l^1,M}^{-d},
\end{align*}
and analogously for $m$ and $n$ reversed. At last, according to \eqref{eq: auxiliary estimates eq7cor1}, we also have
\begin{align*}
 \sum\limits_{\ell\in \Lambda}\lvert \ell \rvert_{l^2,M}^{-2d} \lvert \ell-n \rvert_{l^1}^{-d} \lvert \ell-m \rvert_{l^1}^{-d}\lesssim  \lvert n\rvert_{l^1,M}^{-d} \lvert m\rvert_{l^1,M}^{-d} ( \lvert M\rvert_{l^3}^{-d} + \lvert m-n\rvert_{l^3}^{-d}).
\end{align*}

Inserting these intermediate estimates into \eqref{eq:proof:aux perturbation:1}
we get \eqref{eq: auxprod}, as 
\begin{align}
\sum\limits_{\ell \in \Lambda} \Lr_{1}^M(m,\ell) \Lr_{1}^M(\ell, n) &\lesssim  \lvert n\rvert_{l^1,M}^{-d} \lvert m\rvert_{l^1,M}^{-d} ( \lvert M\rvert_{l^3}^{-d} + \lvert m-n\rvert_{l^3}^{-d}) \label{eq:sharperthansubmulti} \\
&\lesssim \lvert M\rvert_{l^3}^{-d} \Lr_{1}^M(m,n). \nonumber
\end{align}
Finally, summing over $m=n$ in \eqref{eq:sharperthansubmulti},
and using \eqref{eq: auxiliary estimates eq1}, we deduce
\eqref{eq: auxfrobenius}:
\begin{align*}
\sum\limits_{m \in \Lambda} \sum\limits_{\ell \in \Lambda} \Lr_{1}^M(m,\ell) \Lr_{1}^M(\ell, m) &\lesssim  \sum\limits_{m \in \Lambda} \lvert m\rvert_{l^2,M}^{-2d}
\lesssim \lvert M\rvert_{l^3}^{-d}. \qedhere
\end{align*}
\end{proof}

\begin{proposition}
   \label{prop: resolvent new version}
   There exists a constant $C_3 > 0$ such that, for all $u \in \mathcal{U}$, $t \in [0, 1]$, $z \in \mathcal{C}$,
   \begin{equation}
   \Big|\big[\Rr_z^t(u)-\Rhom_z \big]_{minj}\Big| \leq C_3 \Lr_1(m,n).
   \end{equation}
\end{proposition}
\begin{proof}
We split $\Rr_z^t - \Rhom_z = (\Rr_z^t - \RMt_z) + (\RMt_z - \Rhom_z)$. To
estimate the first group we will use that $H^t - \HMt$ has finite rank. To
estimate the second group we will use that $\HMt - \Hhom$ is small.

We begin by estimating $\RMt_z - \Rhom_z$.
Recall from \eqref{eq:RMt-Rhom} that
\begin{equation*}
   \RMt_z(u)-\Rhom_z
   = (z-1)^{-1} \Big( \Big(I-\smfrac{t}{z-1} A_M \Big)^{-1} - I \Big),
\end{equation*}
with $A_M = \Fopd (\HM(u)-\Hhom(0) ) \Fop$.

We can show that for $M$ sufficiently large the associated Neumann series
converges, from which we can deduce not only that $(I-\frac{t}{z-1}A_M)^{-1}$
(and hence also $\RMt_z(u)$) is well-defined but also obtain decay estimates.
Indeed, we can bound the Frobenius norm by
\begin{align*}
\lVert A_M \rVert_F^2 &\leq C_0^2 d^2 \sum_{m,n \in \Lambda} \Lr_{1}^M(m,n)^2
\leq C_0^2 d^2 C_1 \lvert M \rvert^{-d}_{l^{3}},
\end{align*}
according to Lemma \ref{lem: aux perturbation}.
As the Frobenius norm is sub-multiplicative, for $M$ sufficiently large, the Neumann series
\[
   \Big(I-\smfrac{t}{z-1}A_M \Big)^{-1}
   = \sum_{k=0}^{\infty} \frac{t^k}{(z-1)^k}A_M^k
\]
converges strongly in the Frobenius norm, uniformly in $z \in \mathcal{C}, t \in [0, 1]$
and $u \in \mathcal{U}_\epsilon$.
%
%
From Lemma~\ref{lem: aux perturbation} we can moreover deduce that
\[\lvert (A_M^k)_{mn} \rvert \leq (dC_2 \lvert M\rvert_{l^3}^{-d})^{k-1} C_0^k \Lr_{1}^M(m,n), \]
and hence
\begin{align*}
\bigg| \sum_{k=0}^\infty \frac{t^k (A_M^k)_{mn}}{(z-1)^k}  - I_{mn} \bigg| &\leq \Lr_{1}^M(n,m) \sum_{k=1}^\infty \frac{C_0^k C_2^{k-1} d^{k-1} (\lvert M\rvert_{l^3}^{-d})^{k-1}}{\lvert z-1 \rvert^k}.
\end{align*}
For $M$ large enough the series on the right-hand side converges
uniformly in $z$, and therefore
\begin{equation} \label{eq:decay-estimate-RMt}
   \Big|\big[\RMt_z(u)-\Rhom_z\big]_{mn}\Big|
   \lesssim \Lr_{1}^M(m,n).
\end{equation}

It remains to estimate $\Rr^t_z - \RMt_z$. We begin by rewriting
\begin{align*}
\Rr_z^t(u) = \Big(I + \RMt_z(u) \Fopd \big(\HMt(u)-\Hdef^t(u) \big) \Fop \Big)^{-1}\RMt_z(u)
   =: \big( I + B^{M,t} \big)^{-1} \RMt_z(u).
\end{align*}
Lemma~\ref{th:spectrum_bound} implies that the resolvent $\Rr_z^t(u)$ exists for
all $z \in \mathcal{C}, t \in [0, 1], u \in \mathcal{U}$ and hence the inverse on the
right hand side exists as well.

Moreover, $B^{M,t}$ has finite-dimensional range since
clearly this is the case for $\HMt(u)-\Hdef^t(u)$.
More precisely, if we set $X_2 = \{ u \in \ell^2 \colon D(\Fop u)(\ell)=0
\text{ for all } \lvert \ell \rvert \leq M\}$, then
$X_2 \subset {\rm ker}(B^{M,t})$, while $X_1 := X_2^{\perp}$ is finite dimensional. According to Lemma \ref{lem: finite-rank correction} it follows that
\begin{equation} \label{eq:finite-rank-expansion-BMt}
   \big(I + B^{M,t} \big)^{-1} =
   I + \sum_{j=1}^{r+1}  \big(B^{M,t}\big)^{j} d_j,
\end{equation}
with $d_j$ depending continuously on the projected and restricted
operators $\pi_{X_1} B^{M,t}|_{X_1}$. In particular, these constants remain
uniformly bounded in $z, t$ and $u \in \mathcal{U}$.
Therefore, we only have to estimate
\[
   \big(B^{M,t}\big)^j \RMt_z
   = \Big(\RMt_z \Fopd (\HMt - H^t)\Fop\Big)^{j} \RMt_z
\]
for $1 \leq j \leq r+1$. To that end, we note that similarly as in
Lemma~\ref{lem: aux perturbation}
\begin{align} \label{eq:estimateforlargefiniterankpart1}
\lvert (\Fopd (H^{t,M}-H^t) \Fop)_{mn} \rvert
&\lesssim  t \sum_{\lvert \ell \rvert \leq M} \lvert \ell \rvert_{l^0}^{-d} \lvert \ell -m \rvert_{l^0}^{-d} \lvert \ell-n \rvert_{l^0}^{-d} \nonumber \\
&\leq   \sum_{\ell} \lvert \ell \rvert_{l^0}^{-d} \lvert \ell -m \rvert_{l^0}^{-d} \lvert \ell-n \rvert_{l^0}^{-d} \nonumber \\
&\lesssim  \Lr_{1}(m,n),
\end{align}
based on \eqref{eq: auxiliary estimates eq6cor2}. Recall also from \eqref{eq:decay-estimate-RMt} that
\begin{equation} \label{eq:estimateforlargefiniterankpart2}
   \lvert (\RMt_z(u)- \Rhom_z)_{mn} \rvert
   \lesssim \Lr_{1}^M (m,n)  \leq \Lr_{1}(m,n),
\end{equation}
hence we can now use \eqref{eq: auxprod} with $M=0$ to deduce
\begin{equation*}
\Big| \big[ \big(\RMt_z(u) \Fopd (H^{t,M}-H^t) \Fop \big)^{j} \RMt_z(u)
      \big]_{mn} \Big| \lesssim \Lr_{1}(m,n),
\end{equation*}
where the implied constant is independent of $z \in \mathcal{C}, t \in [0, 1]$
and $u \in \mathcal{U}$. Combined with
\eqref{eq:finite-rank-expansion-BMt} this completes the proof.
\end{proof}

%% file: locality.tex

\section{Locality Estimates for $S$}
\label{sec: locality}
In the following, will use the definition
\begin{equation} \label{eq:defSell2}
   \mathcal{S}^+_\ell(u) := -\smfrac{1}{2}{\rm Trace} \Big[\log^+ \big(\Fopd \Hdef(u) \Fop\big)\Big]_{\ell\ell},
\end{equation}
for $u$ satisfying \eqref{eq:logpluscondition}. Of course, we have $\mathcal{S}^+_\ell(u) =\mathcal{S}_\ell(u)$ included as a special case if  \eqref{eq:logcondition} is true.

Let us start with the regularity claim.
\begin{proof}[Proof of Theorem \ref{theo: main result} (1)]
According to Lemma \ref{lem:energyregular}, $\Hdef(u), \Hhom(u) : \Wi \to \mathcal{L}(\Wi, (\Wi)')$ are $(p-2)$-times continuously Fréchet differentiable. Due to Lemma \ref{th:properties_F} and Lemma \ref{th:spectrum_bound}, $\Rhom_z(0)$ and $\Rdef_z(\us)$ exist. As the set of invertible linear operators is open and inverting is smooth, we find that $\Rhom_z(u), \Rdef_z(\us+u): B_\delta(0) \subset \Wi \to \mathcal{L}(\ell^2, \ell^2)$ to be $(p-2)$-times continuously Fréchet differentiable for $\delta>0$ small enough. All of this can be done uniformly in $z \in \mathcal{C}$. Therefore, the same regularity holds true for $\log\big(\Fopd \Hdef(u) \Fop\big)$, $\log\big(\Fopd \Hhom(u) \Fop\big)$, and any of their components.
\end{proof}

Based on $\Hdef^t(u) = (1-t) \Hhom(0) + t \Hdef(u)$ for $u \in \mathcal{U}$, define
\[\SS_\ell^t(u) := -\smfrac{1}{2} {\rm Trace} \Big[\logp \big(\Fopd \Hdef^t(u) \Fop\big)\Big]_{\ell\ell},\]
then
\[\SS^+_\ell(u) = \SS_\ell^0(u) + \frac{\partial \SS_\ell^0(u) }{\partial t} + \int_0^1 (1-t) \frac{\partial^2 \SS_\ell^t(u) }{\partial^2 t}\,dt.\]

Indeed, we directly check that $t \mapsto \SS_\ell^t$ is twice differentiable with
\begin{equation*}
\frac{\partial \SS_\ell^t(u) }{\partial t} = -\frac{1}{2}  \frac{1}{2\pi i} \oint_{\mathcal{C}} \log z\, {\rm Trace} \big( \Rr_z^t\Fopd (\Hdef(u)-\Hhom(0)) \Fop\Rr_z^t \big)_{\ell \ell}\,dz
\end{equation*}
and
\begin{equation*}
\frac{\partial^2 \SS_\ell^t(u) }{\partial t^2} =  -\frac{1}{2} 2\frac{1}{2\pi i} \oint_{\mathcal{C}} \log z\, {\rm Trace} \big( \Rr_z^t\Fopd (\Hdef(u)-\Hhom(0)) \Fop\Rr_z^t\Fopd (\Hdef(u)-\Hhom(0)) \Fop\Rr_z^t \big)_{\ell \ell}\,dz.
\end{equation*}
In particular, we find $\SS_\ell^0(u) = 0$ and
\begin{align*}
\frac{\partial \SS_\ell^0(u) }{\partial t} &= -\frac{1}{2} \frac{1}{2\pi i} \oint_{\mathcal{C}} \frac{\log z}{(z-1)^2} {\rm Trace} \big( \Fopd (\Hdef(u)-\Hhom(0)) \Fop \big)_{\ell \ell}\,dz \\
&= -\frac{1}{2} {\rm Trace} \big( \Fopd (\Hdef(u)-\Hhom(0)) \Fop \big)_{\ell \ell}.
\end{align*}
We can then write
\begin{align*}
\frac{\partial \SS_\ell^0(u) }{\partial t} &= -\frac{1}{2} {\rm Trace} \big( \Fopd \langle \delta \Hhom(0),u\rangle \Fop \big)_{\ell \ell}\\
&\quad  -\frac{1}{2}\int_0^1 (1-s) {\rm Trace} \big( \Fopd \langle \delta^2 \Hhom(su)u,u\rangle \Fop \big)_{\ell \ell}  \,ds\\
&\quad -\frac{1}{2}{\rm Trace} \big( \Fopd \big( \Hdef(u)-\Hhom(u)\big) \Fop \big)_{\ell \ell}.
\end{align*}
Also note that
\begin{align*}
\langle \delta \Shom_\ell(0),u \rangle &= -\frac{1}{2} \frac{1}{2\pi i} \oint_{\mathcal{C}} \log z\, {\rm Trace} \big( \Rhom_z \Fopd \langle \delta \Hhom(0),u\rangle \Fop \Rhom_z \big)_{\ell \ell}\,dz\\
&=-\frac{1}{2} {\rm Trace} \big( \Fopd \langle \delta \Hhom(0),u\rangle \Fop \big)_{\ell \ell}.
\end{align*}
Overall, we have decomposed $\SS^+_\ell(u)$ into
\begin{equation}
      \label{eq:decompose S+}
\SS^+_\ell(u) = \langle \delta \Shom_\ell(0),u \rangle+ \SS_{\ell,1}(u)+ \SS_{\ell,2}(u)+\SS_{\ell,3}(u),
\end{equation}
with
\begin{align*}
\SS_{\ell,1}(u) &:=  -\frac{1}{2}\int_0^1 (1-s) {\rm Trace} \big( \Fopd \langle \delta^2 \Hhom(su)u,u\rangle \Fop \big)_{\ell \ell}  \,ds,\\
\SS_{\ell,2}(u) &:=-\frac{1}{2} {\rm Trace} \big( \Fopd \big( \Hdef(u)-\Hhom(u)\big) \Fop \big)_{\ell \ell},\\
\SS_{\ell,3}(u) &:=-\frac{1}{2} \int_0^1 \oint_{\mathcal{C}} (1-t) \frac{2}{2\pi i}  \log z \\
 &\qquad {\rm Trace} \big( \Rr_z^t\Fopd (\Hdef(u)-\Hhom(0)) \Fop\Rr_z^t\Fopd (\Hdef(u)-\Hhom(0)) \Fop\Rr_z^t \big)_{\ell \ell}\,dz\,dt. \nonumber
\end{align*}

This decomposition will be useful in light of the properties we establish
next:

\begin{proposition} \label{prop: localitydecomposedentropy}
For $u \in \mathcal{U}$
\begin{equation}\label{eq: localitydecomposedentropy1}
\lvert \langle \delta \Shom_\ell(0),u \rangle \rvert \lesssim \lvert \ell \rvert_{l^0}^{-d},
\end{equation}
but
\begin{equation}\label{eq: localitydecomposedentropy2}
\lvert \SS_{\ell,i}(u) \rvert \lesssim \lvert \ell \rvert_{l^0}^{-2d},\quad \text{for } i \in\{1,2\},
\end{equation}
and
\begin{equation}\label{eq: localitydecomposedentropy3}
\lvert \SS_{\ell,3}(u) \rvert \lesssim \lvert \ell \rvert_{l^2}^{-2d}.
\end{equation}

In particular, the sum
\[\SS^+(u):=\sum_\ell \Big(\SS^+_\ell(u) - \langle \delta \Shom_\ell(0),u \rangle \Big)\]
converges absolutely.
\end{proposition}

First, let us look more closely at the variations of $\Hhom$. Remember that $\Hhom(u) = \delta^2 \Ehom(u)$. We can write its components as
\[\Hhom(u)_{minj} := (\Hhom(u)_{mn})_{ij} = \sum_{\xi \in \Lambda} \nabla^2 V (Du(\xi))[D(\delta_m e_i)(\xi), D(\delta_n e_j)(\xi)].\]
Accordingly, the first variation is
\[ \big[\langle \delta \Hhom(u), v \rangle \big]_{minj}= \sum_{\xi \in \Lambda} \nabla^3 V (Du(\xi))[D(\delta_m e_i)(\xi), D(\delta_n e_j)(\xi), Dv(\xi)].\]

Similarly, for the second variation of $\Hdef$ we will use the notation
\[\big[ \langle\delta^2 \Hhom (u)v,w \rangle \big]_{minj} = \sum_{\xi \in \Lambda}  \nabla^4 V(Du(\xi))[D(\delta_m e_i)(\xi), D(\delta_n e_j)(\xi),Dv(\xi),Dw(\xi)].\]

\begin{lemma} For all $t \in [0,1]$ uniformly, it holds that
\label{lem: F H F}
\begin{align}
\Big\lvert \big[ \Fopd \langle \delta \Hhom(tu),v\rangle \Fop \big]_{mn} \Big\rvert &\lesssim \sum_{\xi \in \Lambda} |\xi-m|_{l^0}^{-d}\, |\xi-n|_{l^0}^{-d}\, \lvert Dv(\xi) \rvert,\label{eq: F dHhom F}
\\ \Big\lvert \big[ \Fopd \langle \delta^2 \Hhom(tu)v,w\rangle \Fop \big]_{mn} \Big\rvert &\lesssim \sum_{\xi \in \Lambda} |\xi-m|_{l^0}^{-d}\, |\xi-n|_{l^0}^{-d}\,  \lvert Dv(\xi) \rvert \lvert\,  Dw(\xi) \rvert. \label{eq: FddHhomF}
\\ \Big\lvert \big[ \Fopd (\Hdef(u)-\Hhom(u)) \Fop \big]_{mn} \Big\rvert &\lesssim |m|_{l^0}^{-d}\, |n|_{l^0}^{-d}. \label{eq: FHdef-HhomF}
\end{align}
\end{lemma}
\begin{proof}
We have
\[\big[ \Fopd \delta \Hhom(tu)(\xi) \Fop \big]_{minj} = \nabla^3V(tDu(\xi))[DF_{\cdot i}(\xi-m),DF_{\cdot j}(\xi-n)], \]
and $\lvert DF(\xi) \rvert \lesssim \lvert \xi \rvert_{l^0}^{-d}$ according to Lemma~\ref{th:properties_F}. The same is true for the second variation with $\nabla^4 V$. As $V=V_\xi$ for $\lvert \xi \rvert \geq r_{\rm cut}$, we find
\begin{align*}
\Big\lvert \big[ \Fopd (\Hdef(u)-\Hhom(u)) \Fop \big]_{minj} \Big\rvert &\lesssim  \sum_{\lvert \xi \rvert < r_{\rm cut}} \Big(\nabla^2V_\xi(Du(\xi))-\nabla^2V(Du(\xi))\Big)[DF_{\cdot i}(\xi-m),DF_{\cdot j}(\xi-n)]\\
&\lesssim \sum_{\lvert \xi \rvert < r_{\rm cut}}|\xi-m|_{l^0}^{-d}\, |\xi-n|_{l^0}^{-d}\\
&\lesssim |m|_{l^0}^{-d}\, |n|_{l^0}^{-d}. \qedhere
\end{align*}
\end{proof}
We now have all the tools to prove Proposition \ref{prop: localitydecomposedentropy}.

\begin{proof}[Proof of Proposition \ref{prop: localitydecomposedentropy}]
Let us begin with the first order term. Using \eqref{eq: F dHhom F}, $u \in \mathcal{U}$,
and \eqref{eq: auxiliary estimates eq4} we find that
\begin{align*}
\lvert \langle \delta \Shom_\ell(0),u \rangle \rvert &\lesssim \big\lvert \sum_{\xi \in \Lambda} {\rm Trace} \big( \Fopd \delta \Hhom(0)(\xi)[Du(\xi)] \Fop \big)_{\ell \ell} \big\rvert\\
&\lesssim \sum_{\xi \in \Lambda} \lvert \xi - \ell \rvert_{l^0}^{-2d} \lvert Du(\xi) \rvert\\
&\lesssim \sum_{\xi \in \Lambda} \lvert \xi - \ell \rvert_{l^0}^{-2d} \lvert \xi \rvert_{l^0}^{-d}\\
&\lesssim \lvert \ell \rvert_{l^0}^{-d}.
\end{align*}
This proves \eqref{eq: localitydecomposedentropy1}. Equation \eqref{eq: localitydecomposedentropy2} for $\SS_{\ell,2}(u)$ is already included in \eqref{eq: FHdef-HhomF} in Lemma \ref{lem: F H F}.

To estimate $\SS_{\ell,1}(u)$ we can use \eqref{eq: FddHhomF}
and \eqref{eq: auxiliary estimates eq5} to see that
\begin{align*}
\Big\lvert \big( \Fopd \langle \delta^2 \Hhom(su)u,u\rangle \Fop \big)_{\ell \ell} \Big\lvert &\lesssim \sum_{\xi \in \Lambda} \lvert \xi - \ell \rvert_{l^0}^{-2d} \lvert Du(\xi) \rvert^2\\
&\lesssim \sum_{\xi \in \Lambda} \lvert \xi - \ell \rvert_{l^0}^{-2d} \lvert \xi \rvert_{l^0}^{-2d}\\
&\lesssim \lvert \ell \rvert_{l^0}^{-2d}.
\end{align*}

The last remaining claim \eqref{eq: localitydecomposedentropy3}
requires the resolvent estimates from \S~\ref{sec:resolventestimates}. We have
\begin{align*}
\SS_{\ell,3}(u) &= -\frac{1}{2} \int_0^1 (1-t) \frac{\partial^2 \SS_\ell^t(u) }{\partial^2 t}\,dt\\
&=-\frac{1}{2} \int_0^1 \oint_{\mathcal{C}} (1-t) \frac{2}{2\pi i}  \log z\, {\rm Trace} \big(  \\
& \hspace{2cm} \Rr_z^t\Fopd (\Hdef(u)-\Hhom(0)) \Fop\Rr_z^t\Fopd (\Hdef(u)-\Hhom(0)) \Fop\Rr_z^t \Big)_{\ell \ell}\,dz\,dt.
\end{align*}
We already know from Proposition \ref{prop: resolvent new version} that
\[\Big|\big(\Rr_z^t(u)-\Rhom_z \big)_{mn}\Big| \lesssim \Lr_1(m,n),  \]
with $\Rhom_z = (z-1)^{-1} I$. Furthermore, according to Lemma \ref{lem: F H F}
and \eqref{eq: auxiliary estimates eq6cor2}
\begin{align*}
\Big|\big(\Fopd (\Hdef(u)-\Hhom(0)) \Fop\big)_{minj}\Big| &\lesssim \lvert m \rvert_{l^0}^{-d}\lvert n \rvert_{l^0}^{-d} + \Big|\big(\Fopd (\Hhom(u)-\Hhom(0)) \Fop\big)_{minj}\Big| \\
&\lesssim \lvert m \rvert_{l^0}^{-d}\lvert n \rvert_{l^0}^{-d} + \sum_{\xi \in \Lambda} \lvert m-\xi \rvert_{l^0}^{-d}\lvert n-\xi \rvert_{l^0}^{-d} \lvert Du(\xi)\rvert \\
&\lesssim \lvert m \rvert_{l^0}^{-d}\lvert n \rvert_{l^0}^{-d} + \sum_{\xi \in \Lambda} \lvert m-\xi \rvert_{l^0}^{-d}\lvert n-\xi \rvert_{l^0}^{-d} \lvert \xi\rvert_{l^0}^{-d} \\
&\lesssim \Lr_1(m,n).
\end{align*}

$\Lr_1$ is submultiplicative up to a constant, in the sense that
(see Lemma \ref{lem: aux perturbation})
\[\sum_m \Lr_1(\ell,m) \Lr_1(m,n) \lesssim \Lr_1(\ell,n). \]
 As also
\[\sum_m \Rhom_z(\ell,m) \Lr_1(m,n) \lesssim \Lr_1(\ell,n), \]
we can apply the submultiplicativity several times, and use \eqref{eq:sharperthansubmulti}, to find
\begin{align*}
\lvert \SS_{\ell,3}(u) \rvert &\lesssim \sum_{m \in \Lambda} \Lr_1(\ell,m) \Lr_1(m,\ell)
\lesssim \lvert \ell \rvert^{-2d}_{l^{2}}. \qedhere
\end{align*}
\end{proof}

\begin{proof}[Proof of Theorem \ref{theo: main result} (2)]
As $\bar{u} \in \mathcal{U}$ with $S^+_\ell(\bar{u})=S_\ell(\bar{u})$ for all $\ell$, we can directly apply Proposition \ref{prop: localitydecomposedentropy} to see that
\[\sum_\ell \Big\lvert \SS_\ell(\bar{u}) - \langle \delta \Shom_\ell(0),\bar{u} \rangle \Big\rvert < \infty.  \qedhere\]
\end{proof}

%% file: limit.tex

\section{Periodic Cell Problem and Thermodynamic Limit}
\label{sec:limit}
\subsection{Discrete Fourier Transform}
\label{sec:limit:prelims}
Recall that $\L= \mA \Z^d$, $\L_N= \mB (-N,N]^d \cap \L$, and that $\mA, \mB$ are non-singular with $\mA^{-1}\mB \in \Z^{d \times d}$. We will extend that notation and later also write $\L_t= \mB (- t ,t ]^d\cap \L$ for $t \in \R$, $t >0$, to conveniently discuss smaller and larger sections of the lattice. Based on the periodic cell, we will also use the short notation
\begin{align*}
\lvert \ell \rvert^{-\alpha}_{l^k, \L_N} &:= \max_{z \in \Z^d} \lvert \ell + 2N \mB z \rvert^{-\alpha}_{l^k}\\
\Lr_{k, \L_N}(n,m) &:= |n|_{l^k, \L_N}^{-d}|n-m|_{l^k, \L_N}^{-d}+|m|_{l^k, \L_N}^{-d}|n-m|_{l^k, \L_N}^{-d}+|n|_{l^k, \L_N}^{-d}|m|_{l^k, \L_N}^{-d},
\end{align*}
for estimates respecting the periodicity of the supercell approximation.

We wish to define a Fourier transform of functions $u : \L_N \to \R^m$. To that
end we characterize the dual group of $\L_N$. We expect that the following lemma
is known; indeed, special cases such as cubic domains for fcc or bcc crystals
are commonly used for FFT implementations~\cite{CsD:VMV:2008}. Lacking a clear
source for the general case $\mA \neq \mB$, we included a proof nonetheless.
\begin{lemma} \label{lem:dualgroup}
All the characters on $\L_N$ (i.e., the group homomorphisms $(\L_N, +) \to (\CN \setminus \{0\}, \cdot)$) are given by
\[\chi_k(\ell) = e^{i \ell k},\quad k \in \frac{\pi}{N} \mB^{-T}  \Z^d \cap \pi \mA^{-T} (-1,1]^d =: \B_N.\]
\end{lemma}
\begin{proof}
First we show that the characters on $G= \mB (0,1]^d \cap \L$ are precisely given by $\chi_k$ with $k \in 2 \pi \mB^{-T} \Z^d \cap \pi \mA^{-T} (-1,1]^d =: \hat{G}$. Indeed, as $e^{i k\mB e_j} = 1$ for $k \in 2 \pi \mB^{-T} \Z^d$ and all $j$, the $\chi_k$ with $k \in 2 \pi \mB^{-T} \Z^d$ are all characters on $G$. Furthermore, $\chi_k = \chi_{k'}$ if and only if $e^{i (k-k')\ell}=1$ for all $\ell \in G$. Since $k, k' \in 2 \pi \mB^{-T} \Z^d$, this is equivalent to $e^{i (k-k')\ell}=1$ for all $\ell \in \mA \Z^d$. This is true if and only if $k - k' \in 2 \pi \mA^{-T} \Z^d$. In particular, all the $\chi_k$ with $k \in \hat{G}$ are different characters. As also $\lvert G \rvert = \lvert \hat{G} \rvert$, these are already all characters.

Choosing $\mB'= 2N\mB$ and shifting $G$ by multiples of $\mB$ gives the desired result.
\end{proof}
\begin{corollary}
With $\B_N$ defined in Lemma~\ref{lem:dualgroup} we have
\begin{align}
\sum_{\ell \in \L_N} e^{i \ell (k-k')} &= \delta_{k k'} \lvert \L_N \rvert \quad \forall k, k' \in \B_N, \qquad \text{and}  \label{eq:fouriersum1} \\
\sum_{k \in \B_N} e^{i (\ell-\ell')k} &= \delta_{\ell \ell'} \lvert \B_N \rvert \quad \forall \ell, \ell' \in \Lambda_N \label{eq:fouriersum2}.
\end{align}
\end{corollary}
\begin{proof}
Identity \eqref{eq:fouriersum1} follows directly from Lemma \ref{lem:dualgroup}, as the set of characters on any finite Abelian group $G$ forms an orthogonal basis of the functions $G \to \CN$, see \cite[Thm.\,3.2.2]{Luong2009}. Identity \eqref{eq:fouriersum2} follows for the same reason, using the Pontryagin duality theorem.
\end{proof}

We can now define the {\em discrete Fourier transform} by
\begin{align*}
   \hat{g}(k) &= \sum_{\ell \in \L_N} e^{i k \cdot \ell} g(\ell), \qquad \text{for } k \in \B_N.
\end{align*}
According to \eqref{eq:fouriersum2}, the inverse is given by
\begin{align*}
   g(\ell) &= \frac{1}{\lvert \B_N \rvert}\sum_{k \in \B_N} e^{-i k \cdot \ell} \hat{g}(k), \qquad \text{for } \ell \in \L_N,
\end{align*}
with $\lvert \B_N \rvert = \lvert \L_N \rvert=(2N)^d \lvert \det(\mA^{-1} \mB) \rvert$. Although we use the same notation as for the semi-discrete Fourier transform, it will always be clear from context which one is meant.

Given $f : \Lambda \to \R^m$, for which the SDFT $\hat{f}$ is
well-defined, we can obtain a $\L_N$-periodic ``projection'' $f_N : \L_N \to
\R^m$ via
\begin{equation}   \label{eq:defn periodic projection}
   f_N(\ell) := \frac{1}{\lvert \B_N \rvert} \sum_{k \in \B_N} e^{- i k \cdot \ell} \hat{f}(k).
\end{equation}

\begin{lemma}  \label{th:trapezoidal rule}
   Suppose that $f : \L \to \R^m$ with $|f(\ell)| \lesssim |\ell|_{l^0}^{-\alpha}$
   where $\alpha > d$ (in particular, $\hat{f} \in L^\infty(\B)$), then
   \begin{equation*}
      \| f - f_N \|_{\ell^\infty(\L_N)} \lesssim N^{-\alpha}.
   \end{equation*}
\end{lemma}
\begin{proof}
As $f$ is summable over $\L$, one can directly check the {\em Poisson summation formula}
\[f_N(\ell) = \sum_{z \in \Z^d} f(\ell + 2N \mB z).\]
Employing the decay $|f(\ell)| \lesssim (1+|\ell|)^{-\alpha}$,
\begin{align*}
   |f(\ell) - f_N(\ell)|
   &= \bg| \sum_{z \in \Z^d \setminus \{0\}} f(\ell + 2N\mB z) \bg| \\
      &\lesssim \sum_{\ell \in \L \setminus \{0\}} |\ell + 2N\mB z|^{-\alpha} \\
      &\lesssim N^{-\alpha} \sum_{\ell \in \L \setminus \{0\}}
            \Big| \smfrac{\mB^{-1} \ell}{N} + 2 z\Big|^{-\alpha} \\
      &\lesssim  N^{-\alpha},
\end{align*}
where the sum is finite due to $\alpha > d$ and the estimate is uniform due to $\lvert \smfrac{\mB^{-1}\ell}{N} \rvert \leq 1$.
\end{proof}

\subsection{Periodic projection of $F$}
\label{sec:limit:estimateFN}
Recall the definition of $F$ from \eqref{eq:defn_F_fourier} via its SDFT $\hat{F}(k) =
[\sum_{\rho\in\Rc'} 4\sin^2(\frac{1}{2}k\cdot\rho)A_\rho]^{-1/2}$. Note that $\hat{F}(0)$ is undefined, but this is only related to the constant part of $F_N$. Therefore, we slightly modify \eqref{eq:defn periodic
projection}, to define its periodic projection via
\begin{equation}   \label{eq:defn FN}
   F_{N}(\ell) := \frac{1}{\lvert \B_N \rvert} \sum_{k \in \B_N\setminus\{0\}}
      e^{- i k\cdot \ell} \hat F(k).
\end{equation}
$D^2 F_N$ is then the periodic projection of $D^2 F$ according
to definition \eqref{eq:defn periodic projection}.
\begin{lemma}   \label{th:error estimate FN}
   There exist constants $C_1, C_2$, independent of $N$ such that
   \begin{align*}
      \| DF - DF_N \|_{\ell^\infty(\L_N)} &\leq C_1 N^{-d}, \qquad
            \text{and in particular} \\
      |DF_N(\ell)| &\leq C_2 |\ell|_{l^0, \L_N}^{-d} \qquad \text{for } \ell \in \L.
   \end{align*}
\end{lemma}
\begin{proof}
   We cannot employ Lemma~\ref{th:trapezoidal rule} directly since $|DF(\ell)|
   \lesssim |\ell|_{l^0}^{-d}$ but no faster. Instead, we first estimate $D^2F - D^2
   F_N$.

   Let $\rho_1, \rho_2 \in \Rc$, $f(\ell) := D_{\rho_1} D_{\rho_2} F(\ell)$, and
   $f_N$ its periodic projection \eqref{eq:defn periodic projection}, then it is
   easy to see that in fact $f_N(\ell) = D_{\rho_1} D_{\rho_2} F_N(\ell)$.
   According to Lemma~\ref{th:properties_F}, $|f(\ell)| \lesssim |\ell|_{l^0}^{-1-d}$
   and hence Lemma~\ref{th:trapezoidal rule} yields
   $\| f - f_N \|_{\ell^\infty(\L_{N})} \lesssim N^{-1-d}$. Stated in terms of
   $D^2 F$ we have
   \begin{equation} \label{eq:proof:err est FN:D2F}
      \| D^2 F(\ell) - D^2 F_N(\ell) \|_{\ell^\infty(\L_{N})} \lesssim N^{-1-d}.
   \end{equation}

   To obtain the estimate for $DF - DF_N$ we first note that the following
   discrete Poincaré inequality is easy to establish: As for all $g : \Lambda_N \to \R^m$ we clearly have
   \[\lvert g(x) - g(y) \rvert \leq C N \lVert Dg \rVert_{\ell^\infty(\Lambda_N)}\quad \text{for all } x \in \Lambda_N,\ y \in \Lambda_{N}, \]
   it follows that
   \begin{equation} \label{eq:proof:err est FN:poincare}
      \| g - \<g\>_{\L_N} \|_{\ell^\infty(\L_N)} \leq C N \| Dg \|_{\ell^\infty(\L_N)},
   \end{equation}
   where $\<g\>_{\L_N} = \frac{1}{\lvert \L_N \rvert} \sum_{\ell \in \L_N} g(\ell)$.

   Fix $\rho \in \Rc$ and let $C_N :=  \< D_\rho F-D_\rho F_N \>_{\L_N}$, then combining
   \eqref{eq:proof:err est FN:D2F} and \eqref{eq:proof:err est FN:poincare}
   we obtain
   \begin{align*}
      \| D_\rho F - D_\rho F_N \|_{\ell^\infty(\L_N)}
      &\leq   \| D_\rho F - D_\rho F_N - C_N \|_{\ell^\infty(\L_N)}
               + | C_N | \\
      & \lesssim N \| D D_\rho F - DD_\rho F_N \|_{\ell^\infty(\L_N)}
               + | C_N |
    \lesssim N^{-d} + \big| C_N \big|.
   \end{align*}
   It thus remains to estimate $C_N$.

   Periodicity of $F_N$ implies that $\< D_\rho F_N \>_{\L_N} = 0$, hence,
   \[
      C_N = \frac{1}{\lvert \L_N \rvert} \sum_{\ell \in \L_N} D_\rho F(\ell).
   \]
   Using discrete summation by parts we see that
   \[
      |C_N| = \frac{1}{\lvert \L_N \rvert} \bigg| \sum_{\ell \in (\L_N + \rho) \setminus \L_N} F(\ell)
            - \sum_{\ell \in \L_N \setminus (\L_N + \rho)} F(\ell) \bigg|
      \lesssim N^{-d} N^{d-1} N^{1-d} = N^{-d} \qedhere
   \]
\end{proof}

\subsection{Spectral properties in the periodic setting}
\label{sec:limit:spectralpropertiesFN}
We can now make the definition of $\SS_{N,\ell}$ in \eqref{eq:def SNell}
rigorous by specifying $\Fop_N$ via $F_N$ and proving Lemma \ref{lem:FopN}. In analogy with \eqref{eq:defn_Fop} but with a different constant part, we define
\begin{equation}   \label{eq:defn FopN}
   (\Fop_N f)(\ell) := \sum_{n \in \L_N} F_N(\ell-n) f(n).
\end{equation}
\begin{proof}[Proof of Lemma \ref{lem:FopN}]
Since $\sum_{\ell \in \L_N} F_N(\ell)=0$ we directly see that $\pi_N \Fop_N =0$ and $\Fop_N \pi_N =0$, that is \eqref{eq:FopN3}. As
\[( \Fop_N f, g )_{\ell^2(\L_N)}=  \sum_{n,\ell \in \L_N} F_N(\ell-n) f(n)g(\ell)\]
and $F_N(\ell)=F_N(-\ell)$, $\Fop_N$ is self-adjoint, establishing \eqref{eq:FopN1}. For $k \in \B_N \backslash \{0\}$ we have $\widehat{\Fop_N f}(k) =  \hat{F}(k) \hat{f}(k)$, while $\widehat{\Fop_N f}(0)=0$. Hence,
\begin{align*}
( \Fop_N \Hhom_N \Fop_N f, g )_{\ell^2(\L_N)} &= ( \Hhom_N \Fop_N f, \Fop_N g )_{\ell^2(\L_N)}\\
&= \sum_{\ell \in \L_N} \nabla^2V(0)[D(\Fop_N f)(\ell),D(\Fop_N g)(\ell)]\\
&= \frac{1}{\lvert \B_N \rvert} \sum_{k \in \B_N \backslash\{0\}} (\hat{F}(k)\hat{f}(k))^* \hat{h}(k) \hat{F}(k)\hat{g}(k)\\
&= \frac{1}{\lvert \B_N \rvert} \sum_{k \in \B_N \backslash\{0\}} \hat{f}(k)^*\hat{g}(k)\\
&= (f, g )_{\ell^2(\L_N)} - \frac{1}{\lvert \B_N \rvert} \hat{f}(0)^*\hat{g}(0)\\
&= ((I-\pi_N)f, g )_{\ell^2(\L_N)}.
\end{align*}
This shows \eqref{eq:FopN2} and completes the proof.
\end{proof}

In particular, if $\pi_N v =0$, then
\begin{equation} \label{eq:FNnormequivalence}
c \lVert D \Fop_N v \rVert_{\ell^2}^2 \leq (\Hhom_N \Fop_N v, \Fop_Nv )_{\ell^2(\L_N)} = \lVert v \rVert_{\ell^2}^2 \leq c' \lVert D \Fop_N v \rVert_{\ell^2}^2,
\end{equation}
based on Lemma \ref{thm:pbc stab}. Furthermore, if $\us_N$ is the solution from Theorem \ref{thm:pbc convergence}, then we can combine Lemma \ref{thm:pbc stab} with \eqref{eq:FNnormequivalence} to see that
\begin{equation*}
2\underline\sigma \lVert v \rVert_{\ell^2}^2 \leq (\Fop_N\Hdef_N \Fop_N v, v )_{\ell^2(\L_N)} \leq \frac{\overline\sigma}{2} \lVert v \rVert_{\ell^2}^2,
\end{equation*}
for some $\underline\sigma, \overline\sigma >0$ and any $v$ with $\pi_N v =0$. Therefore
\[\sigma(\Fop_N\Hdef_N(\us_N)\Fop_N + \pi_N) \subset [2\underline\sigma, \smfrac{1}{2}\overline\sigma]. \]
A perturbation argument as in Lemma \ref{th:spectrum_bound} then shows that
\begin{equation}
\sigma (\Fop_N\Hdef_N(u)\Fop_N + \pi_N) \subset [\underline\sigma, \overline\sigma],  \label{eq:FNHFNspectrum}
\end{equation}
for all $u$ with $\lVert Du - D \us_N \rVert_{\ell^2(\Lambda_N)} \leq \epsilon$.
Based on $\Fop_N \pi_N = \pi_N \Fop_N =0$, we have the resolvent identity \[\big( z-\big(\Fop_N\Hdef_N(u)\Fop_N + \pi_N\big)\big)^{-1} - \big( z-\Fop_N\Hdef_N(u)\Fop_N\big)^{-1} = (z-1)^{-1} \pi_N z^{-1}  ,\] which implies
\begin{align}
   \nonumber
\mathcal{S}_{N,\ell}(u) &= -\frac{1}{2} {\rm Trace} \Big[\log \big( \Fop_N\Hdef_N(u)\Fop_N + \pi_N\big)\Big]_{\ell\ell}\\
\nonumber
   &= -\frac{1}{2}\frac{1}{2\pi i}{\rm Trace}  \oint_{\mathcal{C}} \log z \Big[\Big( z-\big(\Fop_N\Hdef_N(u)\Fop_N + \pi_N\big)\Big)^{-1}\Big]_{\ell\ell} \, dz \\
   \nonumber
   &= -\frac{1}{2}\frac{1}{2\pi i}{\rm Trace}  \oint_{\mathcal{C}} \log z \Big[\big( z-\Fop_N\Hdef_N(u)\Fop_N\big)^{-1}+ (z-1)^{-1}\pi_N z^{-1}\Big]_{\ell\ell} \, dz\\
   \nonumber
   &= -\frac{1}{2}\frac{1}{2\pi i}{\rm Trace}  \oint_{\mathcal{C}} \log z \Big[\big( z-\Fop_N\Hdef_N(u)\Fop_N\big)^{-1}\Big]_{\ell\ell} \, dz\\
   &= -\frac{1}{2}{\rm Trace} \logp\, (\Fop_N\Hdef_N(u)\Fop_N)_{\ell \ell},
   \label{eq:calculation above}
\end{align}
as $\log(1)=0$.

For the sake of generality, in the following, we will use the definition
\begin{equation} \label{eq:defSNell2}
   \mathcal{S}_{N,\ell}^+(u) := -\smfrac{1}{2}{\rm Trace} \Big[\logp \big(\Fop_N\Hdef_N(u)\Fop_N\big)\Big]_{\ell\ell},\quad  \mathcal{S}_{N}^+(u) := \sum_{\ell \in \L_N} \mathcal{S}_{N,\ell}^+(u),
\end{equation}
for $u$ satisfying
\begin{equation} \label{eq:logpNcondition}
\sigma\big( \Fop_N \Hdef_N(u) \Fop_N\big) \cap (0, \infty) \subset [ \underline\sigma, \overline\sigma].\end{equation}
Due to the calculation in \eqref{eq:calculation above}, $\mathcal{S}_{N,\ell}(u) = \mathcal{S}^+_{N,\ell}(u)$ is included as a special case for $u$ with $\lVert Du - D \us_N \rVert_{\ell^2} \leq \epsilon$. This generalization allows us to include saddle points in \S\,\ref{sec:saddle}. We also look at a more general sequence. Let us consider any $u_N \in \Wper_N$, $u_\infty \in \Wi$ with
\begin{equation} \label{eq:uNuinftyAssumptions}
\begin{split}
   \lvert D u_\infty (\ell) \rvert &\lesssim \lvert \ell \rvert^{-d}_{l^0},  \\
   \lVert Du_N - Du_\infty \rVert_{\ell^\infty(\Lambda_N)} &\lesssim N^{-d},  \\
   \sigma\big( \Fop_N \Hdef_N(u_N) \Fop_N + \pi_N \big) \cap (-\underline\sigma, \infty) &\subset [ 2\underline\sigma, \overline\sigma/2],\\
   \sigma\big( \Fopd \Hdef(u_\infty ) \Fop \big) \cap (-\underline\sigma, \infty) &\subset [ 2\underline\sigma, \overline\sigma/2].
\end{split}
\end{equation}
In particular, for any $u$ with $\lVert Du - D u_N \rVert_{\ell^2(\Lambda_N)} \leq \epsilon$, \eqref{eq:logpNcondition} is true and $\mathcal{S}_{N,\ell}^+(u)$ is defined according to \eqref{eq:defSNell2}. Similarly, for the limit we have $B_\epsilon(u_\infty) \subset \mathcal{U}$ according to Lemma \ref{th:spectrum_bound}.

\subsection{Resolvent estimates}
\label{sec:limit:resolvent estimates}
\def\vN{v_N}
\def\vNper{v^{\rm per}_N}
Before we can proceed with the convergence analysis for the entropies, we need
to establish decay estimates for the periodic resolvent operators, analogous to
Proposition~\ref{prop: resolvent new version}.

We first introduce a compactly supported $\vN \approx u_\infty$ that allows us to relate $u_\infty$ to the periodic case. To do that we use a previously developed cut-off operator $T_R$.
\begin{lemma}\cite[Lemma 3.2]{2018-uniform} \label{th:TR_estimates}
For all $R \geq R_0$, with some sufficiently large $R_0$, there exist cut-off operators $T_R$ such that for all $2 \leq q \leq \infty$, $u :\L_R \to \R^m$, we have $T_Ru : \L \to \R^m$ and
\begin{align}
    \label{eq:TR_estimates:global}
\lVert D T_R u \rVert_{\ell^q} &\leq C \lVert D u \rVert_{\ell^q(\Lambda_R)},\\
\label{eq:TR_estimates:err1}
\lVert D T_R u -Du \rVert_{\ell^q(\Lambda_R)}  &\leq C \lVert D u \rVert_{\ell^q(\Lambda_R \setminus \Lambda_{R/2})}.
\end{align}
Furthermore, $D T_R u(\ell)=0$ for $\lvert \ell \rvert \geq R$ and $D T_R u(\ell)= Du(\ell) $ for $\lvert \ell \rvert \leq R/2$.
\end{lemma}
Crucially, for $R \leq N$, $T_Ru$ can also be interpreted as a periodic function. We can then define $\vN : \L \to \R^m$ by $v_N := T_{N/2} u_\infty$ to find
\begin{subequations}
\begin{align}
   \label{eq:limit:vN:support}
   & {\rm supp}(D\vN) \subset \L_{N/2},  \\
   \label{eq:limit:vN:error}
   & \|D\vN - D u_\infty \|_{\ell^\infty} \lesssim N^{-d}, \\
   \label{eq:limit:vN:decay}
   & \vN \in \mathcal{U}\quad \text{  for  all sufficiently large } N.
\end{align}
\end{subequations}
Here we used Lemma \ref{th:TR_estimates} and Lemma \ref{th:spectrum_bound}. We can also interpret $\vN$ as a periodic function, in which case we rename it $\vNper \in \Wper_N$ for additional clarity. The uniform convergence rate in
\eqref{eq:uNuinftyAssumptions} and \eqref{eq:limit:vN:error} then imply
\begin{equation}
   \label{eq:limit:vNper:error}
   \| D\vNper - D u_N \|_{\ell^\infty(\L_N)} \lesssim N^{-d}.
\end{equation}

In particular, $\vNper$ satisfies \eqref{eq:logpNcondition} and we can use the definition \eqref{eq:defSNell2}.

\begin{lemma}  \label{th:resolvent estimate vNper}
   For $N$ sufficiently large, and $z \in \mathcal{C}$, the resolvent
   \[
      \Rr_{N,z}(\vNper) := \big(z I_{\ell^2(\Lambda_N)} - \Fop_N \Hdef_N(\vNper) \Fop_N\big)^{-1}
   \]
   is well-defined and
   \begin{equation}  \label{eq:limit:resolvent estimate vNper}
      \big| \big[ \Rr_{N,z}(\vNper) - \Rhom_{N,z} \big]_{\ell n} \big|
      \lesssim \Lr_{1, \L_N}(\ell, n),
   \end{equation}
   where
   \begin{equation*}
   \Rhom_{N,z}:= \big(z I_{\ell^2(\Lambda_N)} - \Fop_N \Hhom_N \Fop_N\big)^{-1} = (z-1)^{-1} I_{\ell^2(\Lambda_N)}.
   \end{equation*}
\end{lemma}
\begin{proof}
   In light of the estimates on $F_N$ that we established in Lemma~\ref{th:error
   estimate FN} this proof is analogous to the proof of Proposition~\ref{prop:
   resolvent new version} and is hence omitted.
\end{proof}

Treating $u_N$ as a perturbation to $\vNper$, we also obtain a decay
estimate on $\Rr_{N, z} (\vNper + s (u_N-\vNper))$.

\begin{lemma}  \label{th:resolvent estimate uNper}
   For $N$ sufficiently large and $u \in {\rm conv} \{ u_N, \vNper\}$ the
   resolvent $\Rr_{N, z}(u) := (z - \Fop_N H_N(u) \Fop_N)^{-1}$ is well-defined
   and
   \begin{align}
      \label{eq:limit:resolvent perturbation uNper}
      \big| \big[ \Rr_{N,z}(u) - \Rr_{N,z}(\vNper) \big]_{nm} \big|
      &\lesssim N^{-d} |n - m|_{l^5, \L_N}^{-d}, \qquad \text{and, in particular,} \\
      \label{eq:limit:resolvent estimate uNper}
      \big| \big[ \Rr_{N,z}(u) - \Rhom_{N,z} \big]_{nm} \big|
      &\lesssim \Lr_{1,\L_N}(n,m) + N^{-d} |n - m|_{l^5, \L_N}^{-d}.
   \end{align}
\end{lemma}
\begin{proof}

We write
\begin{equation*}
\Rr_{N,z}(u) = \Big[ I_{\ell^2(\Lambda_N)} + \Rr_{N,z}(\vNper) \Fop_N (H_N(\vNper)-H_N(u)) \Fop_N  \Big]^{-1} \Rr_{N,z}(\vNper).
\end{equation*}
The resolvent on the left is well-defined if and only if the inverse on the right exists, which is the case if
\[A_N:= \Rr_{N,z}(\vNper) \Fop_N (H_N(\vNper)-H_N(u)) \Fop_N \]
is sufficiently small in the Frobenius norm.

We first calculate
\begin{align*}
& \hspace{-1.5cm} \lvert (\Fop_N (H_N(\vNper)-H_N(u)) \Fop_N)_{ij} \rvert \\
&\lesssim  \sum_\ell (\nabla^2 V_{\ell}(D\vNper(\ell))-\nabla^2 V_{\ell}(Du(\ell)))[DF_N(\ell-i),DF_N(\ell-j)]\\
&\lesssim  \sum_\ell \lvert D\vNper(\ell)-Du(\ell) \rvert \lvert \ell-i \rvert_{l^0, \L_N}^{-d}  \lvert \ell-j \rvert_{l^0, \L_N}^{-d}\\
&\lesssim  N^{-d} \sum_\ell \lvert \ell-i \rvert_{l^0, \L_N}^{-d}  \lvert \ell-j \rvert_{l^0, \L_N}^{-d}.
\end{align*}
Therefore, using the estimate \eqref{eq: auxiliary estimates eq6cor4} and \eqref{eq: auxiliary estimates eq3},
\begin{align}
\lvert (A_N)_{ij} \rvert &\lesssim  \sum_{m, \ell} (\delta_{im} + \Lr_{1, \L_N}(i,m)) N^{-d} \lvert \ell-i \rvert_{l^0, \L_N}^{-d}  \lvert \ell-j \rvert_{l^0, \L_N}^{-d} \nonumber \\
&\lesssim N^{-d} \sum_{ \ell} \lvert \ell-i \rvert_{l^2, \L_N}^{-d}  \lvert \ell-j \rvert_{l^0, \L_N}^{-d} \nonumber\\
&\lesssim N^{-d} \lvert i-j \rvert_{l^3, \L_N}^{-d} \label{eq:limit:resolvent perturbation uNper aux2}.
\end{align}

We can thus estimate the Frobenius norm as
\begin{align*}
\lVert A_N \rVert^2_F &\lesssim N^{-2d}\sum_{i,j} \lvert i-j \rvert_{l^6, \L_N}^{-2d} \lesssim N^{-d}.
\end{align*}
In particular, for $N$ large enough, the resolvent $\Rr_{N,z}(u)$ exists and is given by the Neumann series
\begin{equation}
\Rr_{N,z}(u) - \Rr_{N,z}(\vNper) = \Big(\sum_{k=1}^{\infty} (-A_N)^k \Big) \Rr_{N,z}(\vNper).
\end{equation}
Let us now use the easier estimate \eqref{eq:limit:resolvent perturbation uNper aux2} to estimate products. We have
\begin{align*}
\sum_m N^{-d} \lvert m-i\rvert_{l^3, \L_N}^{-d} N^{-d} \lvert j-m\rvert_{l^3, \L_N}^{-d} & \lesssim N^{-2d} \lvert j-i\rvert_{l^7, \L_N}^{-d} \lesssim \lvert N\rvert_{l^4}^{-d} \lvert j-i\rvert_{l^3, \L_N}^{-d} N^{-d},
\end{align*}
again according to \eqref{eq: auxiliary estimates eq3}. Therefore,
\begin{align*}
\lvert (A_N^k)_{ij} \rvert &\leq  C^k (\lvert N\rvert_{l^4, \L_N}^{-d})^{k-1} \lvert j-i\rvert_{l^3, \L_N}^{-d} N^{-d}
\end{align*}
for some constant $C>0$. Hence, using \eqref{eq: auxiliary estimates eq6cor4},
\begin{align*}
\big| \big[ \Rr_{N,z}(u) - \Rhom_{N,z} \big]_{nm} \big| &\lesssim \sum_{k=1}^\infty \sum_{i \in \Lambda_N} C^k (\lvert N\rvert_{l^4, \L_N}^{-d})^{k-1} \lvert i-n\rvert_{l^3, \L_N}^{-d} N^{-d} (\delta_{im} + \Lr_{1,\L_N}(i,m))\\
&\lesssim \sum_{i \in \Lambda_N} \lvert i-n\rvert_{l^3, \L_N}^{-d} N^{-d} (\delta_{im} + \Lr_1(i,m))\\
&\lesssim \lvert m-n\rvert_{l^5, \L_N}^{-d} N^{-d}. \qedhere
\end{align*}
\end{proof}

\subsection{Entropy error estimates}
\label{sec:limit:site entropy errors}
Our aim is the proof of Theorem~\ref{theo: main
result}(3); that is, a convergence rate for $\SS_N(\us_N) - \SS(\us)$. For the sake of generality we prove the following more general statement.
\begin{proposition} \label{prop:generalconvergencerate}
For $u_N, u_\infty$ satisfying \eqref{eq:uNuinftyAssumptions},
\begin{equation}
\lvert \SS_N^+(u_N) - \SS^+(u_\infty) \rvert \lesssim \lvert N \rvert^{-d}_{l^5}.
\end{equation}
\end{proposition}

To prove this statement, we  split the entropy error into
\begin{align} \label{eq:limit:split of S error}
   \SS^+(u_\infty) - \SS^+_N(u_N)
   &= \big( \SS^+(u_\infty) - \SS^+(\vN) \big)
   + \big( \SS^+(\vN) - \SS^+_N(\vNper) \big)
   + \big( \SS^+_N(\vNper) - \SS^+_N(u_N) \big).
\end{align}

\subsubsection{The term $\SS^+(u_\infty) - \SS^+(\vN)$}
We investigate the term $\SS^+(u_\infty) - \SS^+(\vN)$ first. Substituting $w := \vN-u_\infty$
we rewrite this as
\begin{align}
   \notag
   \SS^+(\vN) - \SS^+(u_\infty)
   &=
   \sum_{\ell \in \L} \Big( \SS^+_\ell(\vN) - \SS^+_\ell(u_\infty)
                  - \big\< \delta \Shom_\ell(0), w \big\> \Big)
   \\ \notag &=
   \sum_{\ell \in \L}
      \big\< \delta \SS^+_\ell(u_\infty) - \delta \Shom_\ell(0), w \big\>
    + \sum_{\ell \in \L}
      \int_{0}^1 (1-s) \big\< \delta^2 \SS^+_\ell(u_\infty+sw) w, w \big\> \,ds
   \\ &= {\bf A}_N + {\bf B}_N.
   \label{eq:limit:AN+BN}
\end{align}

To estimate ${\bf A}_N$ we decompose it into
 ${\bf A}_N = \sum_\ell {\bf A}_{N,\ell}$ where
\begin{equation*}
   {\bf A}_{N,\ell} = \big\< \delta \SS^+_\ell(u_\infty) - \delta \Shom_\ell(0), w \big\>
      = -\frac{1}{2} \frac{1}{2\pi i} \oint_{\mathcal{C}}
      {\rm Trace} \big\< \delta [\Rr_z]_{\ell\ell} - \delta [\Rhom_z]_{\ell\ell}, w \big\> \, dz,
\end{equation*}
where we write $\Rr_z =\Rr_z(u_\infty)$ for simplicity. The resolvent variations can be written as
\begin{align}
   \label{eq:limit:dRll}
   \big\< \delta [\Rr_z]_{\ell\ell}, w \big\>
   &=
   \Big[ \Rr_z \Fopd \< \delta H(u_\infty), w \> \Fop \Rr_z \Big]_{\ell\ell} \\
   \notag
   &= (z-1)^{-2} \Big[ \Fopd \< \delta H(u_\infty), w \> \Fop \Big]_{\ell\ell}
      + 2(z-1)^{-1} \Big[ (\Rr_z-\Rhom_z) \Fopd \< \delta H(u_\infty), w \> \Fop  \Big]_{\ell\ell} \\
       \notag
   & \qquad
      + \Big[ (\Rr_z-\Rhom_z) \Fopd \< \delta H(u_\infty), w \> \Fop (\Rr_z-\Rhom_z) \Big]_{\ell\ell}, \qquad \text{and} \\
   \label{eq:limit:dRhomll}
   \big\< \delta [\Rhom_z]_{\ell\ell}, w \big\>
   &=
   (z-1)^{-2} \Big[ \Fopd \< \delta \Hhom(0), w \> \Fop \Big]_{\ell\ell}.
\end{align}
These expressions highlight the key estimates that we now require.

\begin{lemma} \label{th:convergence:FdHF_ab}
   We have the estimates
   \begin{align}
      \label{eq:convergence:sum_FdHF_ll}
      \sum_{\ell \in \L} \Big| \Big[
            \Fopd \< \delta H(u_\infty) - \delta \Hhom(0), w \> \Fop
         \Big]_{\ell\ell} \Big|
      &\lesssim \lvert N\rvert_{l^1}^{-d}, \\
      \label{eq:convergence:sum_RFdHF_ll}
      \sum_{\ell \in \L} \Big| \Big[
            (\Rr_z-\Rhom_z) \Fopd \< \delta H(u_\infty), w \> \Fop
         \Big]_{\ell\ell} \Big|
      &\lesssim \lvert N \rvert^{-d}_{l^3}, \qquad \text{and} \\
      \label{eq:convergence:sum_RFdHFR_ll}
      \sum_{\ell \in \L} \Big| \Big[
            (\Rr_z-\Rhom_z) \Fopd \< \delta H(u_\infty), w \> \Fop (\Rr_z-\Rhom_z)
         \Big]_{\ell\ell} \Big|
      &\lesssim \lvert N \rvert^{-d}_{l^5}.
   \end{align}
   In particular,
   \[
      | {\bf A}_N | \lesssim \lvert N \rvert^{-d}_{l^5}.
   \]
\end{lemma}
\begin{proof}
   {\it Proof of \eqref{eq:convergence:sum_FdHF_ll}:} We first estimate the site contribution by
   \begin{align*}
      & \hspace{-2cm} \Big[ \Fopd \< \delta H(u_\infty) - \delta \Hhom(0), w \> \Fop \Big]_{\ell\ell}
      \\ &=
      \sum_{n \in \L} \big( \nabla^3 V_\ell(D u_\infty (n)) - \nabla^3 V(0) \big)
            \big[ Dw(n), DF(n - \ell), DF(n - \ell) \big]
      \\ &\lesssim
      \sum_{n \in \L} |D u_\infty (n)| \, |Dw(n)| \, |DF(n-\ell)|^2
   \end{align*}
   Summing over $\ell$ and substituting $|DF(n-\ell)| \lesssim |n-\ell|_{l^0}^{-d}$ and $|D u_\infty (n)|\lesssim |n|_{l^0}^{-d}$ , yields
   \begin{align*}
      \sum_{\ell \in \L} \Big| \Big[
            \Fopd \< \delta H(u_\infty) - \delta \Hhom(0), w \> \Fop
         \Big]_{\ell\ell} \Big|
      &\lesssim \sum_{n \in \L} \sum_{\ell \in \L} |n|_{l^0}^{-d}
        |n-\ell|_{l^0}^{-2d} \lvert Dw(n) \rvert\\
      &\lesssim \sum_{n \in \L} |n|_{l^0}^{-d} |n|_{l^0,N}^{-d}
      \\&\lesssim
      \sum_{n \in \L_N} |n|_{l^0}^{-d} N^{-d} +  \sum_{n \in \L \setminus \L_N} |n|_{l^0}^{-2d}.\\
      &\lesssim \lvert N \rvert^{-d}_{l^1}.
   \end{align*}

   {\it Proof of \eqref{eq:convergence:sum_RFdHF_ll}:}
   Arguing as in the first part of the proof of
   \eqref{eq:convergence:sum_FdHF_ll}, employing Proposition~\ref{prop:
   resolvent new version} to estimate $\Rr_z-\Rhom_z$, we obtain
   \begin{align*}
      \sum_{\ell \in \L} \Big| \Big[
            (\Rr_z-\Rhom_z) \Fopd \< \delta H(u_\infty), w \> \Fop
         \Big]_{\ell\ell} \Big|
      &\lesssim \sum_{\ell, n, m \in \L}
      \Lr_1(\ell, m) |Dw(n)| |DF(n-\ell)| |DF(n-m)| \\
      &\lesssim \sum_{n \in \L} |Dw(n)| \sum_{\ell, m \in \L}
         \Lr_1(\ell, m) |n-\ell|_{l^0}^{-d} |n-m|_{l^0}^{-d}.
   \end{align*}
   As
   \begin{equation*}
   \sum_{\ell \in \L}  \Lr_1(\ell, m) |n-\ell|_{l^0}^{-d} \lesssim \Lr_2(n,m)
   \end{equation*}
   according to \eqref{eq: auxiliary estimates eq6cor4}, we see that
   \begin{align}
      \sum_{\ell, m \in \L}  &\Lr_1(\ell, m) |n-\ell|_{l^0}^{-d} |n-m|_{l^0}^{-d}\nonumber\\
      &\lesssim  \sum_{ m \in \L} \Lr_2(n, m) |n-m|_{l^0}^{-d} \nonumber\\
      &\lesssim  |n|_{l^2}^{-d} \sum_{ m \in \L} (|n-m|_{l^2}^{-d} + |m|_{l^2}^{-d}) |n-m|_{l^0}^{-d} + \sum_{ m \in \L} |n-m|_{l^2}^{-2d}  |m|_{l^2}^{-d} \nonumber\\
      &\lesssim |n|_{l^2}^{-d} + |n|_{l^5}^{-2d} + |n|_{l^2}^{-d}\nonumber\\
      &\lesssim |n|_{l^2}^{-d}, \label{eq:convergence:proofs:doublesum_L2_d_d}
   \end{align}
   where we also used \eqref{eq: auxiliary estimates eq3} and \eqref{eq: auxiliary estimates eq5}. Therefore,
   \[
      \sum_{\ell \in \L} \Big| \Big[
         (\Rr_z-\Rhom_z) \Fopd \< \delta H(u_\infty), w \> \Fop
      \Big]_{\ell\ell} \Big|
      \lesssim \sum_n |n|_{l^0,N}^{-d} |n|_{l^2}^{-d} \lesssim \lvert N \rvert^{-d}_{l^3}.
   \]

   {\it Proof of \eqref{eq:convergence:sum_RFdHFR_ll}:} The proof of this
   estimate is entirely analogous to that of
   \eqref{eq:convergence:sum_RFdHF_ll}, and only requires replacing the estimate
   \eqref{eq:convergence:proofs:doublesum_L2_d_d} with
   \begin{align*}
      \sum_{\ell,a,b \in \L}  \Lr_1(\ell, a) |n-a|_{l^0}^{-d} |n-b|_{l^0}^{-d} \Lr_1(\ell, b)
      &\lesssim  \sum_{ \ell \in \L} \Lr_2(\ell,n)^2 \nonumber\\
      &\lesssim  \sum_{ \ell \in \L} |n|_{l^4}^{-2d} |\ell|_{l^4}^{-2d} + |n-\ell|_{l^4}^{-2d} |\ell|_{l^4}^{-2d}+ |n|_{l^4}^{-2d} |n-\ell|_{l^4}^{-2d}\nonumber\\
      &\lesssim |n|_{l^4}^{-2d},
   \end{align*}
based on \eqref{eq: auxiliary estimates eq6cor4} and \eqref{eq: auxiliary estimates eq5}.

   Finally, the result $|{\bf A}_N| \lesssim \lvert N \rvert^{-d}_{l^5}$ is an immediate
   consequence of
   \eqref{eq:convergence:sum_FdHF_ll}--\eqref{eq:convergence:sum_RFdHFR_ll}.
\end{proof}

We now turn to the second term in \eqref{eq:limit:AN+BN},
${\bf B}_N = \sum_\ell {\bf B}_{N,\ell}$ where
\begin{align*}
   {\bf B}_{N,\ell}
   = -\frac{1}{2} \int_0^1 (1-s) \frac{1}{2\pi i} \oint_{\mathcal{C}} \log z\,
      {\rm Trace} \big\< \delta^2 [\Rr_z(u_\infty+s w)]_{\ell\ell}\, w, w \big\>
   \, dz \, ds,
\end{align*}
thus we now need to estimate the second variation of the resolvents.
Let $u_s := u_\infty+s w$, then
\begin{align*}
   \Big[\big\< \delta^2 \Rr_z(u_s)\, w, w \big\>\Big]_{\ell\ell}
   &=
   \Big[2\Rr_z(u_s) \Fopd \big\< \delta H(u_s), w \big\> \Fop
      \Rr_z(u_s) \Fopd \big\< \delta H(u_s), w \big\> \Fop \Rr_z(u_s)\Big]_{\ell\ell} \\
   & \qquad + \Big[\Rr_z(u_s) \Fopd \big\< \delta^2 H(u_s) w, w \big\> \Fop
                  \Rr_z(u_s)\Big]_{\ell\ell} \\
   &=: {\bf B}_{N,\ell}^{(1)} + {\bf B}_{N,\ell}^{(2)}.
\end{align*}

\begin{lemma} \label{th:limit:estimate BN}
   For sufficiently large $N$, we have the estimates
   \begin{align}
      \label{eq:limit:estimate BN1}
      \sum_{\ell\in\L} \big| {\bf B}_{N,\ell}^{(1)} \big| &\lesssim N^{-d}, \qquad \text{and} \\
      \label{eq:limit:estimate BN2}
      \sum_{\ell\in\L} \big| {\bf B}_{N,\ell}^{(2)} \big| &\lesssim N^{-d};
      \qquad \text{and in particular} \\
      \label{eq:limit:estimate BN}
      | {\bf B}_N | &\lesssim N^{-d},
   \end{align}
   with the implied constants independent of $s, z, N$.
\end{lemma}
\begin{proof}
   {\it Proof of \eqref{eq:limit:estimate BN2}:}
   According to Proposition \ref{prop: resolvent new version}
   and \eqref{eq: auxiliary estimates eq6cor4} we know that
\begin{align} \label{eq:limit:estimate BN2 aux1}
\big\lvert \big[\Rr_z(u_s)\big]_{m\ell} \big\rvert &\lesssim \delta_{m \ell} + \Lr_1(m,\ell),
   \qquad \text{and} \\
 \label{eq:limit:estimate BN2 aux2}
\sum_m(\delta_{m \ell} + \Lr_1(m,\ell))|n-m|_{l^0}^{-d} &\lesssim \lvert n -\ell \rvert_{l^2}^{-d}.
\end{align}
Analogously to the proof of Lemma~\ref{th:convergence:FdHF_ab}, we then calculate
   \begin{align*}
      \sum_{\ell \in \L} \Big| \Big[
            &\Rr_z(u_s)\Fopd \< \delta^2 H(u_s)w, w \> \Fop\Rr_z(u_s)
         \Big]_{\ell\ell} \Big|\\
      &\lesssim
      \sum_{n \in \L} |Dw(n)|^2
      \sum_{\ell,m,k \in \L}  |n-m|_{l^0}^{-d} |n-k|_{l^0}^{-d} (\delta_{m\ell} +\Lr_1(m,\ell)) (\delta_{k\ell} +\Lr_1(k, \ell))\\
      &\lesssim
      \sum_{n \in \L} |Dw(n)|^2  \sum_{\ell \in \L}  |n-\ell|_{l^4}^{-2d}
      \\&\lesssim
      \| Dw \|_{\ell^2}^2\\
      &\lesssim N^{-d}.
   \end{align*}

   {\it Proof of \eqref{eq:limit:estimate BN1}:}
   Throughout this proof let $A = \Fopd
   \big\< \delta H(u_s), w \big\> \Fop$, then
   \[{\bf B}_{N,\ell}^{(1)} =  \big[ 2\Rr_z(u_s) A \Rr_z(u_s) A \Rr_z(u_s) \big]_{\ell\ell}.\]
We  use \eqref{eq:limit:estimate BN2 aux1} and \eqref{eq:limit:estimate BN2 aux2}, as well as
\begin{align*}
      |A_{mn}|
      &= \bigg| \sum_{\xi \in \L} \nabla^3 V(Du_s(\xi))[ Dw(\xi), DF(\xi-m), DF(\xi-j) ]\bigg| \\
      &\lesssim \sum_{\xi \in \L} |Dw(\xi)| |\xi-m|_{l^0}^{-d} |\xi-n|_{l^0}^{-d}
   \end{align*}
   to deduce that
\begin{align*}
      |(\Rr_z(u_s)A)_{mn}|
      &= \bigg| \sum_{\xi,k \in \L} (\delta_{mk} + \Lr_1(m,k)) |Dw(\xi)| |\xi-k|_{l^0}^{-d} |\xi-n|_{l^0}^{-d} \\
      &\lesssim \sum_{\xi \in \L} |Dw(\xi)| |\xi-m|_{l^2}^{-d} |\xi-n|_{l^0}^{-d}.
   \end{align*}
Therefore, using also \eqref{eq: auxiliary estimates eq3} and \eqref{eq: auxiliary estimates eq6cor4}
\begin{align*}
\big\lvert {\bf B}_{N,\ell}^{(1)} \big\rvert &=  \big\lvert \big[ 2\Rr_z(u_s) A \Rr_z(u_s) A \Rr_z(u_s) \big]_{\ell\ell} \big\lvert\\
&\lesssim \sum_{\xi, \eta, k,m} |Dw(\xi)| |\xi-\ell|_{l^2}^{-d} |\xi-k|_{l^0}^{-d} |Dw(\eta)| |\eta-k|_{l^2}^{-d} |\eta-m|_{l^0}^{-d} (\delta_{m\ell} + \Lr_1(m,\ell))\\
&\lesssim \sum_{\xi, \eta} |Dw(\xi)| |\xi-\ell|_{l^2}^{-d} |\xi-\eta|_{l^3}^{-d} |Dw(\eta)| |\eta-\ell|_{l^2}^{-d}.
   \end{align*}
Summing over $\ell$ and applying \eqref{eq: auxiliary estimates eq3} again then gives
\begin{align*}
\sum_\ell\big\lvert {\bf B}_{N,\ell}^{(1)} \big\rvert &= \sum_\ell \big\lvert \big[ 2\Rr_z(u_s) A \Rr_z(u_s) A \Rr_z(u_s) \big]_{\ell\ell} \big\lvert\\
 &\lesssim \sum_{\xi, \eta, \ell} |Dw(\xi)| |\xi-\ell|_{l^2}^{-d} |\xi-\eta|_{l^3}^{-d} |Dw(\eta)| |\eta-\ell|_{l^2}^{-d}\\
&\lesssim \sum_{\xi, \eta} |Dw(\xi)| |Dw(\eta)| |\eta-\xi|_{l^8}^{-2d}.
   \end{align*}
Let us split the domain of the sum. First,
\begin{align*}
\sum_{\xi, \eta \in \L_{2N}} |Dw(\xi)| |Dw(\eta)| |\eta-\xi|_{l^8}^{-2d}
&\lesssim N^{-2d} \sum_{\xi, \eta \in \L_{2N}} |\eta-\xi|_{l^8}^{-2d}
\lesssim N^{-d}.
\end{align*}
For the mixed terms we use \eqref{eq: auxiliary estimates eq1} and $\lvert \eta \rvert \sim \lvert \eta - \xi\rvert$
\begin{align}
\sum_{ \xi \in \L_{N},  \eta \in \L \backslash \L_{2N}} |Dw(\xi)| |Dw(\eta)| |\eta-\xi|_{l^8}^{-2d} &\lesssim  \sum_{\xi \in \L_{N}} |Dw(\xi)| \sum_{\eta \in \L \backslash \L_{2N}}   \lvert \eta \rvert_{l^0}^{-d} |\eta-\xi|_{l^8}^{-2d} \nonumber \\
&\lesssim  \sum_{\xi \in \L_{N}} |Dw(\xi)| \sum_{\eta \in \L \backslash \L_{2N}}   \lvert \eta \rvert_{l^8}^{-3d}  \nonumber \\
 &\lesssim  \sum_{\xi \in \L_{N}}  |Dw(\xi)| \lvert N \rvert_{l^8}^{-2d} \nonumber  \\
 &\lesssim \lvert N \rvert_{l^8}^{-2d} N^{d} N^{-d} = \lvert N \rvert_{l^8}^{-2d} \label{eq:limit:BN1 farfield1}
\end{align}
and, due to \eqref{eq: auxiliary estimates eq4} and \eqref{eq: auxiliary estimates eq1},
\begin{align}
\sum_{\xi,\eta \in \L \backslash \L_N} |Dw(\xi)| |Dw(\eta)| |\eta-\xi|_{l^8}^{-2d} &\lesssim  \sum_{\xi,\eta \in \L \backslash \L_N} \lvert \xi \rvert_{l^0}^{-d} \lvert \eta \rvert_{l^0}^{-d} |\eta-\xi|_{l^8}^{-2d}\nonumber  \\
&\lesssim  \sum_{\xi \in \L \backslash \L_N} \lvert \xi \rvert_{l^0}^{-2d}
\lesssim N^{-d} \label{eq:limit:BN1 farfield2}.
\end{align}
 In summary, we have shown that
\[\sum_\ell\big\lvert {\bf B}_{N,\ell}^{(1)} \big\rvert \lesssim N^{-d} .
   \]

   Finally, the estimate \eqref{eq:limit:estimate BN} is an immediate consequence
   of \eqref{eq:limit:estimate BN1} and \eqref{eq:limit:estimate BN2}.
\end{proof}

\begin{corollary} \label{th:limit:SvN-Su}
   For $N$ sufficiently large,
   \[
      \big| \SS^+(\vN) - \SS^+(u_\infty) \big|
      \lesssim
      \lvert N \rvert^{-d}_{l^5}.
   \]
\end{corollary}
\begin{proof}
   This result follows by combining Lemma~\ref{th:convergence:FdHF_ab}
   and Lemma~\ref{th:limit:estimate BN}.
\end{proof}

\subsubsection{The term $\SS^+_N(u_N) - \SS^+_N(\vNper)$}
\label{sec:limit:SNuN-SNvNper}
Recalling the error split \eqref{eq:limit:split of S error} we now turn to
$\SS^+_N(\vNper) - \SS^+_N(u_N)$, the periodic analogue of $\SS(\vN) - \SS(u_\infty)$.
Recall that in estimating the latter, we relied on the uniform estimate $\|D\vN-D u_\infty\|_{\ell^\infty} \lesssim N^{-d}$, as well as the far field estimate $\lvert D u_\infty \rvert \leq \lvert \ell \rvert_{l^0}^{-d}$.

As the analogous estimate $\|D\vNper - Du_N\|_{\ell^\infty} \lesssim N^{-d}$ holds and the far field estimates are no longer needed, the estimates for $\SS^+_N(\vNper) -\SS^+_N(u_N)$ are therefore, for the most part, analogous. Hence, we will skip many details.

The key difference is that the $N^{-d} |n-m|_{l^5, \L_N}^{-d}$ in the periodic resolvent
estimate, Lemma \ref{th:resolvent estimate uNper}, gives some additional terms.

To justify these claims, we decompose the new error term similarly to the previous one. Let $w_N := u_N - \vNper$. Using the periodicity,
we have
\begin{align}
\sum_{\ell \in \L_N} \big\< \delta \Shom_{N,\ell}(0), w_N \big\> &= -\frac{1}{2}\sum_{\ell,n \in \L_N} \nabla^3V(0)[Dw_N(n), DF_N(n-\ell),DF_N(n-\ell)] \nonumber \\
&= -\frac{1}{2}\sum_{m \in \L_N} \nabla^3V(0)\Big[\sum_{n \in \L_N}Dw_N(n), DF_N(m),DF_N(m)\Big] \nonumber \\
&= 0. \label{eq:periodfirstvarioanvanishes}
\end{align}
Hence we can write
\begin{align}
   \notag
   \SS_N^+(u_N) - \SS^+_N(\vNper)
   &=
   \sum_{\ell \in \L_N} \Big( \SS^+_{N,\ell}(u_N) - \SS^+_{N,\ell}(\vNper)
                  - \big\< \delta \Shom_{N,\ell}(0), w_N \big\> \Big)
   %
   \\ \notag & \hspace{-2cm} =
   \sum_{\ell \in \L_N}
      \big\< \delta \SS^+_{N,\ell}(\vNper) - \delta \Shom_{N,\ell}(0), w_N \big\>
   + \sum_{\ell \in \L_N}
      \int_{0}^1 (1-s) \big\< \delta^2 \SS^+_{N,\ell}(\vNper+sw_N) w_N, w_N \big\> \,ds
   \\ &= {\bf pA}_N + {\bf pB}_N.
   \label{eq:limit:pAN+pBN}
\end{align}

Note in particular that we have expanded $\SS^+_{N,\ell}$ around $\vNper$ instead
of $u_N$. Since the decay estimate for $\Rr_{N,z}(\vNper)$ is equivalent to
that for $\Rr_z(u_\infty)$ according to Lemma \ref{th:resolvent estimate vNper}, it follows that we can repeat the proof of Lemma~\ref{th:convergence:FdHF_ab} almost verbatim to obtain the following
result.
\begin{lemma} \label{th:limit:pAN}
   For $N$ sufficiently large, $|{\bf pA}_N | \lesssim \lvert N\rvert_{l^5}^{-d}$.
\end{lemma}

We can therefore turn immediately towards the second term,
${\bf pB}_N = \sum_{\ell \in \L_N} {\bf pB}_{N,\ell}$, where
\begin{align*}
   {\bf pB}_{N,\ell}
   = -\frac{1}{2} \int_0^1 (1-s) \frac{1}{2\pi i} \oint_{\mathcal{C}} \log z\,
      {\rm Trace} \big\< \delta^2 [\Rr_{N,z}(v_s)]_{\ell\ell}\, w_N, w_N \big\>
   \, dz \, ds,
\end{align*}
$v_s := \vNper + s u_N$ and
\begin{align*}
   \Big[\big\< \delta^2 [\Rr_{N,z}(v_s)]_{\ell\ell}\, w, w \big\>\Big]_{\ell\ell}
   &=
   \Big[2\Rr_{N,z}(v_s) \Fopd \big\< \delta H_N(v_s), w_N \big\> \Fop
      \Rr_{N,z}(v_s) \Fopd \big\< \delta H_N(v_s), w_N \big\> \Fop \Rr_{N,z}(v_s)\Big]_{\ell\ell} \\
   & \qquad + \Big[\Rr_{N,z}(v_s) \Fopd \big\< \delta^2 H_N(v_s) w_N, w_N \big\> \Fop
                  \Rr_{N,z}(v_s)\Big]_{\ell\ell}.
                   \\
   &=: {\bf pB}_{N,\ell}^{(1)} + {\bf pB}_{N,\ell}^{(2)}.
\end{align*}

\begin{lemma} \label{th:limit:estimate pBN}
   For sufficiently large $N$, we have
   \begin{align}
      \label{eq:limit:estimate pBN1+pBN2}
      &\sum_{\ell\in\L} \big| {\bf pB}_{N,\ell}^{(1)} \big| \lesssim N^{-d}, \qquad
      \text{and} \qquad
      \sum_{\ell\in\L} \big| {\bf pB}_{N,\ell}^{(2)} \big| \lesssim N^{-d}; \\
      \label{eq:limit:estimate pBN}
      &\text{and, in particular, } \qquad
      | {\bf pB}_N | \lesssim N^{-d},
   \end{align}
   with the implied constants independent of $s, z, N$.
\end{lemma}
\begin{proof}
Instead of \eqref{eq:limit:estimate BN2 aux1} and \eqref{eq:limit:estimate BN2 aux2}, we now use that
\begin{equation} \label{eq:limit:estimate pBN2 aux1}
\big\lvert \big[\Rr_{N,z}(v_s)\big]_{m\ell} \big\rvert \lesssim \delta_{m \ell} + \Lr_{1,\L_N}(m,\ell)+ N^{-d} |m-\ell|_{l^5,\L_N}^{-d},
\end{equation}
as well as, \eqref{eq: auxiliary estimates eq6cor4} and \eqref{eq: auxiliary estimates eq3} to obtain
\begin{align}
\sum_{m \in \Lambda_N}&(\delta_{m \ell} + \Lr_{1,\L_N}(m,\ell)+ N^{-d} |m-\ell|_{l^5,\L_N}^{-d})|n-m|_{l^0,\L_N}^{-d}\nonumber \\
 &\lesssim \lvert n -\ell \rvert_{l^2,\L_N}^{-d} +N^{-d}\lvert n -\ell \rvert_{l^6,\L_N}^{-d}
 \lesssim \lvert n -\ell \rvert_{l^2,\L_N}^{-d}.\label{eq:limit:estimate pBN2 aux2}
\end{align}
 As the result in \eqref{eq:limit:estimate pBN2 aux2} is the same as in \eqref{eq:limit:estimate BN2 aux2}, the rest of the proof for ${\bf pB}_{N}^{(2)}$ stays the same. For ${\bf pB}_{N}^{(1)}$ we also get
\begin{equation*}
\sum_\ell\big\lvert {\bf pB}_{N,\ell}^{(1)} \big\rvert \lesssim \sum_{\xi, \eta \in \Lambda_N} |Dw_N(\xi)| |Dw_N(\eta)| |\eta-\xi|_{l^8,\L_N}^{-2d}
   \end{equation*}
   exactly as before. Of course, we do not need far field estimates now but only the simpler estimate
 \begin{align*}
\sum_\ell\big\lvert {\bf pB}_{N,\ell}^{(1)} \big\rvert &\lesssim \sum_{\xi, \eta \in \Lambda_N} |Dw_N(\xi)| |Dw_N(\eta)| |\eta-\xi|_{l^8,\L_N}^{-2d}\\
&\lesssim N^{-2d}\sum_{\xi, \eta \in \Lambda_N} |\eta-\xi|_{l^8,\L_N}^{-2d}
      \lesssim N^{-d}. \qedhere
   \end{align*}
\end{proof}

\begin{corollary} \label{th:limit:SvN-SuN}
   For $N$ sufficiently large, we have $|\SS^+_N(\vNper) - \SS^+_N(u_N)| \lesssim
   \lvert N\rvert_{l^5}^{-d}$
\end{corollary}
\begin{proof}
   The result follows by combining Lemma~\ref{th:limit:pAN} and
   Lemma~\ref{th:limit:estimate pBN} with \eqref{eq:limit:pAN+pBN}.
\end{proof}

\subsubsection{The term $\SS^+(\vN) - \SS^+_N(\vNper)$}
The final term from \eqref{eq:limit:split of S error} can be estimated by
comparing $\Rr_z(\vN)$ with $\Rr_{N,z}(\vNper)$, which we will reduce to the error estimate
for $F_N - F$ from Lemma~\ref{th:error estimate FN}.

We begin by recalling the expressions, valid for $N$ sufficiently large,
\begin{align*}
   \Rr_z(\vN) &= \big(z I_{\ell^2(\Lambda)} - \Fopd H(\vN) \Fop \big)^{-1}
            =: \big( z - A \big)^{-1}\\
   \Rr_{N,z}(\vNper) &= \big(z I_{\ell^2(\Lambda_N)} - \Fop_N H_N(\vN) \Fop_N \big)^{-1}
            =: \big( z - A_N \big)^{-1}
\end{align*}
We extend the ``matrix'' $A_N$ by defining $[A_N]_{minj} = \delta_{mn}
\delta_{ij}$ for $(m, n) \in \L^2 \setminus \L_N^2$, which induces a
corresponding extension of $\Rr_{N,z}(\vNper)$ to a $\Rr_z^{\rm ext}(\vNper)$ by $\Rhom_z$. This allows us
to compare
\begin{align*}
   \big[ \Rr_z(\vN) - \Rr_z^{\rm ext}(\vNper) \big]_{\ell\ell}
   &=
   \big[ \Rr_z(\vN) \big( A-A_N \big) \Rr_z^{\rm ext}(\vNper) \big]_{\ell\ell} \\
   &=
    (z-1)^{-2} \big[ A-A_N \big]_{\ell\ell} \\
 &\qquad
    +
    (z-1)^{-1} \Big[ (\Rr_z(\vN)-\Rhom_z) \big( A - A_N \big) \Big]_{\ell\ell} \\
 &\qquad
   +
   (z-1)^{-1} \Big[ \big( A - A_N \big) (\Rr_z^{\rm ext}(\vNper)-\Rhom_z) \Big]_{\ell\ell} \\
 &\qquad
    +
    \Big[ (\Rr_z(\vN)-\Rhom_z) \big( A -A_N \big) (\Rr_z^{\rm ext}(\vNper)-\Rhom_z) \Big]_{\ell\ell} \\
 &=: {\bf R}_\ell^{(1)} + {\bf R}_\ell^{(2)} + {\bf R}_\ell^{(3)} + {\bf R}_\ell^{(4)}.
\end{align*}

\begin{lemma}
   \begin{equation*}
      \big| [A_N - A]_{ij} \big| \lesssim
      \begin{cases}
         N^{-d} \sum_{n \in \L_{N/2}} |n|_{l^0}^{-d} \big( |j-n|_{l^0}^{-d} + |i-n|_{l^0}^{-d}\big), &
            \text{if } (i, j) \in \L_N^2, \\[2mm]
         \sum_{n \in \L_{N/2}} |n|_{l^0}^{-d} |j-n|_{l^0}^{-d} |i-n|_{l^0}^{-d}, &
            \text{if } (i, j) \in \L^2 \setminus \L_N^2.
      \end{cases}
   \end{equation*}
\end{lemma}
\begin{proof}
   For $(i,j) \in \L_N^2$ we calculate
   \begin{align*}
      [A_N - A]_{ij}
      &= \big[ (A_N - I) - (A-I)\big]_{ij}  \\
      &= \sum_{n \in \L_{N/2}}
         \big(\nabla^2 V_n(D\vNper(n)) - \nabla^2 V(0)\big)[DF_N(i-n), DF_N(j-n)] \\
         & \qquad -
         \sum_{n \in \L_{N/2}}
         \big(\nabla^2 V_n(D\vN(n)) - \nabla^2 V(0)\big)[DF(i-n), DF(j-n)],
   \end{align*}
   where we have used the fact that $D\vN(n) = 0$ for $n \in \L \setminus
   \L_{N/2}$. Observing that $D\vNper(n) = D\vN(n)$ for $n \in \L_{N/2}$ and recalling
   that $|D\vN(n)| \lesssim |n|_{\ell^0}^{-d}$, we obtain
   \begin{align*}
      \big| [A_N - A]_{ij} \big|
      &\lesssim
      \sum_{n \in \L_{N/2}}
         |D\vN(n)|\, \Big( |DF_N(i-n) - DF(i-n)| \,|DF(j-n)| \\[-3mm]
      & \hspace{3.5cm} + |DF_N(i-n)| \, |DF_N(j-n) - DF(j-n)| \Big) \\[1mm]
      &\lesssim
      N^{-d} \sum_{n \in \L_{N/2}}
         |n|_{\ell^0}^{-d} \Big( |j-n|_{\ell^0}^{-d} + |i-n|_{\ell^0}^{-d} \Big).
      %
   \end{align*}
   where we used that $j-n,\, i-n \in \Lambda_{3N/2}$.
   This completes the case $(i, j) \in \L_N^2$.

   In the case $(i,j) \in \L^2 \setminus \L_N^2$ we simply have
   \begin{align*}
      \big| [A_N - A]_{ij} \big|  &= \big| [I-A]_{ij} \big| \\
   &\lesssim
   \sum_{n \in \L_{N/2}}
      |D\vN(n)|\,  |DF(i-n)| \,|DF(j-n)| \\
   &\lesssim
   \sum_{n \in \L_{N/2}}
      |n|_{l^0}^{-d}\, |i-n|_{l^0}^{-d} \, |j-n|_{l^0}^{-d}  \qedhere
   \end{align*}
\end{proof}

\begin{lemma}  \label{th:SN-S-resolvent-estimates}
   For all $j \in \{1,2,3,4\}$ we have the estimate
   \begin{equation*}
      \sum_{\ell \in \L} \big| {\bf R}_\ell^{(j)} \big| \lesssim \lvert N \rvert_{l^3}^{-d}.
   \end{equation*}
\end{lemma}
\begin{proof}
   According to \eqref{eq: auxiliary estimates eq2} we can estimate
   \begin{align*}
      \sum_{\ell \in \L} \big|{\bf R}_\ell^{(1)}\big|
      &\lesssim
      \sum_{n \in \L_{N/2}} |n|_{l^0}^{-d} \bigg(
            N^{-d} \sum_{\ell \in \L_N} |\ell-n|_{l^0}^{-d}
            + \sum_{\ell \in \L\setminus \L_N} |\ell-n|_{l^0}^{-2d} \bigg) \\
      &\lesssim
      \sum_{n \in \L_{N/2}} |n|_{l^0}^{-d} \big( \lvert N \rvert_{l^1}^{-d} + \lvert N \rvert_{l^0}^{-d} \big)
      \lesssim \lvert N \rvert_{l^2}^{-d}.
   \end{align*}
 Furthermore, using \eqref{eq: auxiliary estimates eq6cor4}, \eqref{eq: auxiliary estimates eq2}, and \eqref{eq: auxiliary estimates eq3},
   \begin{align*}
      \sum_{\ell \in \L} \big|{\bf R}_\ell^{(2)}\big| +
       \sum_{\ell \in \L} \big|{\bf R}_\ell^{(3)}\big|
      &\lesssim
         \sum_{\ell,i \in \L}
            \Lr_1(\ell,i) \big| [A_N - A]_{i\ell} \big| \\
      &\lesssim
      \sum_{n \in \L_{N/2}} |n|_{l^0}^{-d} \bigg(
         N^{-d} \sum_{(\ell, i) \in \L_N^2}
            \Lr_1(\ell,i) \big( |i-n|_{l^0}^{-d} + |\ell-n|_{l^0}^{-d} \big) \\
         & \hspace{1cm} +
         \sum_{(\ell, i) \in \L^2 \setminus \L_N^2}
            \Lr_1(\ell,i) |i-n|_{l^0}^{-d} |\ell-n|_{l^0}^{-d} \bigg)\\
            &\lesssim
      \sum_{n \in \L_{N/2}} |n|_{l^0}^{-d} \bigg(
         N^{-d} \sum_{i \in \L_N} \Lr_2(n,i) + \sum_{i \in \L \setminus \L_N} \Lr_2(n,i)\,  |i-n|_{l^0}^{-d}  \bigg)\\
         &\lesssim
      \sum_{n \in \L_{N/2}} |n|_{l^0}^{-d} \bigg(
         \lvert N \rvert_{l^3}^{-d} |n|_{l^5}^{-d}  + \sum_{i \in \L \setminus \L_N} |n|_{l^2}^{-d} |i|_{l^2}^{-2d}  \bigg)\\
            &\lesssim
      \sum_{n \in \L_{N/2}} |n|_{l^0}^{-d} \bigg( \lvert N \rvert_{l^3}^{-d} |n|_{l^5}^{-d} + \lvert N \rvert_{l^2}^{-d} |n|_{l^2}^{-d}\bigg)
      \lesssim \lvert N \rvert_{l^3}^{-d}.
   \end{align*}
 And finally,
   \begin{align*}
      \sum_{\ell \in \L} \big|{\bf R}_\ell^{(4)}\big|
      &\lesssim \sum_{\ell,i,j \in \L} \Lr_1(\ell,i) \big| [A_N - A]_{ij} \big| \Lr_1(j,\ell) \\
      &\lesssim \sum_{\ell \in \L} \sum_{n \in \L_{N/2}} |n|_{l^0}^{-d} \bigg( N^{-d} \sum_{(i,j) \in \L_N^2} \Lr_1(j,\ell) \Lr_1(\ell,i)\big( |j-n|_{l^0}^{-d}+ |i-n|_{l^0}^{-d}\big)\\
      &\qquad  + \sum_{(i,j) \in \L^2 \setminus \L_N^2} \Lr_1(j,\ell) \Lr_1(\ell,i) |j-n|_{l^0}^{-d} |i-n|_{l^0}^{-d}  \bigg)\\
       &\lesssim \sum_{n \in \L_{N/2}} |n|_{l^0}^{-d} \bigg( N^{-d} \sum_{(i,j) \in \L_N^2} \Lr_1(i,j)\big( |j-n|_{l^0}^{-d}+ |i-n|_{l^0}^{-d}\big)\\
      &\qquad  + \sum_{(i,j) \in \L^2 \setminus \L_N^2} \Lr_1(i,j) |j-n|_{l^0}^{-d} |i-n|_{l^0}^{-d}  \bigg)\\
      &\lesssim \sum_{n \in \L_{N/2}} |n|_{l^0}^{-d} \bigg( N^{-d} \sum_{i \in \L_N} \Lr_2(i,n) + \sum_{i \in \L \setminus \L_N} \Lr_2(i,n) |i-n|_{l^0}^{-d}  \bigg)\\
       &\lesssim \lvert N \rvert_{l^3}^{-d},
   \end{align*}
   where we used the submultiplicativity of $\Lr_1$ and \eqref{eq: auxiliary estimates eq6cor4}, as well as the end of the previous estimate.
\end{proof}

\begin{corollary} \label{th:SN-S}
   For $N$ sufficiently large, $|\SS^+_N(\vNper) - \SS^+(\vN)| \lesssim \lvert N \rvert_{l^3}^{-d}$.
\end{corollary}
\begin{proof}
According to Proposition \ref{prop: localitydecomposedentropy} and \eqref{eq: auxiliary estimates eq1}, we have
\begin{equation*}
\sum_{\ell \in \L \backslash \L_N} \lvert \SS_\ell(v_N) - \langle \delta \Shom_\ell(0), v_N \rangle \rvert \lesssim \lvert N \rvert_{l^2}^{-d}.
\end{equation*}
Furthermore, we use that
\begin{equation*}
\sum_{\ell \in \Lambda_N} \langle \delta \Shom_{N,\ell},\vNper \rangle = 0
\end{equation*}
according to \eqref{eq:periodfirstvarioanvanishes}. Hence, we have
\begin{align}
|\SS_N(\vNper) - \SS(\vN)| &\lesssim \lvert N \rvert_{l^2}^{-d} + \sum_{\ell \in \Lambda_N} \lvert \SS_{\ell}(\vN) - \SS_{N,\ell}(\vNper) \rvert \nonumber \\
&\qquad + \sum_{\ell \in \Lambda_N} \lvert \langle \delta \Shom_{\ell},\vN \rangle - \langle \delta \Shom_{N,\ell},\vNper \rangle \rvert. \label{eq:SNvN-SvN}
\end{align}
Lemma~\ref{th:SN-S-resolvent-estimates} implies
\begin{align*}
\sum_{\ell \in \Lambda_N}  \lvert \SS_{\ell}(\vN) - \SS_{N,\ell}(\vNper) \rvert &\lesssim \sum_{\ell \in \Lambda_N}  \Big\lvert {\rm Trace}  \oint_{\mathcal{C}} \log z \big( \Rr_{z}(\vN)_{\ell \ell} - \Rr_{N,z}(\vNper)_{\ell \ell}\big) \,dz \Big\rvert\\
&= \sum_{\ell \in \Lambda_N}  \Big\lvert {\rm Trace}  \oint_{\mathcal{C}} \log z \big( \Rr_{z}(\vN) -\Rr^{\rm ext}_z(\vNper)\big)_{\ell \ell} \,dz \Big\rvert\\
&\lesssim \lvert N \rvert^{-d}_{l^3}.
\end{align*}
For the last term in \eqref{eq:SNvN-SvN} we calculate
\begin{align*}
\langle \delta \Shom_{\ell},\vN \rangle - \langle \delta \Shom_{N,\ell},\vNper \rangle &= -\smfrac{1}{2}{\rm Trace} (F  \langle \delta \Hhom ,\vN \rangle F)_{\ell \ell} +\smfrac{1}{2} {\rm Trace} (F_N  \langle \delta \Hhom_N ,\vNper \rangle F_N)_{\ell \ell}\\
&=-\smfrac{1}{2}{\rm Trace} \sum_{m \in \L_N} \nabla^3 V(0)\Big[D\vN(m),DF(m-\ell)+DF_N(m-\ell),\\
&\hspace{5cm} DF(m-\ell)-DF_N(m-\ell)\Big].
\end{align*}
With Lemma \ref{th:error estimate FN} we therefore obtain
\begin{align*}
\sum_{\ell \in \L_N} \lvert \langle \delta \Shom_{\ell},\vN \rangle - \langle \delta \Shom_{N,\ell},\vNper \rangle \rvert &\lesssim \sum_{\ell, m \in \L_N} \lvert m \rvert^{-d}_{l^0} \lvert \ell - m \rvert^{-d}_{l^0} N^{-d}\\
&\lesssim \lvert N\rvert^{-d}_{l^2}. \qedhere
\end{align*}
\end{proof}
\begin{proof}[Proof of Proposition \ref{prop:generalconvergencerate}]
The result is an immediate consequence of the splitting \eqref{eq:limit:split
   of S error} where the three individual terms are, respectively, estimated in
   Corollaries~\ref{th:limit:SvN-Su}, \ref{th:limit:SvN-SuN} and~\ref{th:SN-S}.
\end{proof}

\begin{proof}[Proof of Theorem~\ref{theo: main result}(3)]
Setting $u_N := \us_N$ and $u_\infty := \us$ satisfy \eqref{eq:uNuinftyAssumptions}, as well as $\SS_N(\us_N) = \SS^+_N(\us_N) $ and $\SS(\us) = \SS^+(\us)$. Therefore, the result is a consequence of Proposition \ref{prop:generalconvergencerate}.
\end{proof}

%% file: saddle.tex
\section{Thermodynamic Limit of HTST}
\label{sec:saddle}
\subsection{Approximation of the Saddle point}
\label{sec:approx-saddle}
Recall our starting assumption in \eqref{eq:index-1-saddle-ell2} that there
exist $\usa, \bar\phi \in \Wi$, $\bar\phi \neq 0$, and $\bar\lambda < 0, c_0 > 0$ such that
\begin{equation} \label{eq:prftst:index-1-saddle-review}
   \begin{split}
      \delta \E(\usa) &= 0, \\
      \Hs \bar\phi &= \bar\lambda \bar\phi, \\
      \< \Hs v, v \> & \geq c_0 \| D v \|_{\ell^2}^2 \qquad \text{for all } v \in \dot{\W}^{1,2} \text{ with } \< v, \bar\phi \>_{\dot{\W}^{1,2}, (\dot{\W}^{1,2})'} = 0.
   \end{split}
\end{equation}
Since \cite[Thm 1]{EOS2016} in fact applies to all critical points and not only
minimisers, we again have
\[
   |D^j \usa(\ell)| \lesssim |\ell|^{1-d-j}
   \qquad \text{for $1 \leq j \leq p-2$}.
\]
Furthermore, we even have exponential decay of the unstable mode $\bar\phi$.

\begin{proposition} \label{eq:prftst:decay phi}
   Under Assumption \eqref{eq:prftst:index-1-saddle-review} we have
   \[
      |\bar\phi(\ell)| \lesssim \exp( - c |\ell|).
   \]
\end{proposition}
\begin{proof}
   We rewrite the eigenvalue equation as
   \[
      ((\Hs)^M - \bar\lambda I) \bar\phi = f := ((\Hs)^M - \Hs) \bar\phi,
   \]
   where $(\Hs)^M$ is defined by \eqref{eq:res:defn of HM} which ensures that
   $f$ is compactly supported.

   Since $\|(\Hs)^M - \Hhom\|_{\mathcal{L}(\Wi,(\Wi)')} \to 0$ as $M \to \infty$,
   it follows that, for $M$ sufficiently large, $\sigma((\Hs)^M) \subset [0, \infty)$.
   Since $\bar\lambda$ is negative, standard Coombe--Thomas type estimates
   (see e.g. \cite{2015-qmtb1} for an applicable result) yield
   \[
      \Big| \big[ ((\Hs)^M - \bar\lambda I)^{-1} \big]_{\ell m} \Big|
         \lesssim e^{ - \gamma |\ell - m| },
   \]
   for some $\gamma > 0$. The stated result now follows immediately.
\end{proof}

Next, we observe that the related $\Wi$-eigenvalue problem has the same structure.

\begin{proposition} \label{th:index-1-saddle-H1}
   There exist $\bar\psi \in \Wi$, $\bar\mu < 0$, $c_1 > 0$ such that
   \begin{equation}
      \label{eq:index-1-saddle-H1}
      \begin{split}
      \Hs \bar\psi &= \bar\mu \Hhom \bar\psi,  \\
      \< \Hs v, v \> & \geq c_1 \< \Hhom v, v \> \qquad \text{ whenever } \<\Hhom v, \bar\psi\> = 0.
      \end{split}
   \end{equation}
   Moreover, $|D^j \bar\psi(\ell)| \lesssim |\ell|^{1-d-j}$ for $1 \leq j \leq p-2$.
\end{proposition}
\begin{proof}
   {\it Step 1: Existence. }
As
\[\langle \Hs v,v \rangle \geq \bar\lambda  \frac{ (v,\bar\phi)_{\ell^2}^2}{(\bar\phi,\bar\phi)_{\ell^2}},\]
we can set $v=\psi-\psi_j$ for a sequence with $\psi_j \rightharpoonup \psi$ in $\Wi$ to find
\[\liminf_j \langle \Hs \psi_j,\psi_j \rangle \geq \langle \Hs \psi,\psi \rangle +\frac{\bar\lambda }{(\bar\phi,\bar\phi)_{\ell^2}} \limsup_j \lvert (\psi - \psi_j,\bar\phi)_{\ell^2}\rvert^2.\]
Additionally, the last term vanishes, as $\bar\phi = \bar\lambda^{-1} H_s \bar\phi \in (\Wi)'$.
We have thus shown that $\psi \mapsto \langle H_s \psi, \psi \rangle$ is weakly lower semi-continuous in $\Wi$.

Let $R(\psi) = \< \Hs \psi, \psi \> / \< \Hhom \psi, \psi \>$ be the associated
Rayleigh quotient for $\psi \in \Wi \setminus \{\psi\equiv c: c\in \R^d\}$. Then $R(\bar\phi) < 0$. Furthermore, we have $\Hs \in \mathcal{L}(\Wi, (\Wi)^*)$ which together with {\bf (STAB)} implies that
\[
   R(\psi) \geq \frac{- C \lVert D \psi \rVert_{\ell^2}^2}{c_0/2 \lVert D \psi \rVert_{\ell^2}^2} = \frac{-C}{c_0/2},
\]
where $C = \|\Hs\|_{\mathcal{L}}$; hence, $\inf R$ is finite.

Let $\psi_j$ be a minimising sequence with $\< \Hhom \psi_j, \psi_j \> = 1$ and
$R(\psi_j) \downarrow \inf R$. Then, up to extracting a subsequence, $D\psi_j
\rightharpoonup D\bar\psi$ weakly in $\ell^2$. If $\< \Hhom \bar\psi, \bar\psi \> = 1$, then $\bar\psi$ is a
minimiser of $R$ and the existence of a corresponding $\bar\mu<0$ for
\eqref{eq:index-1-saddle-H1} follows.

Set $\theta=\< \Hhom \bar\psi, \bar\psi \>$. As $\< \Hhom \psi, \psi \>$ is non-negative and weakly lower semi-continuous, we have $\theta \in [0,1]$. It remains to show, that $\theta =1$.
If we had $\theta \in (0,1)$, then
\[R(\bar\psi) =\frac{1}{\theta} \< \Hs \bar\psi, \bar\psi \> \leq \liminf \frac{1}{\theta} \< \Hs \psi_j, \psi_j \> \leq \frac{\inf R}{\theta} < \inf R, \]
a contradiction. As a last case, if $\theta =0$, then $\bar\psi$ would be constant. Using the weak lower semi-continuity of $v \mapsto \< \Hs v, v \>$ we have
\[\inf R = \lim_j R(\psi_j) = \lim_j \< \Hs \psi_j, \psi_j \>  \geq 0,\]
and hence obtain another contradiction. Thus $\bar\psi$ is a minimizer of $R$ and we can set $\bar\mu := R(\bar\psi)$.

{\it Step 2: Stability. }
We now show that the
rest of the spectrum is bounded below by $c := c_0 / \|\Hhom\|_{\mathcal{L}}$, where $c_0$ is the constant from \eqref{eq:prftst:index-1-saddle-review}.
 First note that,
\begin{equation}
\label{eq:index-1-saddle-stabilityestimate}
   \< \Hs v, v \>  \geq c_0 \| D v \|_{\ell^2}^2 \geq c \langle \Hhom v,v \rangle \qquad \text{whenever } \< v, \bar\phi\>_{\dot{\W}^{1,2}, (\dot{\W}^{1,2})'}  = 0.
\end{equation}
If there were a non-constant $\varphi \in \Wi, \varepsilon > 0$ with $\<\Hhom \varphi,
\bar\psi\>=0$ and $R(\varphi) \leq c - \varepsilon$, then
\begin{align*}
\big\< \Hs (t \varphi + s \bar\psi), (t \varphi + s \bar\psi) \big\>
&= t^2 \< \Hs \varphi, \varphi \> +s^2 \< \Hs \bar\psi,\bar\psi \> +2st \bar\mu \< \Hhom \bar\psi, \varphi \> \\
&= t^2 \< \Hs \varphi, \varphi \> +s^2 \< \Hs \bar\psi,\bar\psi \>  \\
&\leq (c-\varepsilon) \Big( t^2 \< \Hhom \varphi, \varphi \> +s^2 \< \Hhom \bar\psi,\bar\psi \> \Big)\\
&= (c-\varepsilon) \< \Hhom (t \varphi + s \bar\psi), (t \varphi + s \bar\psi) \>.
\end{align*}
Since $W := \{t \varphi + s \bar\psi : s, t \in \R \}$ is two-dimensional, there exists $w \in W \setminus \{0\}$ such that $\< w, \bar\phi \>_{\dot{\W}^{1,2}, (\dot{\W}^{1,2})'} = 0$,
a contradiction to \eqref{eq:index-1-saddle-stabilityestimate}.

{\it Step 3: Decay. }
To prove the decay of $\bar\psi$, we can write
\begin{equation*}
	(\Hs - \bar\mu \Hhom) \bar\psi = 0,
\end{equation*}
or, equivalently,
\begin{equation*}
	(1-\bar\mu) \Hhom \bar\psi = (\Hhom - \Hs) \bar\psi =: f,
\end{equation*}
where $1-\bar\mu>0$. We can rewrite the right-hand side as
\begin{align*}
   \< f, v \> &= \sum_{\ell \in \L}
      \big( \nabla^2 V(0) - \nabla^2 V(D\usa) \big) [ D\bar\psi(\ell), Dv(\ell) ] \\
   &=
   \sum_{\ell \in \L} g(\ell) \cdot Dv(\ell),
\end{align*}
with $|g(\ell)| \lesssim |D\usa(\ell)|\,|D\bar\psi(\ell)|$. An application of
\cite[Lemma 13 and Lemma 14]{EOS2016} now yields the stated decay
estimate.
\end{proof}

\def\lamneg{\bar\lambda}
\def\lamnegN{\bar\lambda_N}
\def\muneg{\bar\mu}
\def\munegN{\bar\mu_N}
\def\phineg{\bar\phi}
\def\phinegN{\bar\phi_N}
\def\psineg{\bar\psi}
\def\psinegN{\bar\psi_N}

We can now turn to the approximation results. We begin by citing a result
concerning the convergence of the displacement field. Recall that the cut-off operator $T_R$ was defined in Lemma \ref{th:TR_estimates}.

\begin{lemma} \label{th:tst:conv_usaN}
   (i) For $N$ sufficiently large there exist $\usaN \in \Usper_N$ such that $\delta
   \E_N(\usaN) = 0$ and
   \begin{equation} \label{eq:tst:rate_usaN}
      | \E_N(\usaN) - \E(\usa)|
      + \| D \usa_N - D\usa \|_{\ell^\infty} \lesssim N^{-d}.
   \end{equation}

   (ii) For $N$ sufficiently large, $\usa_N$, is an index-1 saddle, that is,
   there exists an orthogonal decomposition $\Usper_N = Q_{N,-} \oplus Q_{N,0}
   \oplus Q_{N,+}$ where $Q_{N,-} = {\rm span}\{ T_{N/2} \bar\psi \}$, $Q_{N,0}$ is the
   space of constant functions and there exists a constant $a_1>0$ such that
   \[
      \pm \< \Hdef_N(\usaN) v, v \> \geq a_1 \< \Hhom_N v, v \>
         \qquad \forall v \in Q_{N,\pm}.
   \]

   (iii) For $N$ sufficiently large, there also exists an $\ell^2$-orthogonal
   decomposition $\Usper_N = Q_{N,-}' \oplus_{\ell^2} Q_{N,0} \oplus_{\ell^2}
   Q_{N,+}'$ where $Q_{N,-} = {\rm span} \{ T_{N/2} \bar\phi \}$ and a constant
   $a_1' > 0$ such that
   \begin{align*}
      \< \Hdef_N(\usaN) v, v \>
      &\leq - a_1' \|v\|_{\ell^2}^2 \qquad \forall\, v \in Q_{N,-},
      \quad \text{and}   \\
      \< \Hdef_N(\usaN) v, v \> &\geq a_1' \| Dv\|_{\ell^2}^2
      \qquad \forall\, v \in Q_{N,+}.
   \end{align*}
\end{lemma}
\begin{proof}
   The existence of $\usaN$ and the convergence rate follows from
   \cite[Theorem 3.14]{2018-uniform}. The convergence rate for the energy
   is already contained in \cite{EOS2016}.

   The existence of the orthogonal decomposition (ii) is established in
   \cite[Lemma 3.10]{2018-uniform}. Our only claim that is not made explicit
   there is that $Q_{N,-} = {\rm span}\{ T_{N/2} \bar\phi \}$, but this is precisely
   the construction of $Q_{N,-}$ employed in the proof of \cite[Lemma
   3.10]{2018-uniform}.

   The proof of statement (iii) is very similar to the proof of (ii), following \cite{2018-uniform}.
\end{proof}

\begin{proposition}  \label{th:tst:convergence saddle}
   For $N$ sufficiently large, there exist $\phinegN, \psinegN \in \Usper_N$ and
   $\lamnegN, \munegN < 0$ such that
   \[
      \Hdef_N(\usaN) \phinegN = \lamnegN \phinegN,
      \qquad
      \Hdef_N(\usaN) \psinegN = \munegN \Hhom_N \psinegN,
   \]
   with convergence rates
   \begin{align}
      \label{eq:tst:convergence eval:phi_lam}
      \| \phinegN - \phineg \|_{\ell^2(\L_N)}
      + |\lamnegN - \lamneg| &\lesssim N^{-d}, \\
      \label{eq:tst:convergence eval:psi}
      \|D\psinegN - D\psineg \|_{\ell^2(\Lambda_N)}
      &\lesssim N^{-d/2}, \qquad \text{and} \\
      \label{eq:tst:convergence mu}
      |\munegN - \muneg| &\lesssim N^{-d}.
   \end{align}
   Moreover, there exists a constant $a > 0$, independent of $N$, such that
   \begin{align}
      \label{eq:index-1-saddle-H1-N:L2}
         \< \Hdef(\usaN) v, v \> & \geq a \< \HhomN v, v \> \qquad \text{ for } ( v, \phinegN)_{\ell^2(\Lambda_N)} = 0, \\
         \label{eq:index-1-saddle-H1-N:H1}
         \< \Hdef(\usaN) v, v \> & \geq a \< \HhomN v, v \> \qquad \text{ for } \<\HhomN v, \psinegN\> = 0.
   \end{align}
\end{proposition}
\begin{proof}
   These results follow from relatively standard perturbation arguments, hence we
   will keep this proof relatively brief. To simplify notation, let
   $\Hdef_N := \Hdef_N(\usaN)$ and $\Hdef := \Hdef(\usa)$.

   We first consider the $\ell^2$-eigenvalue problem. Let $\tilde\phi_N :=
   T_{N/2} \phineg / \|T_{N/2} \phineg\|_{\ell^2}$, then Lemma~\ref{eq:prftst:decay phi} implies
   \begin{equation} \label{eq:tst:conveval:TNphi}
      \| \tilde\phi_N - \phineg \|_{\ell^p} \lesssim e^{-c N},
   \end{equation}
   for some $c > 0$, and for all $p \in [1, \infty]$. This suggests that
   $\tilde\phi_N$ is an approximate eigenfunction; specifically, we
   can show that
   \begin{equation} \label{eq:tst:conveval:L2-residual}
      \big\| (\Hdef_N - \lamneg) \tilde\phi_N \big\|_{\ell^2}
      \lesssim N^{-d}.
   \end{equation}
   To see this, we split this residual into
   \begin{align*}
      \big\| (\Hdef_N - \lamneg) \tilde\phi_N \big\|_{\ell^2}
      &=
      \big\| (\Hdef(T_N \usaN) - \lamneg) \tilde\phi_N \big\|_{\ell^2} \\
      &\leq
      \big\| \Hdef(T_N \usaN) \tilde\phi_N -  \Hdef(\usa) \phineg \big\|_{\ell^2}
      + |\lamneg| \big\| \tilde\phi_N - \phineg \|_{\ell^2} \\
      &\leq
      \big\| \Hdef(T_N \usaN) \tilde\phi_N -  \Hdef(\usa) \phineg \big\|_{\ell^2}
      + C e^{-cN},
   \end{align*}
   where we used \eqref{eq:tst:conveval:TNphi} in the last step.
   The first term on the left-hand side can be
   readily estimated using \eqref{eq:tst:rate_usaN} to yield the
   rate \eqref{eq:tst:conveval:L2-residual}.

   We now write the $\ell^2$-eigenvalue problem as a nonlinear system,
   \[
      \< \mathcal{F}_N(\phi, \lambda), (w, \tau) \> :=
      \< \Hdef_N \phi - \lambda \phi, w \> +
      \smfrac12 (1 - \|\phi\|_{\ell^2}^2) \tau \overset{!}{=} 0,
   \]
   then \eqref{eq:tst:conveval:L2-residual} implies that
   \[
      \big|\< \mathcal{F}_N(\tilde\phi_N, \bar\lambda), (w, \tau) \> \big|
      \lesssim N^{-d} \|w\|_{\ell^2}.
   \]
   The linearisation of $\mathcal{F}_N$ is given by
   \[
      \big\< \delta\mathcal{F}_N(\phi, \lambda) (v,\varsigma), (w,\tau) \big\>
      =
      \big\< (\Hdef_N - \lambda)  v, w \big\>
      - \varsigma \< \phi, w \>_{\ell^2}
      - \tau \< \phi, v \>_{\ell^2}.
   \]
   It follows readily from Lemma~\ref{th:tst:conv_usaN}(iii), and
   \eqref{eq:tst:conveval:L2-residual} that $\delta\mathcal{F}_N(\tilde\phi_N,
   \bar\lambda)$ is a uniformly bounded isomorphism with uniformly bounded
   inverse. As also $\delta\mathcal{F}_N$ is uniformly continuous, an application of the inverse function theorem shows that there exist $\lamnegN, \phinegN$ such that
   \[
      \| \phinegN - \tilde\phi_N \|_{\ell^2}
         + |\lamnegN - \lamneg| \lesssim N^{-d}.
   \]
   This completes the proof of \eqref{eq:tst:convergence eval:phi_lam}.
   Moreover, Lemma~\ref{th:tst:conv_usaN}(iii) implies
   \eqref{eq:index-1-saddle-H1-N:L2}.

   We can now repeat the foregoing argument almost verbatim for the $\Hdef_N
   \psinegN = \munegN \Hhom_N \psinegN$ eigenvalue problem, employing Part (ii)
   instead of Part (iii) of Lemma~\ref{th:tst:conv_usaN}. The main difference is
   that the best approximation error now scales as
   \[
      \| D\psineg - DT_{N/2}\psineg \|_{\ell^2} \lesssim N^{-d/2},
   \]
   which leads to \eqref{eq:tst:convergence eval:psi}, \eqref{eq:index-1-saddle-H1-N:H1} as well as
   the suboptimal rate
   \[
      |\munegN - \muneg| \lesssim N^{-d/2}
   \]
   instead of \eqref{eq:tst:convergence mu}. To complete the proof we need
   to improve this to the optimal rate $O(N^{-d})$.

   Let $\tilde{\psi}_N := T_{N/2} \psineg$. Convergence of $\psinegN$, \eqref{eq:tst:convergence eval:psi}, implies that
   $\< \Hhom \psinegN, \tilde{\psi}_N \> \to 1$ as $N \to \infty$; hence, we
   can estimate
   \begin{align*}
      (\muneg - \munegN) \< \Hhom_N \psinegN, \tilde{\psi}_N \>
      &=
      \muneg \< \Hhom T_N \psinegN, \tilde{\psi}_N \> - \< \Hdef_N(\usaN) \psinegN, \tilde{\psi}_N \> \\
      &\hspace{-4cm} =
      \muneg \< \Hhom (\tilde{\psi}_N - \psineg), T_N \psinegN \>
         + \< \Hdef(\usa) \psineg, T_N \psinegN \>
         - \< \Hdef_N(\usaN) \psinegN, \tilde{\psi}_N \> \\
      &\hspace{-4cm} = \Big\{\muneg \< \Hhom (\tilde{\psi}_N - \psineg), T_N \psinegN \>
            - \< \Hdef(\usa) (\tilde{\psi}_N - \psineg),  T_N \psinegN \> \Big\} \\
      &
         + \Big\{ \< \Hdef(\usa) \tilde{\psi}_N, T_N \psinegN \>
         - \< \Hdef_N(\usaN) T_N \psinegN, \tilde{\psi}_N \> \Big\} \\
      &\hspace{-4cm} =: {\bf A}_1 + {\bf A}_2.
   \end{align*}
   The first term is readily bounded by
   \begin{align*}
      \big|{\bf A}_1 \big|
      &= \Big| \muneg \< \Hhom (\tilde{\psi}_N - \psineg), (T_N \psinegN - \psineg) \>
            - \< \Hdef(\usa) (\tilde{\psi}_N - \psineg),  (T_N \psinegN - \psineg) \> \Big| \\
      &\lesssim \| D\tilde{\psi}_N - D\psineg \|_{\ell^2} \|DT_N \psinegN - D\psineg \|_{\ell^2}
      \lesssim N^{-d/2} N^{-d/2} \lesssim N^{-d}.
   \end{align*}
   The second term is best written out in detail,
   \begin{align*}
      \big| {\bf A}_2 \big|
      &= \bigg| \sum_{\ell \in \L_{N/2}}
         \Big\< \big[\nabla^2 V_\ell(D\usa(\ell)) - \nabla^2 V_\ell(D\usaN(\ell)) \big]
            D\tilde{\psi}_N(\ell), D \psinegN(\ell) \Big\> \bigg|  \\
      &\lesssim \| D\usa - D \usaN \|_{\ell^\infty(\L_{N/2})}
         \| D\tilde{\psi}_N \|_{\ell^2} \| D\psinegN \|_{\ell^2}
      \lesssim N^{-d}.
   \end{align*}
   This establishes \eqref{eq:tst:convergence mu} and thus completes
   the proof.
\end{proof}

\begin{proof}[Proof of Theorem~\ref{th:main thm saddle}(1)]
   This is included in Proposition \ref{th:tst:convergence saddle}.
\end{proof}

\subsection{Convergence of the transition rate}
\label{sec:tst:k}
\def\SSs{\SS^{\rm +}}
\def\SSsN{\SS^{\rm +}_N}
We can now turn to the analysis of the transition rate,
\begin{align}
   \label{eq:defn-tst}
   \khtst_N &:= \exp\Big( - \beta \Delta  \F_N \Big) :=
   \exp\Big( - \beta \big( \Delta \E_N - \beta^{-1} \Delta \SS_N \big) \Big),
   \qquad \text{where}
   \\
   \notag
   \Delta\E_N &:= \E_N(\usaN) - \E_N(\us_N), \qquad \text{and}
   \\
   \notag
   \Delta\SS_N &:= \SS_N(\usaN) - \SS_N(\us_N)\\
   \notag &= -\smfrac{1}{2} \log \detp \HsN + \smfrac{1}{2} \log \detp H_N \\
   \notag &=
   -\smfrac{1}{2} \sum \log\lambda_j^{\rm saddle}
   +\smfrac{1}{2} \sum \log\lambda_j^{\rm min},
\end{align}
where $\lambda_j^{\rm min}$ and $\lambda_j^{\rm saddle}$ enumerate the positive eigenvalues of, respectively, $\Hdef_N$ and $\HsN$. We already know from Theorem~\ref{thm:pbc convergence} and
Proposition~\ref{th:tst:convergence saddle} that
\begin{equation} \label{eq:tst:convergence Delta E}
   |\Delta\E_N - \Delta \E| \lesssim N^{-d}
   \quad \text{where} \quad
   \Delta \E := \E(\usa) - \E(\us).
\end{equation}
From Theorem~\ref{theo: main result} we
know that
\begin{equation} \label{eq:tst:convergence SNmin}
   \big|\SS_N(\us_N) - \SS(\us)\big| \lesssim N^{-d} \log^5 N,
\end{equation}
hence, it now only remains to characterise the limit $\SS_N(\usaN) \to \SS(\usa)$ and
estimate the rate of convergence. Again, we want to use a localisation argument. To that end, we first rewrite $\SS_N(\usaN)$ in a way that then allows us to exploit
the functional calculus framework that we developed in the prior sections.
This will require us to consider the logarithm of negative numbers. Let us therefore look at the branch of the complex logarithm given by
\[
   \log r e^{i \varphi} := \log r + i \varphi, \qquad \text{for } r > 0,\ \varphi \in (-\pi/2,3\pi/2).
\]
The logarithm of a finite-dimensional, invertible, self-adjoint operator with spectral decomposition $A = \sum_j \alpha_j v_j \otimes v_j$, is then given by
\[
      \log A := \sum_j \log \alpha_j v_j \otimes v_j.
\]
Recalling the definitions of $\bar\lambda_N, \bar\mu_N$ from
Proposition~\ref{th:tst:convergence saddle}
and of $\Fop_N$ and $\pi_N$ from \S~\ref{sec:results:decompositionS}, we calculate
\begin{align}
   \notag
   \SS_N(\usaN) &=
   -\smfrac{1}{2} \sum \log\lambda_j^{\rm saddle}
   +\smfrac{1}{2} \sum \log\lambda_j^{\rm hom}\\
   \notag
   &= -\smfrac{1}{2}{\rm Trace} \log (\HsN + \pi_N) +\smfrac{1}{2} {\rm Trace} \log (\Hhom_N + \pi_N)  + \smfrac{1}{2} \log \bar\lambda_N\\
   \notag
   &= -\smfrac{1}{2}\log \frac{\det (\HsN + \pi_N)}{\det (\Hhom_N + \pi_N) }  + \smfrac{1}{2} \log \bar\lambda_N\\
   \notag
   &= -\smfrac{1}{2}\log \det (\Fop_N+\pi_N) (\HsN + \pi_N) (\Fop_N+\pi_N)  + \smfrac{1}{2} \log \bar\lambda_N \\
   \notag
   &= -\smfrac{1}{2}{\rm Trace} \log (\Fop_N\HsN\Fop_N+\pi_N) + \smfrac{1}{2} \log \bar\lambda_N \\
   \notag
   &= -\smfrac{1}{2}{\rm Trace} \logp\, (\Fop_N\HsN\Fop_N) - \smfrac{1}{2} \log \bar\mu_N + \smfrac{1}{2} \log \bar\lambda_N\\
   \label{eq:tst:SsN splitting}
   &= \sum_{\ell} \mathcal{S}^+_{N,\ell}(\usa_N) - \smfrac{1}{2} \log \bar\mu_N + \smfrac{1}{2} \log \bar\lambda_N,
\end{align}
based on the definition of $\mathcal{S}^+_{N,\ell}$ in \eqref{eq:defSNell2}. Note, that the formula $\log \det A = {\rm Trace} \log A$ is still true for the complex logarithm as there is only one negative eigenvalue. Otherwise a correction by a multiple of $2 \pi i$ would have been needed.

We already know that $\bar\mu_N, \bar\lambda_N$ converge as $N \to \infty$, and
from the definition of the complex logarithm we immediately also obtain
that
\begin{equation} \label{eq:tst:convergence log e-vals}
   | \log \bar\lambda_N - \log \bar\lambda |
   + |\log \bar\mu_N - \log \bar\mu| \lesssim N^{-d}.
\end{equation}

Finally, we must address the group
\[\mathcal{S}^+_{N}(\usa_N) := \sum_{\ell} \mathcal{S}^+_{N,\ell}(\usa_N).\]
Due to Proposition~\ref{th:tst:convergence saddle}, we have indeed $\sigma (\Fop_N\HsN\Fop_N) \setminus \{0,\bar\mu_N\} \subset [\underline\sigma,\overline\sigma]$ for $\underline\sigma>0$ small enough and $\overline\sigma>0$ large enough.

To define the limit, recall that for $u$ satisfying \eqref{eq:logpluscondition}
\[\SS^+_\ell(u) = -\smfrac{1}{2} {\rm Trace} \logp\,(\Fop \Hdef(u) \Fop)_{\ell \ell}.\]
As  $\usa \in \mathcal{U}$, we can apply \eqref{eq:decompose S+}
and Proposition \ref{prop: localitydecomposedentropy} to see that
\[\Big\lvert \SS^+_\ell(\usa) - \langle \delta \Shom_\ell(0),\usa \rangle \Big\rvert \lesssim \lvert \ell \rvert_{l^2}^{-2d},\]
and thus
\[\mathcal{S}^+(\usa) = \sum_{\ell \in \L} \Big(\SS_\ell(\usa) - \langle \delta \Shom_\ell(0),\usa \rangle \Big)\]
is well-defined.

\begin{lemma} \label{th:tst:convergence S+}
   For $N$ sufficiently large, let $\usaN$ be given by
   Proposition~\ref{th:tst:convergence saddle}, then
\begin{equation*}
\lvert \mathcal{S}^+(\usa) - \mathcal{S}^+_{N}(\usa_N) \rvert \lesssim N^{-d} \log^5 N.
\end{equation*}
\end{lemma}
\begin{proof}
   Setting $u_N := \usa_N$ and $u_\infty := \usa$,
   \eqref{eq:uNuinftyAssumptions} is satisfied. Therefore, this result is a
   consequence of Proposition \ref{prop:generalconvergencerate}.
\end{proof}
We can now define
\begin{align}
   \label{eq:defktst}
   \ktst &:= \exp\Big( - \beta \big( \Delta \E - \beta^{-1} \Delta \SS \big) \Big),
   \qquad \text{where}
   \\
   \notag
   \Delta\E &:= \E(\usa) - \E(\us), \qquad \text{and}
   \\
   \notag
   \Delta\SS &:= \SS^+(\usa) - \SS(\us) - \smfrac{1}{2} \log \lvert \bar\mu \rvert + \smfrac{1}{2} \log \lvert \bar\lambda \rvert
\end{align}

\begin{proof}[Proof of Theorem \ref{th:main thm saddle} (2)]
   According to
   \eqref{eq:tst:convergence Delta E},
   \eqref{eq:tst:convergence SNmin},
   \eqref{eq:tst:convergence log e-vals},
   and Lemma~\ref{th:tst:convergence S+}. we have
   \[ |\Delta\E_N - \Delta \E| \lesssim N^{-d}, \quad |\Delta\SS_N - \Delta \SS| \lesssim N^{-d}\log^5(N). \]
   Using $\Delta \E>0$, we have
   \begin{align*}
   \lvert \ktst - \ktst_N \rvert &\lesssim |\Delta\SS_N - \Delta \SS| + |\Delta\E_N - \Delta \E| \sup_\beta \sup_{x \in [\Delta\E/2, 2 \Delta\E]}  \beta e^{- \beta x}\\
   &\lesssim N^{-d} \log^5(N) + N^{-d} \frac{2}{e \Delta\E}\\
   &\lesssim N^{-d} \log^5(N).
   \end{align*}
\end{proof}

%% file: appendix.tex

\section{Appendix}
\label{sec: appendix}

\subsection{Proof of Lemma \ref{th:properties_F}}
\label{sec:proof_properties_F}
{\it Preliminaries: } Recall from \eqref{eq:F-transform_Hhom} the Fourier
representation of $\Hhom$. Expanding $\hat{h}(k)$ in \eqref{eq:fourierhamiltoniansin} as $k \to 0$ yields the
continuum (long wave-length) limit
\[
   \hat{h}^{\rm c}(k) = \sum_{\rho \in \Rc'} A_\rho (k \cdot \rho)^2,
\]
which is the symbol of a linear elliptic PDE operator of a linear elliptic
operator of the form
\[
   H^{\rm c} u := - \div \A \nabla u,
\]
where $\A$ is a fourth-order tensor and {\bf (STAB)} implies that it satisfies
the strong Legendre--Hadamard condition \cite{EOS2016,2012-M2AN-CBstab},
\[
   \sum_{\alpha,\beta,i,j} \A^{\alpha\beta}_{ij}\eta_i\eta_j\xi^\alpha\xi^\beta\geq c_0 |\eta|^2 |\xi|^2 \qquad \forall \eta, \xi \in \R^d,
\]
for some $c_0 > 0$.

Let $\hat{F}^{\rm c}(k) := [\hat{h}^{\rm c}]^{-1/2}$, then $\hat{F}^{\rm c} \in
C^\infty(\R^d \setminus \{0\})$ and it is $(-1)$-homogeneous. It now follows from
\cite[Theorem 6.2.1]{Morrey66}  (see also \cite{BHOdefectdevelopment} for a more detailed enactment of
Morrey's argument specific to our setting) that there exists $F^{\rm c} \in C^\infty(\R^d \setminus
\{0\})$ with symbol $\hat{F}^{\rm c}$ such that $F^{\rm c}$ is $(1-d)$-homogeneous. In particular,
\begin{equation}
\label{eq: continuous F}
   |\nabla^j F^{\rm c}(x)| \leq C |x|^{1-d-j} \quad \text{for } j \geq 0.
\end{equation}

We can now use the sharp decay bounds on $F^{\rm c}$ and the connection between
the symbols $\hat{F}(k)$ and $\hat{F}^{\rm c}(k)$ to modify the arguments
from \cite{EOS2016,OO2016}, to estimate the decay of $F$ as well.

\begin{proof}[Proof of Lemma \ref{th:properties_F}(i): decay estimates]
   Let $\hat\eta(k)\in C_c^\infty(\B)$ with $\hat\eta(k)=1$ in a neighbourhood
   of the origin. Then its inverse Fourier transform $\eta:=\F^{-1}[\hat\eta]\in
   C^\infty(\R^d)$ has super-algebraic decay \cite{book-Tref}. Therefore,
   $\eta \ast F^{\rm c}$ is well-defined,
   \begin{equation} \label{eq: aux est 1}
      |D_{\brho} (\eta \ast F^{\rm c})(\ell)|
      \leq C |\ell|_{l^0}^{1-d-j}
   \end{equation}
   and $\mathcal{F}[\eta \ast F^{\rm c}] = \hat\eta \hat{F}^{\rm c}$ is
   compactly supported in ${\rm BZ}$ and smooth except at the origin.

   Next we show that
   \begin{equation}
   \label{eq: aux est 2}
   |D_{\brho} (F - \eta\ast F^{\rm c})(\ell)|\leq C|\ell|_{l^0}^{-d-j},
   \end{equation}
   which, together with \eqref{eq: aux est 1}, implies the stated result.

   From the explicit representation of $\hat h(k)$ and $\hat h^{\rm c} (k)$ we have
   \begin{equation*}
      \big| |k|^{-2} \hat h(k) - |k|^{-2} \hat h^{\rm c} (k)\big| \leq C |k|^2.
   \end{equation*}
   Recall the {\bf (STAB)} implies that $|k|^{-2} \hat h(k)$
   and $|k|^{-2} \hat h^{\rm c} (k)$ are bounded above and below
   in ${\rm BZ}$, hence
   \begin{align*}
      \Big| |k|\, \hat h(k)^{-1/2} -
         |k| \hat h^{\rm c}(k)^{-1/2} \Big| &\leq C |k|^2,  \\
      \text{or, equivalently, } \qquad
      \big| \hat{F}(k) - \hat{F}^{\rm c}(k) \big| &\leq C |k|.
   \end{align*}
   Along similar lines, we can prove that
   \begin{equation*}
      \big| \nabla^m \hat F (k) - \nabla^m \hat F^{\rm c}(k) \big|
      \lesssim  |k|^{1-m}.
   \end{equation*}
   Applying \cite[Theorem 7 \& Corollary 8]{OO2016}, this implies \eqref{eq: aux
   est 2}.
\end{proof}

\begin{proof}[Proof of Lemma \ref{th:properties_F}(ii),(iii):]
   We need to show that $\Fop \colon \ell^2 \to \Wi$. Let $v
   \in \ell^2$. For a fixed $\ell$,
   \[\lvert F(l-m) -F(-m)\rvert \lesssim \lvert m \rvert_{l^0}^{-d}\]
   due to the decay for $DF$ established in part (i). Therefore, $F(l-\cdot) -F(-\cdot) \in \ell^2$ and $\Fop v(\ell)$ is defined for all $\ell$. Clearly, we also have $\Fop v(0)=0$.

   For any $\rho$ we find
   \[D_\rho (\Fop v)(\ell) = \sum_m  D_\rho F(\ell-m) v(m).\]
The Plancherel theorem then implies
\begin{align*}
D_\rho (\Fop v)(\ell) = \frac{1}{\lvert \B \rvert}\int_\B (e^{ik\cdot\rho}-1) \hat{F}(k) \hat{v}(k) e^{-ik\cdot \ell}\,dk
\end{align*}
As the Fourier-multiplier satisfies $(e^{ik\cdot\rho}-1) \hat{F}(k) \in L^\infty(\B)$, we find $D_\rho (\Fop v) \in\ell^2$ and thus $\Fop v \in \Wi$.

For $v,w \in \ell^2$ we calculate
   \begin{align*}
      \< \Fopd \Hhom \Fop v, w \>_{\ell^2}
      &= \< \Hhom (\Fop v), (\Fop w) \>_{(\Wi)',\Wi} \\
      &= \sum_\ell \nabla^2V(0)[D(\Fop v)(\ell), D(\Fop w)(\ell)] \\
      &= \frac{1}{\lvert \B \rvert}\int_{\rm \B} \nabla^2V(0)[\overline{((e^{ik\cdot\rho}-1) \hat{F}(k)\hat{v}(k))_{\rho \in \mathcal{R}}},((e^{ik\cdot\rho}-1) \hat{F}(k)\hat{w}(k))_{\rho \in \mathcal{R}}] \, dk \\
      &= \frac{1}{\lvert \B \rvert}\int_{\rm \B} (\hat{F} \hat{v})^* \hat{h} \hat{F} \hat{w} \, dk \\
      &= \frac{1}{\lvert \B \rvert}\int_{\rm \B} \hat{v}^* \hat{w} \, dk \\
      &=  \< v, w \>_{\ell^2},
   \end{align*}
   which proves (iii).
As
\[ \lVert Dw \lVert_{\ell^2}^2 \lesssim \<  \Hhom w, w \>_{\ell^2} \lesssim \lVert Dw \lVert_{\ell^2}^2 \]
for all $w \in\Wi$ according to \eqref{eq:stab homogeneous}, we can set $w= \Fop v$ to find
\[ \lVert D\Fop v \lVert_{\ell^2}^2 \lesssim \< \Fopd \Hhom \Fop v, v \>_{\ell^2} = \lVert v \rVert_{\ell^2} \lesssim \lVert D\Fop v \lVert_{\ell^2}^2. \]
In particular, $\Fop$ is one-to-one and continuous.
\end{proof}

\subsection{Proof of Lemma \ref{lem: finite-rank correction}}
\label{sec:proof:finite-rank correction}
We use arguments similar to those in \cite{Seg92}. Let us
start with the finite-dimensional case. For an $r \times r$ matrix $B$ let
\[
p_{B}(\lambda):=\det (\lambda I-B)=\lambda^r+c_1\lambda^{r-1}+\ldots+ c_{r-1}\lambda + c_r
\]
be the characteristic polynomial of $B$. The coefficients $c_k$ are of the form
\[
   c_k = c_k(B)= \rm tr(\Lambda^k B),
\]
where $\Lambda^k B$ is the $k$-th exterior power of $B$, i.e., a homogeneous
degree $k$ polynomial in the coefficients of $B$ which can be written as a sum of minors.

If $I+B$ is invertible, then
\[
\alpha:=p_B(-1)=(-1)^{r}+c_1(-1)^{r-1}+\ldots+c_r\neq 0.
\]
Therefore, there is a polynomial
 \[
\bar{p}_B(\lambda)=\lambda^{r}+\bar{c}_{1}\lambda^{r-1}+\ldots+\bar{c}_{r-1}\lambda + \alpha
\]
such that $\lambda p_B(\lambda)+\alpha =(1+\lambda)\bar{p}_B(\lambda)$. Indeed, the coefficients are given as
\[
\bar{c}_1=c_1-1, \bar{c}_2=c_2-\bar{c}_1,\ldots,\bar{c}_k=c_k-\bar{c}_{k-1},\ldots,\bar{c}_{r-1}=c_{r-1}-\bar{c}_{r-2}
\]
i.e.,
\[
\bar{c}_k=(-1)^k+\sum_{j=1}^k (-1)^{k-j}c_j.
\]
According to the Cayley-Hamilton theorem, $p_B(B)=0$. Therefore, $\alpha= (I + B) \bar{p}_B(B)$. Hence,
\begin{align}
(I+B)^{-1}&=\frac{1}{\alpha}\bar{p}_B(B)\nonumber
\\&=\frac{1}{\alpha}\Big(B^{r}+\bar{c}_1 B^{r-1}+\ldots+\bar{c}_{r-1}B + \alpha\Big)\nonumber
\\&=I+\frac{1}{\alpha}\Big(B^{r}+\bar{c}_1 B^{r-1}+\ldots+\bar{c}_{r-1}B\Big)\nonumber
\\&=I+\frac{B^{r}+\bar{c}_1 B^{r-1}+\ldots+\bar{c}_{r-1}B}{(-1)^{r}+c_1(-1)^{r-1}+\ldots+c_r}\nonumber
\\&=I + \sum_{k=1}^r \tilde{c}_k B^k.\label{eq: finite dimensional inverse}
\end{align}
A representation as desired with coefficients $\tilde{c}_k =\tilde{c}_k(B)$ depending continuously on $B$.

Now let us discuss the general case. We will immediately prove the main statement and (ii), as (i) is clearly a special case of (ii).

So let $X$ be a Hilbert space with orthogonal decomposition $X = X_1 \oplus X_2$ such that ${\rm dim} (X_1) \leq r$ and $X_2 \subset {\rm ker} A $ for an operator $A$. If $P_V: X \to X$ is the orthogonal projection onto $V$, we can the operators as $A=P_{X_1}A P_{X_1} + P_{X_2}A P_{X_1}$, as $P_{X_1}A P_{X_2}=P_{X_2}A P_{X_2}=0$. Let us write $B \colon X_1 \to X_1$ and $C \colon X_1 \to X_2$ for these restricted and projected operators. That means we have
\[A= \iota_{X_1} B \pi_{X_1} +\iota_{X_2} C  \pi_{X_1},\]
where $\iota_{X_i}\colon X_i \to X$ and $\pi_{X_i}\colon X_i \to X$ are the standard embedding and orthogonal projection. In particular, for $j \geq 1$ we have
\[A^j = A \iota_{X_1} B^{j-1} \pi_{X_1}.\]
If $I +A$ is invertible, then so is $I_{X_1} + B$ as $(I_{X_1} + B)^{-1} = \pi_{X_1} (I + A)^{-1}  \iota_{X_1}$.
We can also represent $(I + A)^{-1}$ in terms of  $(I_{X_1} + B)^{-1}$ as a block inverse by
\[(I + A)^{-1} = \iota_{X_1} (I_{X_1} + B)^{-1} \pi_{X_1} - \iota_{X_2} C(I_{X_1} + B)^{-1} \pi_{X_1} + \iota_{X_2}\pi_{X_2}.\]
In particular,
\begin{align*}
(I + A)^{-1} - I &= \iota_{X_1} \big( (I_{X_1} + B)^{-1}-I\big) \pi_{X_1} - \iota_{X_2} C(I_{X_1} + B)^{-1} \pi_{X_1}\\
&= \big(\iota_{X_1}- \iota_{X_2} C\big) \big( (I_{X_1} + B)^{-1}-I_{X_1}\big) \pi_{X_1} - \iota_{X_2} C \pi_{X_1}\\
&= \big(-\iota_{X_1}B- \iota_{X_2} C\big) \big( (I_{X_1} + B)^{-1}-I_{X_1}\big) \pi_{X_1} - \iota_{X_1}B \pi_{X_1} - \iota_{X_2} C \pi_{X_1}\\
&= -A \iota_{X_1} \big( (I_{X_1} + B)^{-1}-I_{X_1}\big) \pi_{X_1} - A\\
\end{align*}

According to \eqref{eq: finite dimensional inverse} we have
\[(I_{X_1}+B)^{-1}-I_{X_1} = \sum_{k=1}^r \tilde{c}_k B^k,\]
and hence,
\[A \iota_{X_1}\Big((I_{X_1}+B)^{-1}-I_{X_1}\Big)\pi_{X_1} = \sum_{k=1}^r \tilde{c}_k A^{k+1}.\]
Overall we have,
\[(I+A)^{-1} = I + \sum_{k=1}^{r+1} \hat{c}_k A^k \]
with $\hat{c}_1=-1$ and $\hat{c}_k = -\tilde{c}_{k-1}$ for $k \geq 2$. In particular, for a family $(A_\alpha)_\alpha$ of operators with the same orthogonal decomposition of $X$, the $\hat{c}_k$ are given as continuous functions of $B_\alpha=\pi_{X_1} A_\alpha \iota_{X_1} \in L(X_1)$.

\subsection{Auxiliary Estimates} \label{sec: auxiliary estimates}
We want to collect a few auxiliary estimates for certain sums that appear in a number of variations throughout.
\begin{lemma}
\label{lem: auxiliary estimates}
All the implied constants in the following are allowed to depend on the exponents $\alpha, \beta, \gamma, p$, as well as the dimension $d$, but not on the lattice points $n,m \in \L$, or the cut-off $M\geq 0$.
\begin{align}
\sum_{\ell \in \L} \lvert \ell \rvert_{l^\alpha, M}^{-p-d} &\lesssim \lvert M \rvert_{l^\alpha}^{-p}\quad  \text{ for all } p>0,\, \alpha\geq 0. \label{eq: auxiliary estimates eq1} \\
\sum_{\ell \in \L, \lvert \ell \rvert \leq M} \lvert \ell \rvert_{l^\alpha}^{-d} &\lesssim \lvert M \rvert_{l^{\alpha+1}}^{0}\quad  \text{ for all } \alpha\geq 0. \label{eq: auxiliary estimates eq2} \\
\sum_{\ell \in \L} \lvert \ell \rvert_{l^\alpha,M}^{-d}\lvert \ell -m \rvert_{l^\beta}^{-d} &\lesssim \lvert m \rvert_{l^{\alpha+\beta+1},M}^{-d}\quad  \text{ for all } \alpha, \beta\geq 0,\, m \in \L. \label{eq: auxiliary estimates eq3} \\
\sum_{\ell \in \L} \lvert \ell \rvert_{l^\alpha}^{-d}\lvert \ell -m \rvert_{l^\beta,M}^{-d-p} &\lesssim \lvert m \rvert_{l^\alpha,M}^{-d} \lvert M \rvert_{l^{\beta+1}}^{-p} \quad  \text{ for all } \alpha, \beta\geq 0,\, p>0,\, m \in \L. \label{eq: auxiliary estimates eq4} \\
\sum_{\ell \in \L} \lvert \ell \rvert_{l^\alpha}^{-d-p}\lvert \ell -m \rvert_{l^\beta}^{-d-p} &\lesssim \lvert m \rvert_{l^{\alpha}}^{-d-p}\quad  \text{ for all } \alpha \geq \beta\geq 0,\, p>0,\, m \in \L. \label{eq: auxiliary estimates eq5} \\
\sum_{\ell \in \L} \lvert \ell \rvert_{l^\alpha,M}^{-d} \lvert \ell -m \rvert_{l^\beta}^{-d} \lvert \ell-n \rvert_{l^\gamma}^{-d} &\lesssim  \lvert n \rvert_{l^\alpha,M}^{-d} \lvert m-n \rvert_{l^{\beta + \gamma+1}}^{-d} +  \lvert n \rvert_{l^\gamma,M}^{-d} \lvert m \rvert_{l^{\alpha+\beta +1},M}^{-d}\nonumber \\
 &\qquad  \text{ for all } \alpha,\beta, \gamma \geq 0,\, m,n \in \L \text{ with } \lvert n \rvert  \geq \lvert m \rvert. \label{eq: auxiliary estimates eq6}\\
\sum_{\ell \in \L} \lvert \ell \rvert_{l^\alpha,M}^{-d-p} \lvert \ell -m \rvert_{l^\beta}^{-d} \lvert \ell-n \rvert_{l^\gamma}^{-d} &\lesssim \lvert n \rvert_{l^\alpha,M}^{-d-p} \lvert m-n \rvert_{l^{\beta + \gamma+1}}^{-d} + \lvert n \rvert_{l^\gamma,M}^{-d} \lvert m \rvert_{l^{\beta},M}^{-d} \lvert M \rvert_{l^{\alpha+1}}^{-p}\nonumber \\
&\qquad  \text{ for all } \alpha,\beta, \gamma \geq 0,\, m,n \in \L \text{ with } \lvert n \rvert  \geq \lvert m \rvert. \label{eq: auxiliary estimates eq7}
\end{align}
As a special case, note that one can always take $M=0$, where one finds $\lvert \ell \rvert_{l^\alpha, M}^{-q} = \lvert \ell \rvert_{l^\alpha}^{-q}$ and $\lvert M \rvert_{l^\alpha}^{-q} = 1$.
\end{lemma}
\begin{corollary}
\label{cor: auxiliary estimates part2}
In particular, if follows that
\begin{align}
\sum_{\ell \in \L} \lvert \ell \rvert_{l^1}^{-d} \lvert \ell -m \rvert_{l^1}^{-d} \lvert \ell-n \rvert_{l^\gamma}^{-d} &\lesssim \Lr_{\gamma+2}(m,n) \quad  \text{ for all } \gamma\geq 0,\, n,m \in \L. \label{eq: auxiliary estimates eq6cor1}\\
\sum_{\ell \in \L, \lvert \ell \rvert \geq M} \lvert \ell \rvert_{l^0}^{-d} \lvert \ell -m \rvert_{l^0}^{-d} \lvert \ell-n \rvert_{l^0}^{-d} &\lesssim \Lr_{1}^M(m,n) \quad  \text{ for all } n,m \in \L.\label{eq: auxiliary estimates eq6cor2}\\
\sum_{\ell \in \L} \lvert \ell \rvert_{l^1, M}^{-d} \lvert m-\ell \rvert_{l^1}^{-d} \lvert n-\ell \rvert_{l^1}^{-d} &\lesssim \lvert m\rvert_{l^1,M}^{-d} \lvert m-n\rvert_{l^3}^{-d} + \lvert m\rvert_{l^1,M}^{-d} \lvert n\rvert_{l^3,M}^{-d} \quad  \text{ for all } n,m \in \L. \label{eq: auxiliary estimates eq6cor3}\\
\sum_{\ell \in \L}  \Lr_1(m,\ell) \lvert n-\ell \rvert_{l^\gamma}^{-d} &\lesssim  \Lr_{\gamma+2}(m,n) \lesssim \lvert m-n \rvert_{l^{\gamma+2}}^{-d} \quad  \text{ for all } \gamma\geq 0,\, n,m \in \L. \label{eq: auxiliary estimates eq6cor4}\\
 \sum\limits_{\ell\in \L}\lvert \ell \rvert_{l^2,M}^{-2d} \lvert \ell-m \rvert_{l^1}^{-d} \lvert \ell-n \rvert_{l^1}^{-d} &\lesssim  \lvert n\rvert_{l^1,M}^{-d} \lvert m\rvert_{l^1,M}^{-d} ( \lvert M\rvert_{l^3}^{-d} + \lvert m-n\rvert_{l^3}^{-d}) \label{eq: auxiliary estimates eq7cor1}.
\end{align}
\end{corollary}
\begin{proof}
To show \eqref{eq: auxiliary estimates eq6cor1} just note that we can estimate  $\lvert m \rvert_{l^{3}}^{-d} \lesssim \lvert m \rvert_{l^{2}}^{-d} \lvert n \rvert_{l^{1}}^{0}$ in \eqref{eq: auxiliary estimates eq6} for the case $\lvert n \rvert \geq \lvert m \rvert$. If on the other hand $\lvert m \rvert \geq \lvert n \rvert$, \eqref{eq: auxiliary estimates eq6} becomes
\[\sum_{\ell \in \L} \lvert \ell \rvert_{l^1}^{-d} \lvert \ell -m \rvert_{l^1}^{-d} \lvert \ell-n \rvert_{l^\gamma}^{-d} \lesssim  \lvert m \rvert_{l^1}^{-d} \lvert m-n \rvert_{l^{\gamma+2}}^{-d} +  \lvert m \rvert_{l^1}^{-d} \lvert n \rvert_{l^{\gamma +2}}^{-d}\]
which already gives the result.

\eqref{eq: auxiliary estimates eq6cor2} follows directly from \eqref{eq: auxiliary estimates eq6} as it is symmetric in $m$, $n$.

\eqref{eq: auxiliary estimates eq6cor3} directly follows from \eqref{eq: auxiliary estimates eq6} and its version with $m$, $n$ reversed, so \eqref{eq: auxiliary estimates eq6cor3} holds true for all $m,n \in \L$.

The first inequality in \eqref{eq: auxiliary estimates eq6cor4} is just a combination of \eqref{eq: auxiliary estimates eq6cor1} and \eqref{eq: auxiliary estimates eq3}. The second follows from $\lvert n-m \rvert_{l^0} \leq \lvert n \rvert_{l^0} \lvert m \rvert_{l^0}$.

\eqref{eq: auxiliary estimates eq7cor1} immediately follows from \eqref{eq: auxiliary estimates eq7} as it is symmetric in $m$, $n$.
\end{proof}

\begin{proof}[Proof of Lemma \ref{lem: auxiliary estimates}]
Let us start with \eqref{eq: auxiliary estimates eq1}. The statement is trivial if the sum is restricted to $\lvert \ell \rvert \leq M$. At the same time,
\[
\sum_{ \lvert \ell \rvert > M} \lvert \ell \rvert_{l^\alpha}^{-p-d} \lesssim \int_M^{\infty} \lvert r \rvert_{l^\alpha}^{-p-1}\,dr \lesssim  \lvert M \rvert_{l^\alpha}^{-p}.
\]
For \eqref{eq: auxiliary estimates eq2}, we estimate
\[
\sum_{ \lvert \ell \rvert \leq M} \lvert \ell \rvert_{l^\alpha}^{-d} \lesssim \log(e+M)^{\alpha} \int_1^{M+2} \lvert r \rvert^{-1}\,dr \lesssim  \lvert M \rvert_{l^{\alpha+1}}^{0}.
\]
In \eqref{eq: auxiliary estimates eq3}, first consider $\lvert m \rvert \leq M$. Then we can split the sum and estimate
\begin{align*}
\sum_{\ell} \lvert \ell \rvert_{l^\alpha,M}^{-d}\lvert \ell -m \rvert_{l^\beta}^{-d} &\lesssim \sum_{\lvert \ell \rvert > 2M} \lvert \ell \rvert_{l^{\alpha+\beta},M}^{-2d} + \lvert M \rvert_{l^\alpha}^{-d} \sum_{\lvert \ell \rvert \leq 2M} \lvert \ell -m \rvert_{l^\beta}^{-d}\\
&\leq \lvert M \rvert_{l^{\alpha+\beta+1}}^{-d},
\end{align*}
according to \eqref{eq: auxiliary estimates eq1} and \eqref{eq: auxiliary estimates eq2}. On the other hand, the case $\lvert m \rvert > M$ follows directly if we can show the entire statement for $M=0$. Splitting up the sum, we find
\begin{align*}
\sum_{\ell} \lvert \ell \rvert_{l^\alpha}^{-d}\lvert \ell -m \rvert_{l^\beta}^{-d} &\lesssim  \sum_{\ell \in B_{\frac{\lvert m \rvert}{3}}(0)} \lvert \ell \rvert_{l^\alpha}^{-d}\lvert m \rvert_{l^\beta}^{-d} + \sum_{\ell \in B_{\frac{\lvert m \rvert}{3}}(m)} \lvert m \rvert_{l^\alpha}^{-d}\lvert \ell -m \rvert_{l^\beta}^{-d}\\
&\qquad + \sum_{\ell \in B_{2\lvert m \rvert}(0)^c} \lvert \ell \rvert_{l^{\alpha+\beta}}^{-2d} + \lvert m \rvert_{l^{\alpha+\beta}}^{-2d} \lvert m \rvert_{l^{0}}^{d}\\
&\lesssim \lvert m \rvert_{l^{\alpha+\beta+1}}^{-d},
\end{align*}
according to \eqref{eq: auxiliary estimates eq1} and \eqref{eq: auxiliary estimates eq2}.

Now let us look at \eqref{eq: auxiliary estimates eq4}. First consider the case $\lvert m \rvert \leq M$. Then
\begin{align*}
\sum_{\ell} \lvert \ell \rvert_{l^\alpha}^{-d}\lvert \ell -m \rvert_{l^\beta,M}^{-d-p} &\lesssim \sum_{\ell \in B_{2M}(m)} \lvert \ell \rvert_{l^\alpha}^{-d}\lvert M \rvert_{l^\beta}^{-d-p} +  \sum_{\ell \in B_{2M}(m)^c} \lvert \ell \rvert_{l^\alpha}^{-d}\lvert \ell \rvert_{l^\beta}^{-d-p}\\
&\lesssim \lvert M \rvert_{l^{\alpha + \beta+1}}^{-d-p}\\
&\lesssim  \lvert m \rvert_{l^\alpha,M}^{-d} \lvert M \rvert_{l^{\beta+1}}^{-p}.
\end{align*}
If on the other hand $\lvert m \rvert > M$, we use the splitting from the proof of \eqref{eq: auxiliary estimates eq3}, to find
\begin{align*}
\sum_{\ell} \lvert \ell \rvert_{l^\alpha}^{-d}\lvert \ell -m \rvert_{l^\beta,M}^{-d-p} &\lesssim \sum_{\ell \in B_{\frac{\lvert m \rvert}{3}}(0)} \lvert \ell \rvert_{l^\alpha}^{-d}\lvert m \rvert_{l^\beta}^{-d-p} + \sum_{\ell \in B_{\frac{\lvert m \rvert}{3}}(m)} \lvert m \rvert_{l^\alpha}^{-d}\lvert \ell -m \rvert_{l^\beta,M}^{-d-p}\\
&\qquad + \sum_{\ell \in B_{2 \lvert m \rvert}(0)^c} \lvert \ell \rvert_{l^\alpha}^{-d}\lvert \ell \rvert_{l^\beta}^{-d-p} + \lvert m \rvert_{l^{\alpha+\beta}}^{-2d-p} \lvert m \rvert_{l^0}^{d}\\
&\lesssim \lvert m \rvert_{l^{\beta+ \alpha +1}}^{-d-p} + \lvert m \rvert_{l^\alpha}^{-d} \lvert M \rvert_{l^\beta}^{-p}\\
&\lesssim  \lvert m \rvert_{l^\alpha,M}^{-d} \lvert M \rvert_{l^{\beta+1}}^{-p}.
\end{align*}
The same splitting of the sum for \eqref{eq: auxiliary estimates eq5} gives
\begin{align*}
\sum_{\ell} \lvert \ell \rvert_{l^\alpha}^{-d-p}\lvert \ell -m \rvert_{l^\beta}^{-d-p} &\lesssim \sum_{\ell \in B_{\frac{\lvert m \rvert}{3}}(0)} \lvert \ell \rvert_{l^\alpha}^{-d-p}\lvert m \rvert_{l^\beta}^{-d-p} + \sum_{\ell \in B_{\frac{\lvert m \rvert}{3}}(m)} \lvert m \rvert_{l^\alpha}^{-d-p}\lvert \ell -m \rvert_{l^\beta}^{-d-p}\\
&\qquad + \sum_{\ell \in B_{2 \lvert m \rvert}(0)^c} \lvert \ell \rvert_{l^\alpha}^{-d-p}\lvert \ell \rvert_{l^\beta}^{-d-p} + \lvert m\rvert_{l^{\alpha+\beta}}^{-2d-2p}\lvert m \rvert_{l^0}^{d}\\
&\lesssim \lvert m \rvert_{l^\beta}^{-d-p} + \lvert m \rvert_{l^\alpha}^{-d-p} + \lvert m\rvert_{l^{\alpha+\beta}}^{-d-2p}\\
&\lesssim \lvert m \rvert_{l^\alpha}^{-d-p}.
\end{align*}
We get to \eqref{eq: auxiliary estimates eq6}. First, let $\lvert m \rvert, \lvert n \rvert \leq 2M$. Then
\begin{align*}
\sum_{\ell \in \L} \lvert \ell \rvert_{l^\alpha,M}^{-d} \lvert \ell -m \rvert_{l^\beta}^{-d} \lvert \ell-n \rvert_{l^\gamma}^{-d} &\lesssim \sum_{\ell \in B_{3M}(0)} \lvert M \rvert_{l^\alpha}^{-d} \lvert \ell -m \rvert_{l^\beta}^{-d} \lvert \ell-n \rvert_{l^\gamma}^{-d} + \sum_{\ell \in B_{3M}(0)^c} \lvert \ell \rvert_{l^{\alpha+\beta + \gamma}}^{-3d}\\
&\lesssim \lvert M \rvert_{l^\alpha}^{-d} \lvert m-n \rvert_{l^{\beta + \gamma +1}}^{-d}+ \lvert M \rvert_{l^{\alpha+\beta + \gamma}}^{-2d},
\end{align*}
according to \eqref{eq: auxiliary estimates eq1} and \eqref{eq: auxiliary estimates eq3}. Next, let $\lvert n \rvert \geq 2M$, $\lvert n \rvert \geq \lvert m \rvert$, and $\lvert m-n \rvert \geq \lvert n \rvert/4$. Then
\begin{align*}
\sum_{\ell \in \L} \lvert \ell \rvert_{l^\alpha,M}^{-d} \lvert \ell -m \rvert_{l^\beta}^{-d} \lvert \ell-n \rvert_{l^\gamma}^{-d} &\lesssim \sum_{\ell \in B_{\frac{\lvert n \rvert}{8}}(n)} \lvert n \rvert_{l^\alpha}^{-d} \lvert n \rvert_{l^\beta}^{-d} \lvert \ell-n \rvert_{l^\gamma}^{-d} + \sum_{\ell \in B_{\frac{\lvert n \rvert}{8}}(n)^c} \lvert \ell \rvert_{l^\alpha,M}^{-d} \lvert \ell -m \rvert_{l^\beta}^{-d} \lvert n \rvert_{l^\gamma}^{-d} \\
&\lesssim \lvert n \rvert_{l^{\alpha +\beta+\gamma +1}}^{-2d} + \lvert n \rvert_{l^\gamma}^{-d} \lvert m \rvert_{l^{\alpha+\beta +1},M}^{-d}.
\end{align*}
At last, let $\lvert n \rvert \geq 2M$ with $\lvert m-n \rvert < \lvert n \rvert/4$. Then,
\begin{align*}
\sum_{\ell \in \L} \lvert \ell \rvert_{l^\alpha,M}^{-d} \lvert \ell -m \rvert_{l^\beta}^{-d} \lvert \ell-n \rvert_{l^\gamma}^{-d} &\lesssim \sum_{\ell \in B_{\frac{\lvert n \rvert}{2}}(n)} \lvert n \rvert_{l^\alpha}^{-d} \lvert \ell -m \rvert_{l^\beta}^{-d} \lvert \ell-n \rvert_{l^\gamma}^{-d} + \sum_{\ell \in B_{2 \lvert n \rvert}(0) \backslash B_{\frac{\lvert n \rvert}{2}}(n)} \lvert \ell \rvert_{l^\alpha,M}^{-d} \lvert n \rvert_{l^\beta}^{-d} \lvert n \rvert_{l^\gamma}^{-d} \\
&\qquad + \sum_{\ell \in B_{2 \lvert n \rvert}(0)^c} \lvert \ell \rvert_{l^\alpha}^{-d} \lvert \ell \rvert_{l^\beta}^{-d} \lvert \ell \rvert_{l^\gamma}^{-d}\\
&\lesssim \lvert n \rvert_{l^\alpha}^{-d} \lvert m-n \rvert_{l^{\beta + \gamma+1}}^{-d} + \lvert n \rvert_{l^{\alpha+\beta + \gamma+1}}^{-2d}.
\end{align*}
Overall, we have shown that if $\lvert n \rvert \geq \lvert m \rvert$, then
\begin{align*}
\sum_{\ell \in \L} \lvert \ell \rvert_{l^\alpha,M}^{-d} \lvert \ell -m \rvert_{l^\beta}^{-d} \lvert \ell-n \rvert_{l^\gamma}^{-d} &\lesssim  \lvert n \rvert_{l^\alpha,M}^{-d} \lvert m-n \rvert_{l^{\beta + \gamma+1}}^{-d} +  \lvert n \rvert_{l^\gamma,M}^{-d} \lvert m \rvert_{l^{\alpha+\beta +1},M}^{-d}.
\end{align*}
That also means, that if $\lvert m \rvert \geq \lvert n \rvert$, then
\begin{align*}
\sum_{\ell \in \L} \lvert \ell \rvert_{l^\alpha,M}^{-d} \lvert \ell -m \rvert_{l^\beta}^{-d} \lvert \ell-n \rvert_{l^\gamma}^{-d} &\lesssim  \lvert m \rvert_{l^\alpha,M}^{-d} \lvert m-n \rvert_{l^{\beta + \gamma+1}}^{-d} +  \lvert m \rvert_{l^\beta,M}^{-d} \lvert n \rvert_{l^{\alpha+\gamma +1},M}^{-d}.
\end{align*}
We are only left with \eqref{eq: auxiliary estimates eq7}. As in the proof of \eqref{eq: auxiliary estimates eq6}, we find for $\lvert m \rvert, \lvert n \rvert \leq 2M$ that
\begin{align*}
\sum_{\ell \in \L} \lvert \ell \rvert_{l^\alpha,M}^{-d-p} \lvert \ell -m \rvert_{l^\beta}^{-d} \lvert \ell-n \rvert_{l^\gamma}^{-d} &\lesssim \lvert M \rvert_{l^\alpha}^{-d-p} \lvert m-n \rvert_{l^{\beta + \gamma +1}}^{-d}+ \lvert M \rvert_{l^{\alpha+\beta + \gamma}}^{-2d-p}.
\end{align*}
Also, for $\lvert n \rvert \geq 2M$, $\lvert n \rvert \geq \lvert m \rvert$, and $\lvert m-n \rvert \geq \lvert n \rvert/4$ we have
\begin{align*}
\sum_{\ell \in \L} \lvert \ell \rvert_{l^\alpha,M}^{-d-p} \lvert \ell -m \rvert_{l^\beta}^{-d} \lvert \ell-n \rvert_{l^\gamma}^{-d} &\lesssim \sum_{\ell \in B_{\frac{\lvert n \rvert}{8}}(n)} \lvert n \rvert_{l^\alpha}^{-d-p} \lvert n \rvert_{l^\beta}^{-d} \lvert \ell-n \rvert_{l^\gamma}^{-d} + \sum_{\ell \in B_{\frac{\lvert n \rvert}{8}}(n)^c} \lvert \ell \rvert_{l^\alpha,M}^{-d-p} \lvert \ell -m \rvert_{l^\beta}^{-d} \lvert n \rvert_{l^\gamma}^{-d} \\
&\lesssim \lvert n \rvert_{l^{\alpha +\beta+\gamma +1}}^{-2d-p} + \lvert n \rvert_{l^\gamma}^{-d} \lvert m \rvert_{l^{\beta},M}^{-d} \lvert M \rvert_{l^{\alpha+1}}^{-p},
\end{align*}
according to \eqref{eq: auxiliary estimates eq4}. At last, let $\lvert n \rvert \geq 2M$ with $\lvert m-n \rvert < \lvert n \rvert/4$. Then,
\begin{align*}
\sum_{\ell \in \L} \lvert \ell \rvert_{l^\alpha,M}^{-d-p} \lvert \ell -m \rvert_{l^\beta}^{-d} \lvert \ell-n \rvert_{l^\gamma}^{-d} &\lesssim \sum_{\ell \in B_{\frac{\lvert n \rvert}{2}}(n)} \lvert n \rvert_{l^\alpha}^{-d-p} \lvert \ell -m \rvert_{l^\beta}^{-d} \lvert \ell-n \rvert_{l^\gamma}^{-d} + \sum_{\ell \in B_{2 \lvert n \rvert}(0) \backslash B_{\frac{\lvert n \rvert}{2}}(n)} \lvert \ell \rvert_{l^\alpha,M}^{-d-p} \lvert n \rvert_{l^\beta}^{-d} \lvert n \rvert_{l^\gamma}^{-d} \\
&\qquad + \sum_{\ell \in B_{2 \lvert n \rvert}(0)^c} \lvert \ell \rvert_{l^\alpha}^{-d-p} \lvert \ell \rvert_{l^\beta}^{-d} \lvert \ell \rvert_{l^\gamma}^{-d}\\
&\lesssim \lvert n \rvert_{l^\alpha}^{-d-p} \lvert m-n \rvert_{l^{\beta + \gamma+1}}^{-d} + \lvert n \rvert_{l^{\beta + \gamma}}^{-2d} + \lvert n \rvert_{l^{\alpha+\beta + \gamma}}^{-2d-p}\\
&\lesssim \lvert n \rvert_{l^\alpha}^{-d-p} \lvert m-n \rvert_{l^{\beta + \gamma+1}}^{-d} + \lvert n \rvert_{l^{\beta + \gamma}}^{-2d}.
\end{align*}
Overall, we have shown that for $\lvert n \rvert \geq \lvert m \rvert$
\begin{align*}
\sum_{\ell \in \L} \lvert \ell \rvert_{l^\alpha,M}^{-d-p} \lvert \ell -m \rvert_{l^\beta}^{-d} \lvert \ell-n \rvert_{l^\gamma}^{-d} &\lesssim \lvert n \rvert_{l^\alpha,M}^{-d-p} \lvert m-n \rvert_{l^{\beta + \gamma+1}}^{-d} + \lvert n \rvert_{l^\gamma,M}^{-d} \lvert m \rvert_{l^{\beta},M}^{-d} \lvert M \rvert_{l^{\alpha+1}}^{-p}.
\end{align*}
\end{proof}

%% file: dffe.bbl
\newcommand{\etalchar}[1]{$^{#1}$}
\begin{thebibliography}{HKM{\etalchar{+}}14}

\bibitem[BBLP10]{BlancBrisLegollPatz2010}
X.~Blanc, C.~L. Bris, F.~Legoll, and C.~Patz.
\newblock Finite-temperature coarse-graining of one-dimensional models:
  Mathematical analysis and computational approaches.
\newblock {\em Journal of Nonlinear Science}, 20(2):241--275, 2010.

\bibitem[BBM10]{BBM2010}
Florent Barret, Anton Bovier, and Sylvie Méléard.
\newblock Uniform estimates for metastable transition times in a coupled
  bistable system.
\newblock {\em Electron. J. Probab.}, 15:323--345, 2010.

\bibitem[BFG07]{Berglund2007-hk}
Nils Berglund, Bastien Fernandez, and Barbara Gentz.
\newblock Metastability in interacting nonlinear stochastic differential
  equations: {II}. {Large-N} behaviour.
\newblock {\em Nonlinearity}, 20(11):2583, October 2007.

\bibitem[BHO]{BHOdefectdevelopment}
J.~Braun, T.~Hudson, and C.~Ortner.
\newblock in preparation.

\bibitem[BL13]{BlancLegoll2013}
X.~Blanc and F.~Legoll.
\newblock A numerical strategy for coarse-graining two-dimensional atomistic
  models at finite temperature: The membrane case.
\newblock {\em Computational Materials Science}, 66:84 -- 95, 2013.

\bibitem[BO18]{2018-uniform}
J.~Braun and C.~Ortner.
\newblock Sharp uniform convergence rate of the supercell approximation of a
  crystalline defect.
\newblock {\em ArXiv e-prints}, 1811.08741, 2018.

\bibitem[BSS14]{Boateng2014-ol}
H~Boateng, T~Schulze, and P~Smereka.
\newblock Approximating {Off-Lattice} kinetic {M}onte {C}arlo.
\newblock {\em Multiscale Model. Simul.}, 12(1):181--199, January 2014.

\bibitem[CD08]{CsD:VMV:2008}
Bal\'azs Cs\'ebfalvi and Bal\'azs Domonkos.
\newblock Pass-band optimal reconstruction on the body-centered cubic lattice.
\newblock In {\em Proceedings of the 13\textsuperscript{th} Vision, Modeling,
  and Visualization Workshop (VMV)}, pages 71--80, Konstanz, Germany, November
  2008.

\bibitem[CO16]{2015-qmtb1}
H.~Chen and C.~Ortner.
\newblock {QM/MM} methods for crystalline defects. {P}art 1: Locality of the
  tight binding model.
\newblock {\em Multiscale Model. Simul.}, 14(1), 2016.

\bibitem[DDO18]{2016-defectFE1d}
{Dobson, Matthew}, {Duong, Manh Hong}, and {Ortner, Christoph}.
\newblock On assessing the accuracy of defect free energy computations.
\newblock {\em ESAIM: M2AN}, 52(4):1315--1352, 2018.

\bibitem[DF05]{DemboFunaki2005}
A.~Dembo and T.~Funaki.
\newblock {\em Stochastic Interface Models. In: {P}icard {J}. (eds) Lectures on
  Probability Theory and Statistics.}, volume 1869 of {\em Lecture Notes in
  Mathematics}.
\newblock Springer, Berlin, Heidelberg, 2005.

\bibitem[DS58]{Dunford1958-qc}
N.~Dunford and J.~T. Schwartz.
\newblock {\em Linear operators. Part {I}: General Theory}.
\newblock Interscience, New York, 1958.

\bibitem[EOS16]{EOS2016}
V.~Ehrlacher, C.~Ortner, and A.~V. Shapeev.
\newblock Analysis of boundary conditions for crystal defect atomistic
  simulations.
\newblock {\em Archive for Rational Mechanics and Analysis}, 222(3):1217--1268,
  2016.

\bibitem[Eyr35]{Eyring}
Henry Eyring.
\newblock The activated complex in chemical reactions.
\newblock {\em The Journal of Chemical Physics}, 3(2):107--115, 1935.

\bibitem[HKM{\etalchar{+}}14]{Herbert2014}
F.W. Herbert, A.~Krishnamoorthy, W.~Ma, K.J.~Van Vliet, and B.~Yildiz.
\newblock Dynamics of point defect formation, clustering and pit initiation on
  the pyrite surface.
\newblock {\em Electrochimica Acta}, 127:416 -- 426, 2014.

\bibitem[HO12]{2012-M2AN-CBstab}
T.~Hudson and C.~Ortner.
\newblock On the stability of {B}ravais lattices and their {C}auchy--{B}orn
  approximations.
\newblock {\em M2AN Math. Model. Numer. Anal.}, 46:81--110, 2012.

\bibitem[HO14]{Hudson2014}
T.~Hudson and C.~Ortner.
\newblock Existence and stability of a screw dislocation under anti-plane
  deformation.
\newblock {\em Arch. Ration. Mech. Anal.}, 213(3):887--929, 2014.

\bibitem[HTB90]{HTB90}
Peter H\"anggi, Peter Talkner, and Michal Borkovec.
\newblock Reaction-rate theory: fifty years after {K}ramers.
\newblock {\em Rev. Mod. Phys.}, 62:251--341, Apr 1990.

\bibitem[Hud17]{Hudson2017-bx}
T~Hudson.
\newblock Upscaling a model for the {Thermally-Driven} motion of screw
  dislocations.
\newblock {\em Arch. Ration. Mech. Anal.}, 224(1):291--352, April 2017.

\bibitem[KLP{\etalchar{+}}14]{hyperqc}
Woo~Kyun Kim, Mitchell Luskin, Danny Perez, Ellad Tadmor, and Art Voter.
\newblock {Hyper-QC}: An accelerated finite-temperature quasicontinuum method
  using hyperdynamics.
\newblock {\em Journal of the Mechanics and Physics of Solids}, 63:94--112,
  2014.

\bibitem[LM13]{Lu2013-xj}
Jianfeng Lu and Pingbing Ming.
\newblock Convergence of a {Force-Based} hybrid method in three dimensions.
\newblock {\em Commun. Pure Appl. Math.}, 66(1):83--108, 2013.

\bibitem[LO13]{LuskinOrtner2013-Acta}
M.~Luskin and C.~Ortner.
\newblock Atomistic-to-continuum-coupling.
\newblock {\em Acta Numerica}, 2013.

\bibitem[LOSK16]{2014-bqce}
X.~H. Li, C.~Ortner, A.~Shapeev, and B.~Van Koten.
\newblock Analysis of blended atomistic/continuum hybrid methods.
\newblock {\em Numer. Math.}, 134, 2016.

\bibitem[Luo09]{Luong2009}
B.~Luong.
\newblock {\em {F}ourier Analysis on Finite {A}belian Groups}.
\newblock Applied and Numerical Harmonic Analysis. Birkhaeuser, Boston, 2009.

\bibitem[MJ66]{Morrey66}
Charles~B. Morrey~Jr.
\newblock {\em Multiple Integrals in the Calculus of Variations}.
\newblock Classics in Mathematics. Springer-Verlag Berlin Heidelberg, 1966.

\bibitem[OO17]{OO2016}
Derek Olson and Christoph Ortner.
\newblock Regularity and locality of point defects in multilattices.
\newblock {\em Applied Mathematics Research eXpress}, 2017(2):297--337, 2017.

\bibitem[Put92]{Putnis92}
A.~Putnis.
\newblock {\em An Introduction to Mineral Sciences}.
\newblock Cambridge University Press, 1992.
\newblock Cambridge Books Online.

\bibitem[Seg92]{Seg92}
J.~Segercrantz.
\newblock Improving the {C}ayley-{H}amilton equation for low-rank
  transformations.
\newblock {\em The American Mathematical Monthly}, 99(1):42--44, 1992.

\bibitem[SK09]{Seebauer09}
E.~G. Seebauer and M.~C. Kratzer.
\newblock Fundamentals of defect ionization and transport.
\newblock In {\em Charged Semiconductor Defects}, Engineering Materials and
  Processes, pages 5--37. Springer London, 2009.

\bibitem[SL17]{LuskinShapeev2014TMP}
A.~V. Shapeev and M.~Luskin.
\newblock Approximation of crystalline defects at finite temperature, 2017.

\bibitem[TLK{\etalchar{+}}13]{TadmorLegoll2013}
E.~B. Tadmor, F.~Legoll, W.~K. Kim, L.~M. Dupuy, and R.~E. Miller.
\newblock Finite-temperature quasicontinuum.
\newblock {\em Appl. Mech. Rev.}, 65:010803, 2013.

\bibitem[Tre00]{book-Tref}
L.~N. Trefethen.
\newblock {\em Spectral methods in {MATLAB}}, volume~10 of {\em Software,
  Environments, and Tools}.
\newblock Society for Industrial and Applied Mathematics (SIAM), Philadelphia,
  PA, 2000.

\bibitem[Vin57]{GHVineyard:1957}
G.~H. Vineyard.
\newblock Frequency and isotope effects in solid rate processes.
\newblock {\em J. Phys. Chem. Solids}, 3:121--127, 1957.

\bibitem[Vot07]{Voter2007-nw}
Arthur~F. Voter.
\newblock Introduction to the kinetic {Monte} {Carlo} method.
\newblock In Kurt~E Sickafus, Eugene~A Kotomin, and Blas~P Uberuaga, editors,
  {\em Radiation Effects in Solids}, volume 235 of {\em NATO Science Series},
  pages 1--23. Springer Netherlands, Dordrecht, 2007.

\bibitem[Wig38]{Wigner}
E.~Wigner.
\newblock The transition state method.
\newblock {\em Trans Faraday Soc}, 34:29--41, 1938.

\bibitem[WSC11]{Walsh11}
A.~Walsh, A.~A. Sokol, and C.~R.~A. Catlow.
\newblock Free energy of defect formation: Thermodynamics of anion {F}renkel
  pairs in indium oxide.
\newblock {\em Phys. Rev. B}, 83:224105, Jun 2011.

\bibitem[WZLH13]{FSWang2013}
J.~Wang, Y.L. Zhou, M.~Li, and Q.~Hou.
\newblock A modified w-w interatomic potential based on ab initio calculations.
\newblock {\em Modelling and Simulation in Materials Science and Engineering},
  22(1), 2013.

\end{thebibliography}
